\documentclass[12pt,a4paper, oneside,reqno]{article}
  \usepackage{tikz} \usepackage{amsmath} \usepackage{graphicx}
\usepackage[utf8]{inputenc}
\usepackage[english]{babel}

\usepackage[left=1in,right=1in,top=1in,bottom=1in]{geometry}

\usepackage{amsfonts,amssymb,amsmath,amsthm}

\usepackage[symbol]{footmisc} 
\usepackage{longtable,pifont}
\usepackage{pifont}
\usepackage{enumerate}
\usepackage{comment}
\usepackage{cite}
\usepackage{hyperref}
\usepackage{tkz-graph}

\usepackage{fancyhdr}
\theoremstyle{plain}
\newtheorem{theorem}{Theorem}
\newtheorem{lemma}[theorem]{Lemma}
\newtheorem{proposition}[theorem]{Proposition}

\newtheorem{definition}[theorem]{Definition}
\newtheorem{corollary}[theorem]{Corollary}

\usetikzlibrary{arrows}

\newtheorem*{theoremA1}{Theorem A1}
\newtheorem*{theoremA2}{Theorem A2}

\usepackage{fouriernc} 
\usepackage{pifont}

\usepackage{fontawesome5}
\newcommand{\cmark}{\text{\ding{51}}}
\newcommand{\xmark}{\text{\ding{55}}}

\begin{document}


\medskip

\noindent{\Large
The algebraic and geometric classification of \\
noncommutative Jordan superalgebras}
 \footnote{
 The    work is supported by
FCT   2023.08031.CEECIND, UIDB/00212/2023 and UIDP/00212/2023.}

 \medskip

\begin{center}

 {\bf
Hani Abdelwahab\footnote{Department of Mathematics,
 Mansoura University,  Mansoura, Egypt; \ haniamar1985@gmail.com},
   Ivan Kaygorodov\footnote{CMA-UBI, University of  Beira Interior, Covilh\~{a}, Portugal; \    kaygorodov.ivan@gmail.com}   \&
   Abror Khudoyberdiyev \footnote{Institute of Mathematics Academy of
Sciences of Uzbekistan, Tashkent, Uzbekistan; National University of Uzbekistan, Tashkent, Uzbekistan; \ khabror@mail.ru}
}

\end{center}

\noindent {\bf Abstract:}
{\it
The algebraic and geometric classifications of complex
$3$-dimensional noncommutative Jordan superalgebras are given.
In particular, we obtain the algebraic and geometric classification of $3$-dimensional Kokoris and standard superalgebras, and, due to one-to-one correspondences between suitable superalgebras, we have classifications for generic Poisson-Jordan and generic Poisson superalgebras.
As a byproduct, we have the algebraic and geometric classification of the variety of $3$-dimensional anticommutative superalgebras and its principal subvarieties:
Lie, Malcev, binary Lie, Tortkara, anticommutative $\mathfrak{CD}$-, $\mathfrak{s}_4$-, anticommutative terminal  superalgebras,
anticommutative conservative and anticommutative quasi-conservative $\big($rigid$\big)$ superalgebras.
}

 \medskip

\noindent {\bf Keywords}:
{\it
superalgebra,
noncommutative Jordan algebra,
anticommutative algebra,
algebraic classification,
geometric classification.}

\medskip

\noindent {\bf MSC2020}:
17A30 (primary);
17A70,
17A15,
14L30 (secondary).


\tableofcontents

\section*{Introduction}

The algebraic classification (up to isomorphism) of algebras of dimension $n$ of a certain variety
defined by a family of polynomial identities is a classic problem in the theory of non-associative algebras, see \cite{kkl20,  akl}.
There are many results related to the algebraic classification of small-dimensional algebras in different varieties of
associative and non-associative algebras.
 Deformations and geometric properties of a variety of algebras defined by a family of polynomial identities have been an object of study since the 1970's, see \cite{ben,BC99,kkl20,GRH, GRH3, FKS, ikv, kz, als, ahk }  and references in \cite{k23,l24,MS}.
 Burde and Steinhoff constructed the graphs of degenerations for the varieties of    $3$-dimensional and $4$-dimensional Lie algebras~\cite{BC99}.
 Grunewald and O'Halloran studied the degenerations for the variety of $5$-dimensional nilpotent Lie algebras~\cite{GRH}.

Noncommutative Jordan algebras were introduced by Albert in \cite{Alb}. He noted that the structure theories of alternative and Jordan algebras share so many nice properties that it is natural to conjecture that these algebras are members of a more general class with a similar theory. So he introduced the variety of noncommutative Jordan algebras defined by the Jordan identity and the flexibility identity.
Namely, the variety of noncommutative Jordan algebras is defined by the following identities:
\begin{longtable}{rclrcl}
$x(yx) $&$=$&$  (xy)x,$ &
$x^2(yx) $&$=$&$  (x^2y)x.$\end{longtable}
The class of noncommutative Jordan algebras turned out to be vast: for example, apart from alternative and Jordan algebras, it contains quasi-associative and quasi-alternative algebras,
quadratic flexible algebras and all anticommutative algebras. However, the structure theory of this class is far from being nice.
Nevertheless, certain progress was made in the study of the structure theory of noncommutative Jordan algebras and superalgebras
 (see, for example \cite{ccr22,Sch,McC,KLP17,popov,spa,k58,k60,O64,p2v,ps19,ps10,ps13, sh71}).
So,
Schafer gave the first structure theory for noncommutative Jordan algebras of characteristic zero in \cite{Sch}. Kleinfeld and  Kokoris proved that
a simple, flexible, power-associative algebra of finite-dimension over a field of characteristic zero is a noncommutative Jordan algebra  \cite{Kleinfeld}.
A coordinatization theorem for noncommutative Jordan algebras was obtained in a paper by McCrimmon \cite{McC}.
Noncommutative Jordan algebras with additional identity
$\big([x,y],y,y\big)=0$ were studied in papers by Shestakov and Schafer \cite{sh71,S94}.
Strongly prime noncommutative Jordan algebras were studied by
Skosyrskiĭ \cite{sk91}.
A connection between  noncommutative Jordan algebras and  $(-1,-1)$-balanced Freudenthal Kantor triple systems was established by
Elduque,   Kamiya, and Okubo \cite{EKO}. Cabrera Serrano,  Cabrera García, and   Rodríguez Palacios studied
the algebra of multiplications of prime and semiprime noncommutative Jordan algebras \cite{spa, ccr22}. Jumaniyozov, Kaygorodov, and  Khudoyberdiyev classified complex $4$-dimensional nilpotent noncommutative Jordan algebras \cite{ikv18};  Abdelwahab, Abdurasulov, and Kaygorodov classified complex $3$-dimensional  noncommutative Jordan algebras \cite{akl}.
The study of super-generalizations of principal varieties of non-associative algebras is under a certain interest now  \cite{als,BM25,AE96,ahk,kkl20,k23,KM,4,3,2,1,5}.
There are also some results in noncommutative Jordan superalgebras. Namely, Pozhidaev and Shestakov classified all simple finite-dimensional
noncommutative Jordan superalgebras in \cite{ps10,ps13,ps19};
Kaygorodov, Popov, and Lopatin studied the structure of simple noncommutative Jordan superalgebras \cite{KLP17};
Popov described representations of simple noncommutative Jordan superalgebras in \cite{popov, p2v}.

The present paper completely classifies $3$-dimensional superalgebras in various well-known varieties.
Firstly, we develop a method for classifying noncommutative Jordan superalgebras \big(Section \ref{prem}\big).
Then, we obtain  anticommutative $3$-dimensional  superalgebras in Theorem \ref{ant12} (type $(1,2)$) and Theorem \ref{ant21} (type $(2,1)$).
The present results give classifications of the most important subclasses of $3$-dimensional anticommutative superalgebras: so, the classifications of
Lie, Malcev, binary Lie, Tortkara, $\mathfrak{CD}$-, $\mathfrak{s}_4$-, terminal,
conservative and quasi-conservative $\big($rigid$\big)$ superalgebras were obtained \big(Section \ref{anticom}\big).
As a corollary, we  prove a Grishkov--Shestakov's conjecture for $3$-dimensional binary Lie superalgebras \big(Corollary \ref{conj}\big): there are no nontrivial complex $3$-dimensional simple binary Lie superalgebras.
The main theorems of the paper \big(Theorems ${\rm A1}$ and ${\rm A2}$\big) are given, respectively, in
Sections \ref{A1} and \ref{A2}.
They give the full classification of $3$-dimensional noncommutative Jordan superalgebras.
As a byproduct, a classification of noncommutative Jordan superalgebras gives a classification of superalgebras from another two important subvarieties:
Kokoris superalgebras \big(Theorems \ref{K1} and \ref{K2}\big) and
standard superalgebras \big(Theorems \ref{S1} and \ref{S2}\big).

 \medskip

The second part of the paper is dedicated to a description of irreducible components in the varieties of cited superalgebras. Namely, our results can be summarized in the following table\footnote{The information about the case of type $(3,0)$ can be found in \cite{akl,ikv};
  about associative commutative and Jordan superalgebras, see \cite{ahk}.}
  \footnote{$"i$ -- $j$ -- $k"$ means
  that the suitable variety of superalgebras is $i$-dimension,
it has   $j$ irreducible components and $k$ rigid superalgebras.}:

\begin{longtable}{l |rcl |rcl| rcl|   }
\hline{\rm Superalgebra / type} & \multicolumn{3}{|c|}{ ${(1,2)}$} &  \multicolumn{3}{|c|}{ $ {(2,1)}$ }& \multicolumn{3}{|c|}{$  {(3,0)}$} \\
\hline{\rm Lie }  &4 --&2 &-- 1 &  3 --&2 &-- 1 & 6 --& 2& --  1 \\
\hline{\rm Malcev }  &5 --&2 &-- 1 & 3 --&3 &-- 2 &6 -- &2& -- 1\\
\hline{\rm Binary Lie  }  &5 --&5 &-- 4 & 4 --&3 &-- 2 &6 -- &2& -- 1\\
\hline{\rm Tortkara } &4 --&3 &-- 2 &  4 --&2 &-- 2 &8 -- &1& -- 1 \\
\hline{\rm $\mathfrak{aCD}$- } &4 --&3 &-- 2& 3 --&3 &-- 2 & 6 -- &2& -- 1\\
\hline{\rm $\mathfrak{s}_4$- }  &4 --&3 &-- 2& 4 --&2 &-- 1 & 9 -- &1& -- 0\\
\hline{\rm ${\mathfrak a}$-terminal } &4 --&3&  -- 2 &3 --&2 &-- 1 &6 -- &2& -- 1\\
\hline{\rm Anticommutative  }  &7 --&1 &-- 0&  6 --&1 &-- 0 & 9 -- &1& -- 0\\
\hline{\rm AssCom  }  &3 --&4 &-- 4 &  4 --&2 &-- 2& 9 -- &1& -- 1\\
\hline{\rm Jordan }  & 3 --&7 &-- 7& 4 --&7 &-- 7 &9 -- &5& -- 5\\
\hline{\rm Kokoris }  &7 --&6 &-- 3 & 6 --&5 &-- 3 & 9 -- &5& -- 3\\

\hline{\rm Associative }  &4 -- &12&-- 10 & 5 --&13 &-- 13 & 9 -- &8& -- 7\\
\hline{\rm Standard }  &4 --&16 &-- 14 & 5 --&22 &-- 22 & 9 -- &14& -- 13\\
\hline{\rm NCJ }  &7 --&9 &-- 3 & 6 --&10 &-- 4 & 9 -- &7& -- 4\\
\hline
\end{longtable}

Let us also remember the following inclusions:

  \medskip

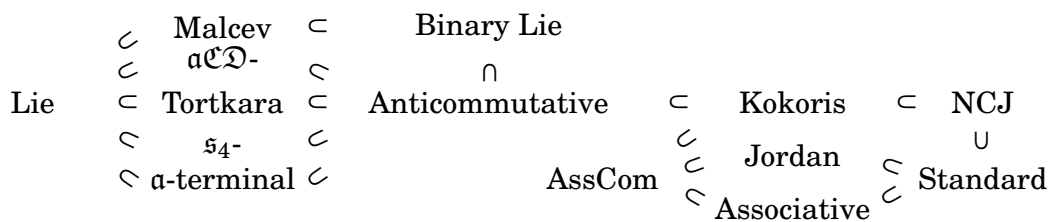
\begin{center}

\begin{tikzpicture}[node distance=1.5cm, auto, text=black] 
\node (lie) at (0,0) {Lie};

\node (malcev) at (2.5, 1.0) {Malcev};
\node (aCD) at (2.5, 0.6) {$\mathfrak{aCD}$-};
\node (tortkara) at (2.5, 0) {Tortkara};
\node (s4) at (2.5, -0.6) {$\mathfrak{s}_4$-};
\node (at) at (2.5, -1.0) {${\mathfrak a}$-terminal};

\node (binarylie) at (6, 1.0) {Binary Lie};
\node (anticommutative) at (6, 0) {Anticommutative};
\node (kokoris) at (10, 0) {Kokoris};
\node (ncj) at (12.5, 0) {NCJ};
\node (asscom) at (7.5, -1) {AssCom};

\node (jordan) at (10, -0.7) {Jordan};
\node (standard) at (10, -1.4) {Associative};
\node (standard) at (12.5, -1) {Standard};

\node at (1.25, 0.8) {\rotatebox{35}{$\subset$}};
\node at (1.25, 0.4) {\rotatebox{25}{$\subset$}};
\node at (1.25, 0) {$\subset$};
\node at (1.25, -1.0) {\rotatebox{-35}{$\subset$}};
\node at (1.25, -0.5) {\rotatebox{-25}{$\subset$}};

\node at (3.75, 1.0) {\rotatebox{0}{$\subset$}};
\node at (3.75, 0.4) {\rotatebox{-25}{$\subset$}};
\node at (3.75, 0) {$\subset$};
\node at (3.75, -0.5) {\rotatebox{25}{$\subset$}};
\node at (3.75, -1.0) {\rotatebox{35}{$\subset$}};

\node at (6, 0.4) {\rotatebox{-90}{$\subset$}};

\node at (8.5, 0) {\rotatebox{ 0}{$\subset$}};
\node at (11.5, 0) {\rotatebox{ 0}{$\subset$}};

\node at (12.5, -0.5) {\rotatebox{90}{$\subset$}};
\node at (8.6, -0.5) {\rotatebox{35}{$\subset$}};

\node at (8.7, -0.85) {\rotatebox{20}{$\subset$}};
\node at (8.7, -1.25) {\rotatebox{-30}{$\subset$}};

\node at (11.3, -0.85) {\rotatebox{-20}{$\subset$}};
\node at (11.3, -1.25) {\rotatebox{30}{$\subset$}};

\end{tikzpicture}

\end{center}

\section{The algebraic classification of superalgebras}

\begin{definition}
A superalgebra $\rm{A}$ over
a field $\mathbb{C}$ with a ${\mathbb Z}_{2}$-grading $\rm{A=A}_{0}\rm{\oplus A}_{1}$ is called noncommutative Jordan  if its multiplication satisfies
the following two superidentities:
\begin{center}
$(xy) z - x (yz) \  =\   - \big(-1\big)^{|x||y|+|y||z|+|z||x|}\big( (zy) x-z (yx) \big),$
\end{center}

\begin{flushleft}
$\big((xt)(yz) -((xt)y) z \big) + \big(-1\big)^{|x||t|} \big((tx)(yz)-((tx)y) z \big) \ +$
\end{flushleft}

\begin{center}
$\big(-1\big)^{|x|(|y|+|z|+|t|)+|z||y|} \big((tz)(yx) -((tz)y) x\big) + \big(-1\big)^{|y||z|+|z||t|+|t||y|} \big((xz)(yt) -((xz)y) t\big) \ +$
\end{center}
\begin{flushright}

$\big(-1\big)^{|x|(|y|+|z|+|t|)+|z|(|t|+|y|)}\big((zt)(yx)-((zt)y) x\big) +
\big(-1\big)^{|z|(|x|+|y|+|t|)+|t||y|}\big((zx)(yt) - ((zx)y) t\big) \ = \ 0.$
\end{flushright}

\noindent for all homogeneous elements $x,y,z$ in $\rm{A}$, where $\left\vert
x\right\vert $ denotes the degree of $x$ $\big("0"$ for even elements, $"1"$ for odd
elements$\big)$.
\end{definition}

\subsection{Preliminaries: the method for algebraic classification}\label{prem}
In this section, we establish a straightforward method to obtain the
algebraic classification of the noncommutative Jordan superalgebra structures of
type $\left( n,m\right) $ defined over an arbitrary Jordan superalgebra of
type $\left( n,m\right) $.

The  Jordan superproduct and supercommutator are defined as follows:

\begin{longtable}{lcl}
$x\circ y $&$=$&$\frac{1}{2} \big(xy+\left( -1\right) ^{\left\vert x\right\vert \left\vert
y\right\vert }yx \big),$ \\
$\left[ x,y\right] $&$=$&$\frac{1}{2} \big(xy-\left( -1\right) ^{\left\vert x\right\vert
\left\vert y\right\vert }yx \big).$
\end{longtable}
The superalgebra $\left( \rm{A},\circ \right) $ is called the
symmetrized superalgebra of $\rm{A}$ and is denoted by $\rm{A}^{+}.$
If $\rm{A}$ is a noncommutative Jordan superalgebra, then $\rm{A}^{+}$
is a Jordan superalgebra.

\begin{definition}
A superalgebra $({\bf P}, \circ, [\cdot,\cdot])$ is called  a  generic Poisson--Jordan  superalgebra (resp., generic Poisson  superalgebra)\footnote{To the best of our knowledge, such "associative-commutative"  algebras were introduced independently by Cannas and Weinstein  under the name "almost Poisson algebras" and by Shestakov  under the name "general Poisson algebras" $\big($later changed into "generic Poisson algebras" in
 a paper by  Kolesnikov,   Makar-Limanov, and Shestakov (2014)$\big)$.
 We will use the last terminology.} if
  $({\bf P}, \circ)$ is a Jordan (resp., associative commutative) superalgebra,
   $({\bf P},  [\cdot,\cdot])$ is an anticommutative superalgebra and
these two operations are required to satisfy the following   superidentity:
\begin{equation*}
[ x\circ y,z] = (-1)^{|y||z|}[ x,z] \circ y+x\circ [ y,z].
\end{equation*}%

\end{definition}

The class of generic Poisson--Jordan superalgebras is extremely extensive. It contains all Poisson superalgebras, Malcev--Poisson superalgebras, and Malcev--Poisson--Jordan superalgebras, as well as all generic Poisson superalgebras.

\begin{proposition}
$({\rm A},\cdot) $ is a noncommutative Jordan superalgebra if and only
if $({\rm A},\circ ,[\cdot,\cdot])$ is a generic Poisson--Jordan  superalgebra.
\end{proposition}

\begin{proof}
Since the flexible super-law is equivalent to $[x\circ y,z] = x\circ [y,z] + (-1)^{|x||y|}
y\circ [x,z]$, the proof is finished.
\end{proof}

\begin{definition}
\label{def_homomorphism}\textrm{Let $({\bf P}_1,\circ_1,[\cdot,\cdot]_1)$
and $({\bf P}_2,\circ_2,[\cdot,\cdot]_2)$ be two
generic Poisson--Jordan superalgebras. A   linear map $\phi :%
{\bf P}_1 \to {\bf P}_2$ is a {homomorphism} if it is preserving the products, that is,
\begin{equation*}
\phi (x\circ_1 y) =\phi(x) \circ_2 \phi(y), \hspace{2cm} \phi([x,y]_1) =
[\phi(x),\phi(y)]_2.
\end{equation*}
 }
\end{definition}

Let $({\rm A}_1,\cdot_1)$ and $({\rm A}_2,\cdot_2)$ be two
noncommutative Jordan superalgebras and let $({\rm A}_1,\circ_1,[\cdot,%
\cdot]_1)$ and $({\rm A}_2,\circ_2,[\cdot,\cdot]_2)$ be its associated
generic Poisson--Jordan superalgebras. If $({\rm A}_1,\cdot_1)$ and $({\rm A}_2,\cdot_2)$
are isomorphic, then  it is easy to see that
the generic Poisson--Jordan superalgebras $({\rm A}_1,\circ_1,[\cdot,\cdot]_1)$ and $({\rm A}_2,\circ_2,[\cdot, \cdot]_2)$ are isomorphic. Conversely, we can show that if the generic
Poisson--Jordan superalgebras $({\rm A}_1,\circ_1,[\cdot,\cdot]_1)$ and $(
{\rm A}_2,\circ_2,[\cdot,\cdot]_2)$ are isomorphic, then the
noncommutative Jordan superalgebras $({\rm A}_1,\cdot_1)$ and $({\rm A}%
_2,\cdot_2)$ are isomorphic. So we have the following result:

\begin{proposition}
Every generic Poisson--Jordan  superalgebra
(resp., generic Poisson superalgebra) $({\rm A},\circ, [\cdot,\cdot])$ is
associated with precisely one noncommutative Jordan (resp., Kokoris\footnote{About Kokoris algebras, see subsection \ref{koko}.}) algebra $({\rm A}%
,\cdot)$. That is, we
have a bijective correspondence between generic Poisson--Jordan (resp., generic Poisson) superalgebras
  and
noncommutative Jordan (resp., Kokoris) superalgebras.
\end{proposition}

Let $\mathbb{V=V}_{0}\mathbb{\oplus V}_{1}$ be a ${\mathbb Z}_{2}$-graded vector
space over a field $\mathbb{C}$. A bilinear map%
\begin{equation*}
\theta :\mathbb{V}\times \mathbb{V\to V}
\end{equation*}%
is called supercommutative if for all homogeneous elements $x,y\in \mathbb{V}$, the following graded commutativity condition holds:%
\begin{equation*}
\theta \left( x,y\right) =\left( -1\right) ^{\left\vert x\right\vert
\left\vert y\right\vert }\theta \left(y,x\right) .
\end{equation*}%
Moreover, $\theta $ is called super-skew-symmetric if for all homogeneous
elements $x,y\in \mathbb{V}$, the following graded commutativity condition
holds:%
\begin{equation*}
\theta \left( x,y\right) =-\left( -1\right) ^{\left\vert x\right\vert
\left\vert y\right\vert }\theta \left( y,x\right) .
\end{equation*}

\begin{definition}
Let $\left( \rm{A},\circ  \right) $ be a Jordan superalgebra of type $\left( n,m\right) $. Define ${\rm Z}^{2}\left( \rm{A},\rm{A}\right) $ to
be the set of all super-skew-symmetric bilinear maps $\theta :\rm{A}%
\times \rm{A}\to \rm{A}$ such that:
\begin{equation*}
\theta(x\circ y,z) = (-1)^{|y||z|}\theta(x,z) \circ y +x\circ \theta(y,z).
\end{equation*} \noindent{}for all $x,y,z  \in {\rm A}_0 \cup {\rm A}_1$.
\end{definition}

Observe that, for $\theta \in {\rm Z}^{2}\left( \rm{A},\rm{A}\right) $,
if we define a multiplication $\ast _{\theta }$ on $\rm{A}$ by
\begin{center}
    $x\ast_{\theta }y\ =\ x\circ  y+\theta \left( x,y\right) $ for all $x,y$ in $\rm{A}$,
\end{center}
then $\left( \rm{A},\ast _{\theta }\right) $ is a noncommutative Jordan   superalgebra of type $\left( n,m\right) $. Conversely, if $\left( \rm{A},\cdot \right) $ is a noncommutative Jordan superalgebra of
type $\left( n,m\right) $, then there exists $\theta \in {\rm Z}^{2}\left(
\rm{A},\rm{A}\right) $ such that $\left( \rm{A},\cdot,\right) \cong \left( \rm{A},\ast _{\theta }\right) $. To see this,
consider the super-skew-symmetric bilinear map $\theta :\rm{A}\times \rm{A}\to \rm{A}$ defined by $\theta \left( x,y\right)
=\left[ x,y\right] $ for all $x,y$ in $\rm{A}$. Then $\theta \in
{\rm Z}^{2}\left( \rm{A},\rm{A}\right) $ and $\left( \rm{A},\cdot
,\right) \cong \left( \rm{A},\ast _{\theta }\right) $.

\medskip

Now, let $\left( \rm{A},\cdot \right) $ be a noncommutative Jordan
superalgebra and $\rm{Aut}\left( \rm{A}\right) $ be the automorphism
group of $\rm{A}$. Then $\rm{Aut}\left( \rm{A}\right) $ acts on $%
{\rm Z}^{2}\left( \rm{A},\rm{A}\right) $ by
\begin{equation*}
\left( \theta \ast \phi \right) \left( x,y\right) =\phi ^{-1}\left( \theta
\left( \phi \left( x\right) ,\phi \left( y\right) \right) \right) ,
\end{equation*}%
for $\phi \in \text{Aut}\left( \rm{A}\right) $, and $\theta \in {\rm Z}^{2}\left( \rm{A},\rm{A}\right) $.

\begin{lemma}
Let $\left( \rm{A},\cdot \right) $ be a Jordan superalgebra and $\theta
,\vartheta \in {\rm Z}^{2}\left( \rm{A},\rm{A}\right) $. Then $\left(
\rm{A},\ast _{\theta }\right) $ and $\left( \rm{A},\ast
_{\vartheta }\right) $ are isomorphic if and only if there is a $\phi \in
\rm{Aut}\left( \rm{A}\right) $ with $\theta \ast \phi =\vartheta $.
\end{lemma}

Hence, we have a procedure to classify the noncommutative Jordan  superalgebras
of type $\left( n,m\right) $ associated with a given Jordan superalgebra $%
\left( \rm{A},\circ  \right) $ of type $\left( n,m\right) $. It
consists of three steps:

\begin{enumerate}
\item[{\bf Step 1.}] Compute ${\rm Z}^{2}\left( \rm{A},\rm{A}\right) $.

\item[{\bf Step 2.}] Find the orbits of $\rm{Aut}\left( \rm{A}\right) $ on $%
{\rm Z}^{2}\left( \rm{A},\rm{A}\right) $.

\item[{\bf Step 3.}] Choose a representative $\theta $ from each orbit and then construct
the noncommutative Jordan superalgebra $\left( \rm{A},\ast _{\theta
}\right) $.
\end{enumerate}

\medskip

Let us introduce the following notations. Let
\begin{center}
    $e_{1},\ldots
,e_{n},e_{n+1}=f_{1},\ldots,e_{n+m}=f_{m},$
\end{center} be a fixed basis of a Jordan
superalgebra $\left( \rm{A},\circ  \right) $ of type $\left( n,m\right) $.
Define $\mathrm{\Lambda }^{2}\left( \rm{A},\mathbb{C}\right) $ to be
the space of all super-skew-symmetric bilinear forms on $\rm{A}$. Then
\begin{center}$%
\mathrm{\Lambda }^{2}\left( \rm{A},\mathbb{C}\right) =\left\langle
\Delta _{i,j}:1\leq i<j\leq n+m\right\rangle \oplus \left\langle \Delta
_{i,i}:n<i\leq n+m\right\rangle,$
\end{center} where $\Delta _{i,j}$ is the super-skew-symmetric bilinear form $\Delta _{i,j}:\rm{A}\times \rm{A}%
\to \mathbb{C}$\ defined by%
\begin{equation*}
\Delta _{i,j}\left( e_{l},e_{m}\right) :=\left\{
\begin{tabular}{ll}
$1,$ & if $\left( i,j\right) =\left( l,m\right) $ and $i\neq j,$ \\
$-\left( -1\right) ^{\left\vert e_{l}\right\vert \left\vert e_{m}\right\vert
},$ & if $\left( i,j\right) =\left( m,l\right) $ and $i\neq j,$ \\
$1,$ & if $i=l=m$ and $i=j>n,$ \\
$0,$ & otherwise.%
\end{tabular}%
\right.
\end{equation*}%
 Now, if $\theta \in {\rm Z}^{2}\left( \rm{A},\rm{A}\right) $, then $%
\theta $ can be uniquely written as $\theta \left( x,y\right) =\underset{i=1}%
{\overset{n}{\sum }}B_{i}\left( x,y\right) e_{i}$ where $B_{1},B_{2},\ldots
,B_{n}$ is a sequence of super-skew-symmetric bilinear forms on $\rm{A}$%
. Also, we may write $\theta =\left( B_{1},B_{2},\ldots ,B_{n}\right).$
Let $\phi ^{-1}\in \text{Aut}\left( \rm{A}\right) $ be given by the
matrix $\left( b_{ij}\right) $. If
\begin{center}
    $\left( \theta \ast \phi \right) \left(
x,y\right) =\underset{i=1}{\overset{n}{\sum }}B_{i}^{\prime }\left(
x,y\right) e_{i}$, then $B_{i}^{\prime }=\underset{j=1}{\overset{n}{\sum }}%
b_{ij}\phi ^{t}B_{j}\phi $.

\end{center}

\subsection{Noncommutative Jordan superalgebras of type $(1,2)$}

\begin{theorem}
\label{(1,2)} Let ${\rm J}$ be a nontrivial $3$-dimensional Jordan superalgebra of
type $\left( 1,2\right) $. Then $\rm{J}$ is isomorphic to one of the following superalgebras:

\begin{longtable}{l c l l ll | r}
    \hline
    ${\rm J}_{01}$ & $:$ &  & $e_{1} \circ f_{1} = f_{2}$ & &$f_{1} \circ f_{2} = e_{1}$ & $\mathfrak{non-ass}$\\
    ${\rm J}_{02}$ & $:$ &  &  & &$f_{1} \circ f_{2} = e_{1}$ & $\mathfrak{ass}$\\
    ${\rm J}_{03}$ & $:$ &  & $e_{1} \circ f_{1} = f_{2}$ && & $\mathfrak{ass}$\\
    ${\rm J}_{04}$ & $:$ & $e_{1} \circ e_{1} = e_{1}$ & $e_{1} \circ f_{1} = \frac{1}{2} f_{1}$ &&& $\mathfrak{non-ass}$\\
    ${\rm J}_{05}$ & $:$ & $e_{1} \circ e_{1} = e_{1}$ & $e_{1} \circ f_{1} = f_{1}$ & $e_{1} \circ f_{2} = \frac{1}{2} f_{2}$ && $\mathfrak{non-ass}$\\
    ${\rm J}_{06}$ & $:$ & $e_{1} \circ e_{1} = e_{1}$ & $e_{1} \circ f_{1} = f_{1}$ && & $\mathfrak{ass}$\\
    ${\rm J}_{07}$ & $:$ & $e_{1} \circ e_{1} = e_{1}$ &  &  &&   $\mathfrak{ass}$\\
    ${\rm J}_{08}$ & $:$ & $e_{1} \circ e_{1} = e_{1}$ & $e_{1} \circ f_{1} = \frac{1}{2} f_{1}$ & $e_{1} \circ f_{2} = \frac{1}{2} f_{2}$ && $\mathfrak{non-ass}$\\
    ${\rm J}_{09}$ & $:$ & $e_{1} \circ e_{1} = e_{1}$ & $e_{1} \circ f_{1} = f_{1}$ & $e_{1} \circ f_{2} = f_{2}$ &&  $\mathfrak{ass}$\\
    ${\rm J}_{10}$ & $:$ & $e_{1} \circ e_{1} = e_{1}$ & $e_{1} \circ f_{1} = \frac{1}{2} f_{1}$ & $e_{1} \circ f_{2} = \frac{1}{2} f_{2}$ &$f_{1} \circ f_{2} = e_{1}$ & $\mathfrak{non-ass}$\\
    ${\rm J}_{11}$ & $:$ & $e_{1} \circ e_{1} = e_{1}$ & $e_{1} \circ f_{1} = f_{1}$ & $e_{1} \circ f_{2} = f_{2}$ &$f_{1} \circ f_{2} = e_{1}$ & $\mathfrak{non-ass}$\\
    \hline
\end{longtable}

\end{theorem}

\subsubsection{Anticommutative superalgebras of type $(1,2)$}
\begin{theorem}\label{ant12}
Let ${\rm A}$ be a nontrivial $3$-dimensional anticommutative
superalgebra of type $\left( 1,2\right) $. Then $\rm{A}$ is isomorphic
to one of the following superalgebras:

\begin{longtable}{l c l l l l l l l}
    \hline
     ${\rm A} $&$:$&$[e_{1},f_{1}] $&  $[e_{1},f_{2}]$&
 $[f_{1},f_{1}] $&$[f_{1},f_{2}] $&$[f_{2},f_{2}] $ \\
 \hline
  ${\rm A}_{01}$&$:$& &&&$[f_{1},f_{2}]=e_{1}$\\

 ${\rm A}_{02}$&$:$&& &&& $[f_{2},f_{2}]=e_{1}$ \\

 ${\rm A}_{03}$&$:$&&$[e_{1},f_{2}]=f_{1}$&$[f_{1},f_{1}]=e_{1}$\\

 ${\rm A}_{04}$&$:$&&$[e_{1},f_{2}]=f_{1}$&$[f_{1},f_{1}]=e_{1}$&&$[f_{2},f_{2}]=e_{1}$ \\

 ${\rm A}_{05}$&$:$& & $[e_{1},f_{2}]=f_{1}$& & $[f_{1},f_{2}]=e_{1}$\\

${\rm A}_{06}$&$:$&&$[e_{1},f_{2}]=f_{1}$ \\

${\rm A}_{07}$&$:$& &$[e_{1},f_{2}]=f_{1}$ &&& $[f_{2},f_{2}]=e_{1}$ \\

 ${\rm A}_{08}^{\alpha}$&$:$&$[e_{1},f_{1}]=f_{1}$& $[e_{1},f_{2}]=\alpha f_{2}$  \\

 ${\rm A}_{09}^{\alpha}$&$:$&  $[e_{1},f_{1}]=f_{1}$ &$[e_{1},f_{2}]=\alpha f_{2}$ && $[f_{1},f_{2}]=e_{1}$&\\

${\rm A}_{10}^{\alpha}$&$:$&$[e_{1},f_{1}]=f_{1}$& $[e_{1},f_{2}]= \alpha f_{2}$ &&& $[f_{2},f_{2}]=e_{1}$  \\

 ${\rm A}_{11}^{\alpha \neq 1,\beta}$&$:$& $[e_{1},f_{1}]=f_{1}$&$[e_{1},f_{2}]=\alpha f_{2}$&
$[f_{1},f_{1}]=e_{1}$&$[f_{1},f_{2}]=\beta e_{1}$&$[f_{2},f_{2}]=e_{1}$\\

 ${\rm A}_{12}^{\alpha \neq 1}$&$:$& $[e_{1},f_{1}]=f_{1}$& $[e_{1},f_{2}]=\alpha f_{2}$&&
$[f_{1},f_{2}]=e_{1}$&  $[f_{2},f_{2}]=e_{1}$ \\

  ${\rm A}_{13}$&$:$& $[e_{1},f_{1}]=f_{1}$ && $[f_{1},f_{1}]=e_{1}$&\\

${\rm A}_{14}$&$:$&$[e_{1},f_{1}]=f_{1}$ &&$[f_{1},f_{1}]=e_{1}$&$[f_{1},f_{2}]=e_{1}$&  \\

 ${\rm A}_{15}^{\alpha }$&$:$&$[e_{1},f_{1}]=f_{1}$&  $[e_{1},f_{2}]=f_{1}+f_{2}$&
 $[f_{1},f_{1}]=e_{1}$&&$[f_{2},f_{2}]=\alpha e_{1}$ \\

 ${\rm A}_{16}$&$:$& $[e_{1},f_{1}]=f_{1}$& $[e_{1},f_{2}]=f_{1}+f_{2}$ &&$[f_{1},f_{2}]=e_{1}$&\\

 ${\rm A}_{17}$&$:$&$[e_{1},f_{1}]=f_{1}$& $[e_{1},f_{2}]=f_{1}+f_{2}$ \\

${\rm A}_{18}$&$:$&$[e_{1},f_{1}]=f_{1}$& $[e_{1},f_{2}]=f_{1}+f_{2}$ &&&$[f_{2},f_{2}]=e_{1}$\\
    \hline
\end{longtable}

\noindent All listed superalgebras are non-isomorphic except:
\begin{center}${\rm A}_{08}^{\alpha
}\cong {\rm A}_{08}^{{\alpha^{-1} }},$ \
${\rm A}_{09}^{\alpha }\cong {\rm A}_{09}^{{\alpha^{-1} }},$ and
${\rm A}_{11}^{\alpha ,\beta }\cong {\rm A}_{11}^{{\alpha^{-1} },\beta }\cong {\rm A}_{11}^{{\alpha^{-1} },-\beta }\cong {\rm A}_{11}^{\alpha ,-\beta }$.
\end{center}
\end{theorem}

\begin{proof}
Let $\theta =\big(
B_{1},B_{2},B_{3}\big) \neq 0$ be an arbitrary element of ${\rm Z}^{2}  (
\mathbb{C}^{1,2},\mathbb{C}^{1,2}  ) $. Then%
\begin{equation*}
\theta \ = \ \big(
\alpha _{1}\Delta _{22}+\alpha _{2}\Delta _{33}+\alpha_{3}\Delta _{23},\
\alpha _{4}\Delta _{12}+\alpha _{5}\Delta _{13},\
\alpha_{6}\Delta _{12}+\alpha _{7}\Delta _{13}\big) ,
\end{equation*}%
for some $\alpha _{1},\ldots ,\alpha _{7}\in \mathbb{C}$. The automorphism
group ${\rm Aut}\left( \mathbb{C}^{1,2}\right) $
consists of the automorphisms $\phi $ given by a matrix of the following
form:%
\begin{equation*}
\phi =%
\begin{pmatrix}
a_{11} & 0 & 0 \\
0 & a_{22} & a_{23} \\
0 & a_{32} & a_{33}%
\end{pmatrix}%
.
\end{equation*}

Let $\phi =\bigl(a_{ij}\bigr)\in $ ${\rm Aut}\left( \mathbb{C}^{1,2}\right)
$. Then
\begin{center}$\theta \ast \phi \ =\ \big( \beta _{1}\Delta _{22}+\beta _{2}\Delta
_{33}+\beta _{3}\Delta _{23},\
\beta _{4}\Delta _{12}+\beta _{5}\Delta_{13},\
\beta _{6}\Delta _{12}+\beta _{7}\Delta _{13}\big),$ where
\end{center}
\begin{longtable}{lcl}
$\beta _{1} $&$=$&$ \quad {a^{-1}_{11}}\big( \alpha _{1}a_{22}^{2}+2\alpha
_{3}a_{22}a_{32}+\alpha _{2}a_{32}^{2}\big) ,$ \\
$\beta _{2} $&$=$&${\quad a^{-1}_{11}}\big( \alpha _{1}a_{23}^{2}+2\alpha
_{3}a_{23}a_{33}+\alpha _{2}a_{33}^{2}\big) ,$ \\
$\beta _{3} $&$=$&$ {\quad a^{-1}_{11}}\big( \alpha _{1}a_{22}a_{23}+\alpha
_{2}a_{32}a_{33}+\alpha _{3}a_{22}a_{33}+\alpha _{3}a_{23}a_{32}\big) ,$ \\
$\beta _{4} $&$=$&${\quad a_{11}}\Xi\big( \alpha
_{4}a_{22}a_{33}-\alpha _{6}a_{22}a_{23}+\alpha _{5}a_{32}a_{33}-\alpha
_{7}a_{23}a_{32}\big) ,$ \\
$\beta _{5} $&$=$&${\quad a_{11}}\Xi\big( \alpha
_{5}a_{33}^{2}-\alpha _{6}a_{23}^{2}+\alpha _{4}a_{23}a_{33}-\alpha
_{7}a_{23}a_{33}\big) ,$ \\
$\beta _{6} $&$=$&$-\ {a_{11}}\Xi\big( \alpha
_{5}a_{32}^{2}-\alpha _{6}a_{22}^{2}+\alpha _{4}a_{22}a_{32}-\alpha
_{7}a_{22}a_{32}\big) ,$ \\
$\beta _{7} $&$=$&$-\ {a_{11}}\Xi\big( \alpha
_{4}a_{23}a_{32}- \alpha _{6}a_{22}a_{23}+\alpha _{5}a_{32}a_{33}-\alpha
_{7}a_{22}a_{33}\big) .$
\end{longtable}%
and   $\Xi= \big(a_{22}a_{33}-a_{23}a_{32}\big)^{-1}.$
Whence
\begin{center}$%
\begin{pmatrix}
\beta _{4} & \beta _{5} \\
\beta _{6} & \beta _{7}%
\end{pmatrix}%
\ =\ a_{11}%
\begin{pmatrix}
a_{22} & a_{23} \\
a_{32} & a_{33}%
\end{pmatrix}%
^{-1}%
\begin{pmatrix}
\alpha _{4} & \alpha _{5} \\
\alpha _{6} & \alpha _{7}%
\end{pmatrix}%
\begin{pmatrix}
a_{22} & a_{23} \\
a_{32} & a_{33}%
\end{pmatrix}%
$.\end{center}
So we may assume
\begin{equation*}
\begin{pmatrix}
\alpha _{4} & \alpha _{5} \\
\alpha _{6} & \alpha _{7}%
\end{pmatrix}%
\ \in\  \left\{
\begin{pmatrix}
0 & 0 \\
0 & 0%
\end{pmatrix}%
, \
\begin{pmatrix}
0 & 1 \\
0 & 0%
\end{pmatrix}%
, \
\begin{pmatrix}
1 & 0 \\
0 & \alpha%
\end{pmatrix}%
, \
\begin{pmatrix}
1 & 1 \\
0 & 1%
\end{pmatrix}%
\right\} .
\end{equation*}

\begin{itemize}
\item $%
\begin{pmatrix}
\alpha _{4} & \alpha _{5} \\
\alpha _{6} & \alpha _{7}%
\end{pmatrix}%
=%
\begin{pmatrix}
0 & 0 \\
0 & 0%
\end{pmatrix}%
$. In this case, we have
\begin{center}$%
\begin{pmatrix}
-\beta _{3} & -\beta _{2} \\
\beta _{1} & \beta _{3}%
\end{pmatrix}%
\ =\ \frac{1}{a_{11}}%
\begin{pmatrix}
a_{33} & -a_{23} \\
-a_{32} & a_{22}%
\end{pmatrix}%
\begin{pmatrix}
-\alpha _{3} & -\alpha _{2} \\
\alpha _{1} & \alpha _{3}%
\end{pmatrix}%
\begin{pmatrix}
a_{22} & a_{23} \\
a_{32} & a_{33}%
\end{pmatrix}%
$.\end{center}
Since $\theta \neq 0$, we may assume%
\begin{equation*}
\begin{pmatrix}
-\alpha _{3} & -\alpha _{2} \\
\alpha _{1} & \alpha _{3}%
\end{pmatrix}%
\ \in\  \left\{
\begin{pmatrix}
-1 & 0 \\
0 & 1%
\end{pmatrix}%
,%
\begin{pmatrix}
0 &  -1 \\
0 & 0%
\end{pmatrix}%
\right\} .
\end{equation*}%
So we get the superalgebras ${\rm A}_{01}$ and ${\rm A}_{02}$.

\item $%
\begin{pmatrix}
\alpha _{4} & \alpha _{5} \\
\alpha _{6} & \alpha _{7}%
\end{pmatrix}%
=%
\begin{pmatrix}
0 & 1 \\
0 & 0%
\end{pmatrix}%
$.

\begin{itemize}
\item $\alpha _{1}\neq 0$. Let $\lambda =\alpha _{1}\alpha _{2}- \alpha _{3}^{2},$
then $\phi=\phi_1$  if  $\lambda  =0$ and $\phi=\phi_2$  if  $\lambda  \neq 0$:%
\begin{equation*}
\phi_1 \ = \ \begin{pmatrix}
\frac{1}{\alpha _{1}} & 0 & 0 \\
0 & \frac{1}{\alpha _{1}} & -\frac{\alpha _{3}}{\alpha _{1}} \\
0 & 0 & 1%
\end{pmatrix}%
,\
\phi_2 \ = \  \begin{pmatrix}
\frac{\sqrt{\lambda }}{\alpha_{1}} & 0 & 0 \\
0 & \frac{\sqrt[4]{\lambda}}{\alpha _{1}} & -\frac{\alpha _{3}}{\alpha_1\sqrt[4]{\lambda}} \\
0 & 0 & \frac{1}{\sqrt[4]{\lambda}}%
\end{pmatrix}
.
\end{equation*}%
Then $\theta \ast \phi \in \big\{ \left( \Delta _{22},\ \Delta _{13},\ 0\right),\ \left( \Delta _{22}+\Delta _{33},\ \Delta _{13},\ 0\right) \big\} $. So we
get the superalgebras ${\rm A}_{03}$ and ${\rm A}_{04}$.

\item $\alpha _{1}=0,\alpha _{3}\neq 0$. We choose $\phi $ to be the
following automorphism:%
\begin{equation*}
\phi =%
\begin{pmatrix}
1 & 0 & 0 \\
0 & \alpha^{-\frac 12} _{3} & -\frac{\alpha _{2}\alpha
_{3}^{-\frac 32}}{2} \\
0 & 0 &  \alpha _{3}^{-\frac 12}%
\end{pmatrix}%
.
\end{equation*}%
Then $\theta \ast \phi =\big( \Delta _{23},\ \Delta _{13},\ 0\big) $. So we
get the superalgebra ${\rm A}_{05}$.

\item $\alpha _{1}=\alpha _{3}=0$. If $\alpha _{2}=0$, we get the
superalgebra ${\rm A}_{06}$. If $\alpha _{2}\neq 0$, we choose $\phi $ to
be the following automorphism:%
\begin{equation*}
\phi =%
\begin{pmatrix}
\alpha _{2} & 0 & 0 \\
0 & \alpha _{2} & 0 \\
0 & 0 & 1%
\end{pmatrix}%
.
\end{equation*}%
Then $\theta \ast \phi \ =\ \big( \Delta _{33},\ \Delta _{13}, \ 0\big) $. So we
get the superalgebra ${\rm A}_{07}$.
\end{itemize}

\item $%
\begin{pmatrix}
\alpha _{4} & \alpha _{5} \\
\alpha _{6} & \alpha _{7}%
\end{pmatrix}%
=%
\begin{pmatrix}
1 & 0 \\
0 & 1%
\end{pmatrix}%
$. Let $\phi $ be the following automorphism:%
\begin{equation*}
\phi =%
\begin{pmatrix}
a_{11} & 0 & 0 \\
0 & a_{22} & a_{23} \\
0 & a_{32} & a_{33}%
\end{pmatrix}%
.
\end{equation*}%
Hence $%
\begin{pmatrix}
\beta _{4} & \beta _{5} \\
\beta _{6} & \beta _{7}%
\end{pmatrix}%
=%
\begin{pmatrix}
1 & 0 \\
0 & 1%
\end{pmatrix}%
$ and $%
\begin{pmatrix}
-\beta _{3} & -\beta _{2} \\
\beta _{1} & \beta _{3}%
\end{pmatrix}%
=%
\begin{pmatrix}
a_{33} & -a_{23} \\
-a_{32} & a_{22}%
\end{pmatrix}%
\begin{pmatrix}
-\alpha _{3} & -\alpha _{2} \\
\alpha _{1} & \alpha _{3}%
\end{pmatrix}%
\begin{pmatrix}
a_{22} & a_{23} \\
a_{32} & a_{33}%
\end{pmatrix}%
$. So we may assume%
\begin{equation*}
\begin{pmatrix}
-\alpha _{3} & -\alpha _{2} \\
\alpha _{1} & \alpha _{3}%
\end{pmatrix}%
\ \in \  \left\{
\begin{pmatrix}
0 & 0 \\
0 & 0%
\end{pmatrix}%
, \
\begin{pmatrix}
-1 & 0 \\
0 & 1%
\end{pmatrix}%
, \
 \begin{pmatrix}
0 & - 1 \\
0 & 0%
\end{pmatrix}
\right\} .
\end{equation*}%
Thus we get the superalgebras ${\rm A}_{08}^{1},$ ${\rm A}_{09}^{1},$ and ${\rm A}_{10}^{1}$.

\item $%
\begin{pmatrix}
\alpha _{4} & \alpha _{5} \\
\alpha _{6} & \alpha _{7}%
\end{pmatrix}%
=%
\begin{pmatrix}
1 & 0 \\
0 & \alpha%
\end{pmatrix}%
$ with $\alpha \neq 1$.

\begin{itemize}
\item $\alpha _{1}\alpha _{2}\neq 0$. We choose $\phi $ to be the following
automorphism:%
\begin{equation*}
\phi =%
\begin{pmatrix}
1 & 0 & 0 \\
0 & \frac{1}{\sqrt{\alpha _{1}}} & 0 \\
0 & 0 & \frac{1}{\sqrt{\alpha _{2}}}%
\end{pmatrix}%
\text{.}
\end{equation*}%
Then $\theta \ast \phi \ =\ \big( \Delta _{22}+\Delta _{33}+\beta _{3}\Delta_{23},\ \Delta _{12},\ \alpha \Delta _{13}\big) $. So we get
${\rm A}_{11}^{\alpha \neq 1,\beta }$.

\item $\alpha _{1}=0,\alpha _{2}\neq 0$. Let $\phi=\phi_1 $ if $\alpha _{3}=0$ and $\phi=\phi_2$  if $\alpha _{3}\neq 0$:%
\begin{equation*}
\phi_1 \ = \ \begin{pmatrix}
1 & 0 & 0 \\
0 & 1 & 0 \\
0 & 0 & \frac{1}{\sqrt{\alpha _{2}}}%
\end{pmatrix}%
, \
\phi_2 \ = \ \begin{pmatrix}
1 & 0 & 0 \\
0 & \frac{\sqrt{\alpha _{2}}}{\alpha _{3}} & 0 \\
0 & 0 & \frac{1}{\sqrt{\alpha _{2}}}%
\end{pmatrix}%
.
\end{equation*}%
Then $\theta \ast \phi \in \big\{ \left( \Delta _{33}, \ \Delta _{12},\ \alpha \Delta _{13}\right) ,\ \left( \Delta _{33}+\Delta _{23},\ \Delta _{12},\ \alpha
\Delta _{13}\right) \big\} $. Hence we get the superalgebras ${\rm A}_{10}^{\alpha \neq 1}$ and ${\rm A}_{12}^{\alpha \neq 1}$.

\item $\alpha _{1}\neq 0,\alpha _{2}=0$. Assume first $\alpha =0$.
Let $\phi = \phi_1$  if $\alpha _{3}=0$ and $\phi=\phi_2$
if $\alpha _{3}\neq 0$:%
\begin{equation*}
\phi_1 \ = \ \begin{pmatrix}
1 & 0 & 0 \\
0 & \frac{1}{\sqrt{\alpha _{1}}} & 0 \\
0 & 0 & 1%
\end{pmatrix}%
, \
\phi_2 \ = \ \begin{pmatrix}
1 & 0 & 0 \\
0 & \frac{1}{\sqrt{\alpha _{1}}} & 0 \\
0 & 0 & \frac{\sqrt{\alpha _{1}}}{\alpha _{3}}%
\end{pmatrix}%
\text{.}
\end{equation*}%
Then $\theta \ast \phi \in \big\{ \left( \Delta _{22},\ \Delta _{12},\ 0\right), \ \left( \Delta _{22}+\Delta _{23},\ \Delta _{12},\ 0\right) \big\} $. So we
get   ${\rm A}_{13}$ and ${\rm A}_{14}$.

Assume now $\alpha \neq 0$.
Let $\phi=\phi_1 $  if $\alpha _{3}=0$ and $\phi=\phi_2$ if $\alpha _{3}\neq 0$:%
\begin{equation*}
\phi_1 \ = \ \begin{pmatrix}
{\alpha^{-1} } & 0 & 0 \\
0 & 0 & \frac{ {\bf i}}{\sqrt{\alpha \alpha _{1}}}  \\
0 & 1 & 0%
\end{pmatrix}%
\allowbreak, \
\phi_2 \ = \ \begin{pmatrix}
{\alpha^{-1} } & 0 & 0 \\
0 & 0 & \frac{1}{\sqrt{\alpha \alpha _{1}}} \\
0 & \frac{\alpha _{1}}{\alpha _{3}\sqrt{\alpha \alpha _{1}}} & 0%
\end{pmatrix}%
.
\end{equation*}%
Then $\theta \ast \phi \in \big\{ ( -\Delta _{33},\ \Delta _{12},\ {\alpha^{-1} }\Delta _{13}),\
  ( \Delta _{33}+\Delta _{23},\ \Delta _{12}, \
{\alpha^{-1} }\Delta _{13}) \big\} $. So we get
${\rm A}_{10}^{{\alpha^{-1} }}$ and ${\rm A}_{12}^{ {\alpha^{-1} }%
}$.

\item $\alpha _{1}=\alpha _{2}=0$. If $\alpha _{3}=0$, then
$\theta \ =\ \big( 0,\ \Delta _{12},\ \alpha \Delta _{13}\big) $. So we get the superalgebras $%
{\rm A}_{08}^{\alpha }$. If $\alpha _{3}\neq 0$, we choose $\phi $ to be
the following automorphism:%
\begin{equation*}
\phi =%
\begin{pmatrix}
1 & 0 & 0 \\
0 & 1 & 0 \\
0 & 0 & \frac{1}{\alpha _{3}}%
\end{pmatrix}%
.
\end{equation*}%
Then $\theta \ast \phi \ =\ \big( \Delta _{23},\ \Delta _{12},\ \alpha \Delta_{13}\big) $. So we get the superalgebras ${\rm A}_{09}^{\alpha \neq 1}$%
.
\end{itemize}

\item $%
\begin{pmatrix}
\alpha _{4} & \alpha _{5} \\
\alpha _{6} & \alpha _{7}%
\end{pmatrix}%
=%
\begin{pmatrix}
1 & 1 \\
0 & 1%
\end{pmatrix}%
$.

\begin{itemize}
\item $\alpha _{1}\neq 0$. We choose $\phi $ to be the following
automorphism:%
\begin{equation*}
\phi =%
\begin{pmatrix}
1 & 0 & 0 \\
0 & \alpha^{-\frac 12} _{1}  & - \alpha _{3} \alpha _{1}^{-\frac 32}%
 \\
0 & 0 &  \alpha^{-\frac 12} _{1}
\end{pmatrix}%
.
\end{equation*}%
Then $\theta \ast \phi \ =\ \big( \Delta _{22}+\alpha \Delta_{33},\ \Delta _{12}+\Delta_{13},\
\Delta _{13}\big) $. Thus we get the superalgebras ${\rm A}_{15}^{\alpha }$.

\item $\alpha _{1}=0,\alpha _{3}\neq 0$. We choose $\phi $ to be the
following automorphism:%
\begin{equation*}
\phi =%
\begin{pmatrix}
1 & 0 & 0 \\
0 &  {\alpha^{-\frac 12}  _{3}} & -\frac{\alpha _{2}\alpha^{-\frac 32} _{3}}{2 } \\
0 & 0 &  {\alpha^{-\frac 12}  _{3}}
\end{pmatrix}%
.
\end{equation*}%
Then $\theta \ast \phi \ =\ \big( \Delta _{23},\ \Delta _{12}+\Delta _{13},\ \Delta_{13}\big) $. So we get the superalgebra  ${\rm A}_{16}$.

\item $\alpha _{1}=\alpha _{3}=0$. If $\alpha _{2}=0$, we get the
superalgebra ${\rm A}_{17}$. If $\alpha _{2}\neq 0$, we choose $\phi $ to
be the following automorphism:%
\begin{equation*}
\phi =%
\begin{pmatrix}
1 & 0 & 0 \\
0 & \alpha^{-\frac 12} _{2}  & 0 \\
0 & 0 &\alpha^{-\frac 12} _{2}
\end{pmatrix}%
.
\end{equation*}%
Then $\theta \ast \phi \ =\ \big( \Delta _{33},\ \Delta _{12}+\Delta _{13},\ \Delta
_{13}\big) $. Hence we get the superalgebra ${\rm A}_{18}$.
\end{itemize}
\end{itemize}
\end{proof}

\subsubsection{The classification Theorem A1}\label{A1}
\begin{theoremA1}
Let $\rm{N}$ be a nontrivial $3$-dimensional noncommutative Jordan superal\-gebra of
type $\left( 1,2\right) $. Then $\rm{N}$ is isomorphic to one Jordan superalgebra listed in Theorem \ref{(1,2)},
or one anticommutative superalgebra listed in Theorem \ref{ant12},
or to one of the following
superalgebras:

\begin{longtable}{lllllllllllllll}

 \hline
${\rm N}_{01}^{\alpha \neq 0}$&$:$&$e_{1}f_{1}=\left( 1+\alpha \right) f_{2}$&$f_{1}e_{1}=\left( 1-\alpha \right) f_{2} $\\
&&$f_{1}f_{2}=\left(\alpha +1\right) e_{1}$&$f_{2}f_{1}=\left( \alpha -1\right) e_{1}$\\

${\rm N}_{02}$&$:$&$e_{1}f_{1}=f_{2}$&$f_{1}e_{1}=f_{2}$&$f_{1}f_{1}=e_{1}$\\
&&$f_{1}f_{2}=e_{1}$&$f_{2}f_{1}=-e_{1}
$\\

${\rm N}_{03}^{\alpha \neq 0}$&$:$&$f_{1}f_{2}=\left( \alpha+1\right) e_{1}$&$f_{2}f_{1}=\left( \alpha -1\right) e_{1}$\\

${\rm N}_{04}$&$:$&$f_{1}f_{2}=e_{1}$&$f_{2}f_{1}=-e_{1}$&$f_{2}f_{2}=-e_{1}$\\

${\rm N}_{05}^{\alpha }$&$:$&$e_{1}f_{1}=\left( 1+\alpha \right)f_{2}$&$f_{1}e_{1}=\left( 1-\alpha \right) f_{2}$&$f_{1}f_{1}=e_{1}$\\

${\rm N}_{06}^{\alpha \neq 0}$&$:$&$e_{1}f_{1}=\left( 1+\alpha\right) f_{2}$&$f_{1}e_{1}=\left( 1-\alpha \right) f_{2}$\\

${\rm N}_{07}^{\alpha \neq0}$&$:$&$e_{1}e_{1}=e_{1}$&$e_{1}f_{1}=\left( \frac{1}{2}+\alpha \right) f_{1}$&$f_{1}e_{1}=\left( \frac{1}{2}-\alpha \right) f_{1}$\\

${\rm N}_{08}^{\alpha \neq0}$&$:$&
$e_{1}e_{1}=e_{1}$&$e_{1}f_{1}=f_{1}$&$f_{1}e_{1}=f_{1}$\\
&&$e_{1}f_{2}=\left(
\frac{1}{2}+\alpha \right) f_{2}$&$f_{2}e_{1}=\left( \frac{1}{2}-\alpha \right) f_{2}$\\

 ${\rm N}_{09}$&$:$&$e_{1}e_{1}=e_{1}$&$e_{1}f_{1}=f_{1}$\\
 &&$f_{1}e_{1}=f_{1}$&$f_{1}f_{1}=e_{1}$\\

${\rm N}_{10}^{(\alpha ,\beta)\neq0 }$&$:$&
$e_{1}e_{1}=e_{1}$&$e_{1}f_{1}=\left(  \frac{1}{2}+\alpha \right) f_{1}$&$f_{1}e_{1}=\left( \frac{1}{2}-\alpha\right) f_{1}$\\
&&$e_{1}f_{2}=\left( \frac{1}{2}+\beta \right)
f_{2}$&$f_{2}e_{1}=\left( \frac{1}{2}-\beta \right) f_{2}$\\

${\rm N}_{11}^{\alpha }$&$:$&
$e_{1}e_{1}=e_{1}$&$e_{1}f_{1}=\left( \frac{1}{2}+\alpha \right) f_{1}$&$
f_{1}e_{1}=\left( \frac{1}{2}-\alpha \right)f_{1}$\\
&&$
e_{1}f_{2}=f_{1}+\left( \frac{1}{2}+\alpha \right)f_{2}$&$
f_{2}e_{1}=-f_{1}+\left( \frac{1}{2}-\alpha \right) f_{2}$\\

${\rm N}_{12}^{\alpha }$&$:$&
$e_{1}e_{1}=e_{1}$&$e_{1}f_{1}=\frac{1}{2}f_{1}+\alpha f_{2}$&$f_{1}e_{1}=\frac{1}{2}f_{1}-\alpha f_{2}$\\&&$e_{1}f_{2}=f_{1}+
\frac{1}{2}f_{2}$&$f_{2}e_{1}=-f_{1}+\frac{1}{2}f_{2}$&$f_{1}f_{1}=-2\alpha e_{1}$\\&&$f_{1}f_{2}=e_{1}$&$f_{2}f_{1}=-e_{1}$&$f_{2}f_{2}=2e_{1}$\\

 \hline
\end{longtable}

\noindent All listed superalgebras are non-isomorphic except:
${\rm N}_{03}^{\alpha}\cong {\rm N}_{03}^{-\alpha }$ and ${\rm N}_{10}^{\alpha ,\beta }\cong
{\rm N}_{10}^{\beta ,\alpha }.$
\end{theoremA1}

\subsubsection{The proof of  Theorem A1}
Let $\rm{N}$ be a nontrivial $3$-dimensional noncommutative Jordan
superalgebra of type $\left( 1,2\right) $. Then $\rm{N}^{+}$ is a $3$-dimensional Jordan superalgebra of type $\left( 1,2\right) $. Then we may
assume $\rm{N}^{+}\in \big\{ {\rm J}_{01},\ldots ,{\rm J}_{11}\big\} $. So we have the following cases:

\begin{enumerate}[I.]
    \item

\underline{$\rm{N}^{+}={\rm J}_{01}$}. Let $\theta \ =\ \big(B_{1},B_{2},B_{3}\big) \neq 0$ be an arbitrary element of ${\rm Z}^{2}\left(
{\rm J}_{01},{\rm J}_{01}\right) $. Then
\begin{center}
     $\theta\ =\
     \big( \alpha_{1}\Delta _{22}+\alpha _{2}\Delta _{23},\ 0,\ \alpha _{2}\Delta _{12}\big)$
for some $\alpha _{1},\alpha _{2}\in \mathbb{C}$.
\end{center} The automorphism group ${\rm Aut}\left( {\rm J}_{01}\right) $ consists of the automorphisms $\phi$ given by a matrix of the following form:
\begin{equation*}
\phi =%
\begin{pmatrix}
a_{11} & 0 & 0 \\
0 & \epsilon & 0 \\
0 & a_{32} & \epsilon a_{11}%
\end{pmatrix}%
:\epsilon ^{2}=1.
\end{equation*}

\begin{itemize}
\item $\alpha _{2}\neq 0$. We choose $\phi $ to be the following
automorphism:%
\begin{equation*}
\phi =%
\begin{pmatrix}
1 & 0 & 0 \\
0 & 1 & 0 \\
0 & - \frac{\alpha _{1}}{2\alpha _{2}} & 1%
\end{pmatrix}%
.
\end{equation*}%
Then $\theta \ast \phi \ =\ \big( \alpha _{2}\Delta _{23},\ 0,\ \alpha _{2}\Delta_{12}\big) $. So we get the superalgebras ${\rm N}_{01}^{\alpha
\neq 0}$.

\item $\alpha _{2}=0$. We choose $\phi $ to be the following automorphism:%
\begin{equation*}
\phi =%
\begin{pmatrix}
\alpha _{1} & 0 & 0 \\
0 & 1 & 0 \\
0 & 0 & \alpha _{1}%
\end{pmatrix}%
.
\end{equation*}%
Then $\theta \ast \phi  \ =\ \big( \Delta _{22},\ 0,\ 0\big) $. So we get the
superalgebra ${\rm N}_{02}$.
\end{itemize}

\item \underline{$\rm{N}^{+}={\rm J}_{02}$}. Let $\theta \ =\ \big(
B_{1},B_{2},B_{3}\big) $ be an arbitrary element of ${\rm Z}^{2}\left( {\rm J}_{02},{\rm J}_{02}\right) $. Then
\begin{center}
    $\theta \ = \ \big( \alpha _{1}\Delta_{22}+\alpha _{2}\Delta _{33}+\alpha _{3}\Delta _{23},\ 0,\ 0\big)$ for some $%
\alpha _{1},\alpha _{2},\alpha _{3}\in \mathbb{C}$.
\end{center} The automorphism group ${\rm Aut}\left( {\rm J}_{02}\right)$ consists
of the automorphisms $\phi $ given by a matrix of the following form:%
\begin{equation*}
\phi =%
\begin{pmatrix}
a_{22}a_{33}-a_{23}a_{32} & 0 & 0 \\
0 & a_{22} & a_{23} \\
0 & a_{32} & a_{33}%
\end{pmatrix}%
.
\end{equation*}

Let $\phi =\bigl(a_{ij}\bigr)\in $ ${\rm Aut}\left( {\rm J}_{02}\right) $%
. Then $\theta \ast \phi \ =\ \left(
\beta _{1}\Delta _{22}+\beta _{2}\Delta_{33}+\beta _{3}\Delta _{23},\ 0,\ 0\right) $ where%
\begin{longtable}{lcl}
$\beta _{1} $&$=$&$\frac{ \alpha
_{1}a_{22}^{2}+2\alpha _{3}a_{22}a_{32}+\alpha _{2}a_{32}^{2}}{a_{22}a_{33}-a_{23}a_{32}},$\\
$\beta _{2} $&$=$&$\frac{\alpha
_{1}a_{23}^{2}+2\alpha _{3}a_{23}a_{33}+\alpha _{2}a_{33}^{2}}{a_{22}a_{33}-a_{23}a_{32}},$ \\
$\beta _{3} $&$=$&$\frac{\alpha
_{1}a_{22}a_{23}+\alpha _{2}a_{32}a_{33}+\alpha _{3}a_{22}a_{33}+\alpha
_{3}a_{23}a_{32}}{a_{22}a_{33}-a_{23}a_{32}}.$
\end{longtable}
Whence%
\begin{equation*}
\begin{pmatrix}
-\beta _{3} & -\beta _{2} \\
\beta _{1} & \beta _{3}%
\end{pmatrix}%
=%
\begin{pmatrix}
a_{22} & a_{23} \\
a_{32} & a_{33}%
\end{pmatrix}%
^{-1}%
\begin{pmatrix}
-\alpha _{3} & -\alpha _{2} \\
\alpha _{1} & \alpha _{3}%
\end{pmatrix}%
\begin{pmatrix}
a_{22} & a_{23} \\
a_{32} & a_{33}%
\end{pmatrix}%
.
\end{equation*}%
So we may assume $%
\begin{pmatrix}
-\alpha _{3} & -\alpha _{2} \\
\alpha _{1} & \alpha _{3}%
\end{pmatrix}%
\in \left\{
\begin{pmatrix}
-\alpha & 0 \\
0 & \alpha%
\end{pmatrix}%
, \
\begin{pmatrix}
0 & 1 \\
0 & 0%
\end{pmatrix}%
\right\} $. Therefore we obtain the superalgebras ${\rm N}_{03}^{\alpha
  \neq 0}$ and ${\rm N}_{04}$.

\item
\underline{$\rm{N}^{+}={\rm J}_{03}$}. Let $\theta \ =\ \big(
B_{1},B_{2},B_{3}\big) \neq 0$ be an arbitrary element of ${\rm Z}^{2}\left(
{\rm J}_{03},{\rm J}_{03}\right) $. Then
\begin{center}$\theta \ =\ \big( \alpha
_{1}\Delta _{22},\ 0,\ \alpha _{2}\Delta _{12}\big)$ for some $\alpha
_{1},\alpha _{2}\in \mathbb{C}$.\end{center} The automorphism group  ${\rm Aut}\left( {\rm J}_{03}\right)$ consists of the automorphisms $%
\phi $ given by a matrix of the following form:%
\begin{equation*}
\phi =%
\begin{pmatrix}
a_{11} & 0 & 0 \\
0 & a_{22} & 0 \\
0 & a_{32} & a_{11}a_{22}%
\end{pmatrix}%
.
\end{equation*}

Let $\phi =\bigl(a_{ij}\bigr)\in $ ${\rm Aut}\left( {\rm J}_{03}\right).$ Then
\begin{center}$\theta \ast \phi \ =\ \left( \beta _{1}\Delta _{22},\ 0,\ \beta _{2}\Delta
_{12}\right),$ where
$\beta _{1}=\alpha _{1}a^{-1}_{11}a_{22}^{2}$ and $%
\beta _{2}=\alpha _{2}$.\end{center}
So we have the representatives
$\big( \Delta_{22},\ 0,\ \alpha \Delta _{12}\big)$ and
$\big( 0,\ 0,\ \alpha \Delta _{12}\big).$
So we get the superalgebras ${\rm N}_{05}^{\alpha }$ and ${\rm N}_{06}^{\alpha \neq 0}$.

\item
\underline{$\rm{N}^{+}={\rm J}_{04}$}. Let $\theta \ =\ \big(
B_{1},B_{2},B_{3}\big) $ be an arbitrary element of ${\rm Z}^{2}\left( {\rm J}_{04},{\rm J}_{04}\right) $. Then
\begin{center}$\theta \ =\
\big( 0,\ \alpha _{1}\Delta_{12},\ 0\big) $ for some $\alpha _{1}\in \mathbb{C}^{\ast}$.\end{center}
The
automorphism group ${\rm Aut}\left( {\rm J}_{04}\right) $ consists of the automorphisms $\phi $ given by a matrix of
the following form:%
\begin{equation*}
\phi =%
\begin{pmatrix}
1 & 0 & 0 \\
0 & a_{22} & 0 \\
0 & 0 & a_{33}%
\end{pmatrix}%
.
\end{equation*}

Let $\phi =\bigl(a_{ij}\bigr)\in $ ${\rm Aut}\left( {\rm J}_{04}\right).$
Then $\theta \ast \phi =\theta $. Hence we have the representatives $%
\big( 0,\ \alpha \Delta _{12},\ 0\big) $. Therefore we obtain the
superalgebras ${\rm N}_{07}^{\alpha \neq 0}$.

\item
\underline{$\rm{N}^{+}={\rm J}_{05}$}. Let $\theta \ =\ \big(B_{1},B_{2},B_{3}\big) \neq 0$ be an arbitrary element of ${\rm Z}^{2}\left(
{\rm J}_{05},{\rm J}_{05}\right) $. Then
\begin{center}
    $\theta\  =\ \big( 0,\ 0,\ \alpha_{1}\Delta _{13}\big)$ for some $\alpha _{1}\in \mathbb{C}^{\ast}$.
\end{center} The
automorphism group ${\rm Aut}\left( {\rm J}_{05}\right) $ consists of the automorphisms $\phi $ given by a matrix of the following form:
\begin{equation*}
\phi =%
\begin{pmatrix}
1 & 0 & 0 \\
0 & a_{22} & 0 \\
0 & 0 & a_{33}%
\end{pmatrix}%
.
\end{equation*}

Let $\phi =\bigl(a_{ij}\bigr)\in $ ${\rm Aut}\left( {\rm J}_{05}\right).$
Then $\theta \ast \phi =\theta $. Hence, we have the representatives $%
\big( 0,\ 0,\ \alpha \Delta _{13}\big) $. Thus, we get the superalgebras $%
{\rm N}_{08}^{\alpha \neq 0}$.

\item
\underline{$\rm{N}^{+}={\rm J}_{06}$}. Let $\theta \ =\ \big(B_{1},B_{2},B_{3}\big) \neq 0$ be an arbitrary element of ${\rm Z}^{2}\left(
{\rm J}_{06},{\rm J}_{06}\right) $. Then
\begin{center}
    $\theta \ =\ \big( \alpha_{1}\Delta _{22},\ 0,\ 0\big)$ for some $\alpha _{1}\in \mathbb{C}^{\ast}$.
\end{center}
The automorphism group ${\rm Aut}\left( {\rm J}_{06}\right) $ consists of the automorphisms $\phi $ given by a matrix of
the following form:%
\begin{equation*}
\phi =%
\begin{pmatrix}
1 & 0 & 0 \\
0 & a_{22} & 0 \\
0 & 0 & a_{33}%
\end{pmatrix}%
.
\end{equation*}%
Let $\phi =\bigl(a_{ij}\bigr)\in $ ${\rm Aut}\left( {\rm J}_{06}\right) $%
. Then $\theta \ast \phi \  =\
\left( \alpha _{1}a_{22}^{2}\Delta_{22},\ 0,\ 0\right) $. Hence we have the representative
$\big( \Delta_{22},\ 0,\ 0\big) $. Thus we get the superalgebra ${\rm N}_{09}$.%

\item
\underline{$\rm{N}^{+}={\rm J}_{07}$}. Then ${\rm Z}^{2}\left(
{\rm J}_{07},{\rm J}_{07}\right) =\left\{ 0\right\} $.

\item
\underline{$\rm{N}^{+}={\rm J}_{08}$}. Let $\theta \ =\ \big(B_{1},B_{2},B_{3}\big) $ be an arbitrary element of ${\rm Z}^{2}\left( {\rm J}_{08},{\rm J}_{08}\right) $.
 Then
 \begin{center}
     $\theta \ =\ \big( 0,\ \alpha _{1}\Delta_{12}+\alpha _{2}\Delta _{13},\ \alpha _{3}\Delta_{12}+\alpha _{4}\Delta_{13}\big)$ for some $\alpha _{1},\alpha _{2},\alpha _{3},\alpha _{4}\in
\mathbb{C}$.
 \end{center} The automorphism group ${\rm Aut}\left(
{\rm J}_{08}\right) $ consists of the automorphisms $\phi $ given by a
matrix of the following form:
\begin{equation*}
\phi =%
\begin{pmatrix}
1 & 0 & 0 \\
0 & a_{22} & a_{23} \\
0 & a_{32} & a_{33}%
\end{pmatrix}%
.
\end{equation*}%
Let $\phi =\bigl(a_{ij}\bigr)\in $ ${\rm Aut}\left( {\rm J}_{08}\right)
$. Then $\theta \ast \phi \ =\
\left( 0,\ \beta _{1}\Delta _{12}+\beta _{2}\Delta_{13},\ \beta _{3}\Delta _{12}+\beta _{4}\Delta _{13}\right),$ where%
\begin{longtable}{lcl}
$\beta _{1} $&$=$&$\frac{\alpha
_{1}a_{22}a_{33}-\alpha _{3}a_{22}a_{23}+\alpha _{2}a_{32}a_{33}-\alpha
_{4}a_{23}a_{32}}{a_{22}a_{33}-a_{23}a_{32}},$ \\[2mm]
$\beta _{2} $&$=$&$\frac{\alpha
_{2}a_{33}^{2}-\alpha _{3}a_{23}^{2}+\alpha _{1}a_{23}a_{33}-\alpha
_{4}a_{23}a_{33}}{a_{22}a_{33}-a_{23}a_{32}},$ \\[2mm]
$\beta _{3} $&$=$&$-\frac{\alpha
_{2}a_{32}^{2}-\alpha _{3}a_{22}^{2}+\alpha _{1}a_{22}a_{32}-\alpha
_{4}a_{22}a_{32}}{a_{22}a_{33}-a_{23}a_{32}},$\\[2mm]
$\beta _{4} $&$=$&$-\frac{\alpha
_{1}a_{23}a_{32}-\alpha _{3}a_{22}a_{23}+\alpha _{2}a_{32}a_{33}-\alpha
_{4}a_{22}a_{33}}{a_{22}a_{33}-a_{23}a_{32}}.$
\end{longtable}%
Whence $%
\begin{pmatrix}
\beta _{1} & \beta _{2} \\
\beta _{3} & \beta _{4}%
\end{pmatrix}%
=%
\begin{pmatrix}
a_{22} & a_{23} \\
a_{32} & a_{33}%
\end{pmatrix}%
^{-1}%
\begin{pmatrix}
\alpha _{1} & \alpha _{2} \\
\alpha _{3} & \alpha _{4}%
\end{pmatrix}%
\begin{pmatrix}
a_{22} & a_{23} \\
a_{32} & a_{33}%
\end{pmatrix}%
\allowbreak $. So we may assume%
\begin{equation*}
\begin{pmatrix}
\alpha _{1} & \alpha _{2} \\
\alpha _{3} & \alpha _{4}%
\end{pmatrix}%
\in \left\{
\begin{pmatrix}
\alpha & 0 \\
0 & \beta%
\end{pmatrix}%
,%
\begin{pmatrix}
\alpha & 1 \\
0 & \alpha%
\end{pmatrix}%
\right\} .
\end{equation*}%
Thus we get the superalgebras ${\rm N}_{10}^{\alpha ,\beta }$ and $%
{\rm N}_{11}^{\alpha }$.

\item
\underline{$\rm{N}^{+}={\rm J}_{09}$}. Then ${\rm Z}^{2}\left( {\rm J}_{09},{\rm J}_{09}\right) =\left\{ 0\right\} $.

\item
\underline{$\rm{N}^{+}={\rm J}_{10}$}. Let $\theta \ =\ \big(B_{1},B_{2},B_{3}\big) $ be an arbitrary element of ${\rm Z}^{2}\left( {\rm J}_{10},{\rm J}_{10}\right) $. Then
\begin{center}
$\theta \ =\
\big( -2\alpha _{1}\Delta _{22}+2\alpha _{2}\Delta _{23}+2\alpha
_{3}\Delta _{33}, \ \alpha _{2}\Delta _{12}+\alpha _{3}\Delta _{13},\ \alpha
_{1}\Delta _{12}-\alpha _{2}\Delta _{13}\big),$
\end{center}
for some $\alpha _{1},\alpha _{2},\alpha _{3},\alpha _{4}\in \mathbb{C}$.
The automorphism group  ${\rm Aut}\left( {\rm J}_{10}\right) $ consists of the automorphisms $\phi $ given by a matrix of
the following form:
\begin{equation*}
\begin{pmatrix}
1 & 0 & 0 \\
0 & a_{22} & a_{23} \\
0 & a_{32} & a_{33}
\end{pmatrix}
, \mbox{ where } \ a_{22}a_{33} - a_{23}a_{32} =1.
\end{equation*}%
Without any loss of generality we may assume $\alpha _{3}\neq 0$. To see
this, assume $\alpha _{3}=0$. Let $\phi =\bigl(a_{ij}\bigr)\in $ ${\rm Aut}%
\left( {\rm J}_{10}\right) $. Write%
\begin{center}
$\theta \ast \phi\ =\ \big( -2\beta _{1}\Delta _{22}+2\beta _{2}\Delta_{23}+2\beta _{3}\Delta _{33},\ \beta _{2}\Delta _{12}+\beta _{3}\Delta_{13},\
\beta _{1}\Delta _{12}-\beta _{2}\Delta _{13}\big).$
\end{center}
Now, let $\phi=\phi_1 $ if $\alpha _{2}\neq
0$ and $\phi=\phi_2$   if $\alpha _{2}=0$:%
\begin{equation*}
\phi_1 \ =\ \begin{pmatrix}
1 & 0 & 0 \\
0 & \alpha _{2} & 1 \\
0 & \frac{\alpha _{1}-1}{2} & \frac{\alpha
_{1}+1}{2\alpha _{2}}
\end{pmatrix}%
, \
\phi_2 \ = \ \begin{pmatrix}
1 & 0 & 0 \\
0 & 0 & -\frac{{\bf i}}{\sqrt{\alpha _{1}}} \\
0 & - {\bf i}\sqrt{\alpha _{1}} & 0%
\end{pmatrix}%
.
\end{equation*}%
Then $\beta _{3}=1$. So we may assume $\alpha _{3}\neq 0$. Choose $\phi $ to
be the following automorphism:%
\begin{equation*}
\phi =%
\begin{pmatrix}
1 & 0 & 0 \\
0 & \sqrt{2\alpha _{3}} & 0 \\
0 & -{\alpha _{2}} \sqrt{{2}{\alpha^{-1} _{3}}} & \frac{1}{\sqrt{%
2\alpha _{3}}}%
\end{pmatrix}%
.
\end{equation*}%
Then $\theta \ast \phi \ =\ \big( -2\alpha\Delta _{22}+2\Delta _{33}, \
\Delta_{13},\ \alpha\Delta _{12}\big) $. So we get the superalgebras ${\rm N}_{12}^{\alpha }$.

\item
\underline{$\rm{N}^{+}={\rm J}_{11}$}. Then ${\rm Z}^{2}\left( {\rm J}_{11},{\rm J}_{11}\right) =\left\{ 0\right\} $.

\end{enumerate}

\subsection{Noncommutative Jordan superalgebras of type $(2,1)$}

\begin{theorem}
\label{(2,1)}Let $\mathbf{J}$ be a $3$-dimensional Jordan superalgebra of
type $\left( 2,1\right) $. Then $\mathbf{J}$ is isomorphic to one of the
following superalgebras:

\begin{longtable}{l llllll |r }

 \hline
$\mathbf{J}_{01} $&$:$&$e_{1}\circ e_{1}=e_{1}$&&$e_{2}\circ e_{2}=e_{2}$ & && $\mathfrak{ass}$\\

$\mathbf{J}_{02}$&$:$&$e_{1}\circ e_{1}=e_{1}$&&$e_{2}\circ e_{2}=e_{2}$&$e_{1}\circ f_{1}=f_{1}$ && $\mathfrak{ass}$\\

$\mathbf{J}_{03}$&$:$&$e_{1}\circ e_{1}=e_{1}$&&$e_{2}\circ e_{2}=e_{2}$&$e_{1}\circ f_{1}=\frac{1}{2}f_{1}$ &&$\mathfrak{non-ass}$\\

$\mathbf{J}_{04}$&$:$&$e_{1}\circ e_{1}=e_{1}$&&$e_{2}\circ e_{2}=e_{2}$&$e_{1}\circ f_{1}=\frac{1}{2}f_{1}$&$e_{2}\circ f_{1}=\frac{1}{2} f_{1}$ &$\mathfrak{non-ass}$\\

$\mathbf{J}_{05}$&$:$&$e_{1}\circ e_{1}=e_{1}$ &&&&& $\mathfrak{ass}$\\

$\mathbf{J}_{06}$&$:$&$e_{1}\circ e_{1}=e_{1}$&&&$e_{1}\circ f_{1}=\frac{1}{2}f_{1}$ & &$\mathfrak{non-ass}$\\

 $\mathbf{J}_{07}$&$:$&$e_{1}\circ e_{1}=e_{1}$&&&$e_{1}\circ f_{1}=f_{1}$  & &$\mathfrak{ass}$\\

$\mathbf{J}_{08}$&$:$&$e_{1}\circ e_{1}=e_{1}$&$e_{1}\circ e_{2}=e_{2}$ &&&&$\mathfrak{ass}$\\

$\mathbf{J}_{09}$&$:$&$e_{1}\circ e_{1}=e_{1}$&$e_{1}\circ e_{2}=e_{2}$&&$e_{1}\circ f_{1}=\frac{1}{2}f_{1}$ &&$\mathfrak{non-ass}$\\

 $\mathbf{J}_{10}$&$:$&$e_{1}\circ e_{1}=e_{1}$&$e_{1}\circ e_{2}=e_{2} $&&$e_{1}\circ f_{1}=f_{1}$ &&$\mathfrak{ass}$\\

 $\mathbf{J}_{11}$&$:$&$e_{1}\circ e_{1}=e_{1}$&$e_{1}\circ e_{2}=\frac{1}{2}e_{2}$&&&&$\mathfrak{non-ass}$\\

 $\mathbf{J}_{12}$&$:$&$e_{1}\circ e_{1}=e_{1}$&$e_{1}\circ e_{2}=\frac{1}{2}e_{2}$&&$e_{1}\circ f_{1}=\frac{1}{2}f_{1}$ &&$\mathfrak{non-ass}$\\

 $\mathbf{J}_{13}$&$:$&$e_{1}\circ e_{1}=e_{1}$&$e_{1}\circ e_{2}=\frac{1}{2}e_{2}$&&$e_{1}\circ f_{1}=f_{1}$ &&$\mathfrak{non-ass}$\\

 $\mathbf{J}_{14}$&$:$&$e_{1}\circ e_{1}=e_{2}$ &&&&&$\mathfrak{ass}$\\

 \hline

\end{longtable}
\end{theorem}

\subsubsection{Anticommutative superalgebras of type $(2,1)$}

\begin{theorem}\label{ant21}
Let $\mathbf{A}$ be a nontrivial $3$-dimensional anticommutative
superalgebra of type $\left( 2,1\right) $. Then $\mathbf{A}$ is isomorphic
to one of the following superalgebras:

\begin{longtable}{lllllllllllllllllllllllllll}

\hline
   ${\rm A} $&$:$&$[e_{1},e_{2}] $&  $[e_{1},f_{1}]$&
 $[e_{2},f_{1}] $&$[f_{1},f_{1}] $  \\
\hline

$\mathbf{A}_{01}^{\alpha} $&$:$&$[e_{1},e_{2}]=e_{1}$& &$[e_{2},f_{1}]=\alpha f_{1}$& $[f_{1},f_{1}]=e_{2}$\\

$\mathbf{A}_{02}^{\alpha}$&$:$&$[e_{1},e_{2}]=e_{1}$& $[e_{1},f_{1}]=f_{1}$& $[e_{2},f_{1}]=\alpha f_{1}$& $[f_{1},f_{1}]=e_{2}$\\

$\mathbf{A}_{03}^{\alpha}$&$:$&$[e_{1},e_{2}]=e_{1}$&& $[e_{2},f_{1}]=\alpha f_{1}$&  $[f_{1},f_{1}]=e_{1}$\\

$\mathbf{A}_{04}$&$:$&$[e_{1},e_{2}]=e_{1}$& $[e_{1},f_{1}]=f_{1}$& &$[f_{1},f_{1}]=e_{1}$\\

$\mathbf{A}_{05}^{\alpha}$&$:$&$[e_{1},e_{2}]=e_{1}$& &$[e_{2},f_{1}]=\alpha f_{1}$ \\

$\mathbf{A}_{06}$&$:$&$[e_{1},e_{2}]=e_{1}$& $[e_{1},f_{1}]=f_{1}$& \\

$\mathbf{A}_{07}$&$:$&&$[e_{1},f_{1}]=f_{1}$&& \\

$\mathbf{A}_{08}$&$:$&&$[e_{1},f_{1}]=f_{1}$& &$[f_{1},f_{1}]=e_{2}$\\

$\mathbf{A}_{09}$&$:$&&$[e_{1},f_{1}]=f_{1}$& &$[f_{1},f_{1}]=e_{1}$\\

$\mathbf{A}_{10}$&$:$&&&&$[f_{1},f_{1}]=e_{1}$\\
\hline
\end{longtable}

\end{theorem}

\begin{proof}
 Let $\theta \ =\ \big(B_{1},B_{2},B_{3}\big) \neq 0$ be an arbitrary element of ${\rm Z}^{2}(
\mathbb{C}^{2,1},\mathbb{C}^{2,1}) $. Then
\begin{center}$\theta \ =\
\big( \alpha_{1}\Delta _{12}+\alpha _{2}\Delta _{33},\
\alpha _{3}\Delta _{12}+\alpha_{4}\Delta _{33},\
\alpha _{5}\Delta _{13}+\alpha _{6}\Delta _{23}\big) $
for some $\alpha _{1},\ldots ,\alpha _{6}\in \mathbb{C}$.\end{center}
The automorphism
group ${\rm Aut}\left( \mathbb{C}^{2,1}\right)$
consists of the automorphisms $\phi $ given by a matrix of the following
form:%
\begin{equation*}
\phi =%
\begin{pmatrix}
a_{11} & a_{12} & 0 \\
a_{21} & a_{22} & 0 \\
0 & 0 & a_{33}%
\end{pmatrix}%
.
\end{equation*}%
Let $\phi =\bigl(a_{ij}\bigr)\in $ ${\rm Aut}\left( \mathbb{C}^{2,1}\right)
$. Then
\begin{center}$\theta \ast \phi \ =\ \left( \beta _{1}\Delta _{12}+\beta _{2}\Delta
_{33},\ \beta _{3}\Delta _{12}+\beta _{4}\Delta _{33}, \ \beta _{5}\Delta
_{13}+\beta _{6}\Delta _{23}\right),$\end{center} where%
\begin{longtable}{lcllcllcl}
$\beta _{1} $&$=$&$\alpha _{1}a_{22}-\alpha _{3}a_{12},$ &
$\beta _{3} $&$=$&$\alpha _{3}a_{11}-\alpha _{1}a_{21}, $ &
$\beta _{5} $&$=$&$\alpha _{5}a_{11}+\alpha _{6}a_{21},$ \\

$\beta _{2} $&$=$&$\frac{a_{33}^{2}\big( \alpha
_{2}a_{22}-\alpha _{4}a_{12}\big) }{a_{11}a_{22}-a_{12}a_{21}},$ &
$\beta _{4} $&$=$&$-\frac{a_{33}^{2}\big( \alpha_{2}a_{21}-\alpha _{4}a_{11}\big)}{a_{11}a_{22}-a_{12}a_{21}},$ &
$\beta _{6} $&$=$&$\alpha _{5}a_{12}+\alpha _{6}a_{22}.$
\end{longtable}%
From here, we may assume $\left( \alpha _{1},\alpha _{3}\right) \in
\big\{\left( 1,0\right) ,\left( 0,0\right) \big\} $.

\begin{itemize}
\item $\left( \alpha _{1},\alpha _{3}\right) =\left( 1,0\right) $.

\begin{itemize}
\item $\alpha _{4}\neq 0$. Let
$\phi=\phi_1 $ if $\alpha _{5}=0$ and $\phi=\phi_2$ if $\alpha _{5}\neq 0$:
\begin{equation*}
\phi_1 \ = \
\begin{pmatrix}
1 & \frac{\alpha _{2}}{\alpha _{4}} & 0 \\
0 & 1 & 0 \\
0 & 0 & \frac{1}{\sqrt{\alpha _{4}}}%
\end{pmatrix}%
, \
\phi_2 \ = \ \begin{pmatrix}
\frac{1}{\alpha _{5}} & \frac{\alpha _{2}}{\alpha _{4}} & 0 \\
0 & 1 & 0 \\
0 & 0 & \frac{1}{\sqrt{\alpha _{4}}}%
\end{pmatrix}%
.
\end{equation*}%
Then
\begin{center}
    $\theta \ast \phi \in
    \big\{ \left( \Delta _{12},\ \Delta _{33},\ \alpha\Delta _{23}\right) ,\
    \left( \Delta _{12},\ \Delta _{33},\ \ \Delta _{13}+\alpha
\Delta _{23}\right) \big\} $.
\end{center} So we get the superalgebras $\mathbf{A}_{01}^{\alpha }$ and $\mathbf{A}_{02}^{\alpha }$.

\item $\alpha _{4}=0,\alpha _{2}\neq 0$. Let $\phi=\phi_1$  if $\alpha _{5}=0$ and  $\phi=\phi_2$ if $\alpha _{5}\neq 0$:%
\begin{equation*}
\phi_1\ = \ \begin{pmatrix}
1 & 0 & 0 \\
0 & 1 & 0 \\
0 & 0 & \frac{1}{\sqrt{\alpha _{2}}}%
\end{pmatrix}%
, \
\phi_2 \ = \ \begin{pmatrix}
\frac{1}{\alpha _{5}} & -\frac{\alpha _{6}}{\alpha _{5}} & 0 \\
0 & 1 & 0 \\
0 & 0 & \frac{1}{\sqrt{\alpha _{2}\alpha _{5}}}%
\end{pmatrix}%
.
\end{equation*}%
Then
\begin{center}$\theta \ast \phi \ \in \
\big\{ \left( \Delta _{12}+\Delta _{33},\ 0,\ \alpha\Delta _{23}\right),\
\left( \Delta _{12}+\Delta _{33},\ 0,\ \Delta _{13}\right) \big\}$.
\end{center}
So we get the superalgebras $\mathbf{A}_{03}^{\alpha }$ and $\mathbf{A}_{04}$.

\item $\alpha _{4}=\alpha _{2}=0$. Let $\phi=\phi_1 $ if $\alpha _{5}=0$ and
$\phi=\phi_2$  if $\alpha _{5}\neq 0$:%
\begin{equation*}
\phi_1\ =\ \begin{pmatrix}
1 & 0 & 0 \\
0 & 1 & 0 \\
0 & 0 & 1%
\end{pmatrix}%
, \
\phi_2 \ = \ \begin{pmatrix}
\frac{1}{\alpha _{5}} & -\frac{\alpha _{6}}{\alpha _{5}} & 0 \\
0 & 1 & 0 \\
0 & 0 & 1%
\end{pmatrix}%
.
\end{equation*}%
Then
\begin{center}$\theta \ast \phi \in \big\{
\left( \Delta _{12},\ 0,\ \alpha \Delta_{23}\right), \
\left( \Delta _{12},\ 0,\ \Delta _{13}\right) \big\} $.
\end{center} So we
get the superalgebras $\mathbf{A}_{05}^{\alpha }$ and $\mathbf{A}_{06}$.
\end{itemize}

\item $\left( \alpha _{1},\alpha _{3}\right) =\left( 0,0\right) $. Then we
may assume $\left( \alpha _{5},\alpha _{6}\right) \in \left\{ \left(
1,0\right) ,\left( 0,0\right) \right\} $.

\begin{itemize}
\item $\left( \alpha _{5},\alpha _{6}\right) =\left( 1,0\right) $. If $%
\left( \alpha _{2},\alpha _{4}\right) =\left( 0,0\right) $, we get the
superalgebra $\mathbf{A}_{07}$. Assume now $\left( \alpha _{2},\alpha
_{4}\right) \neq \left( 0,0\right) $. Let $\phi=\phi_1 $ if $\alpha _{2}=0$ and
$\phi=\phi_2$   if $\alpha _{2}\neq 0$:%
\begin{equation*}
\phi_1\ =\ \begin{pmatrix}
1 & 0 & 0 \\
0 & \alpha _{4} & 0 \\
0 & 0 & 1%
\end{pmatrix}%
, \
\phi_2 \ = \ \begin{pmatrix}
1 & 0 & 0 \\
\frac{\alpha _{4}}{\alpha _{2}} & 1 & 0 \\
0 & 0 & \frac{1}{\sqrt{\alpha _{2}}}%
\end{pmatrix}%
.
\end{equation*}%
Then $\theta \ast \phi \in \big\{ \left( 0,\ \Delta _{33},\ \Delta _{13}\right),\
\left( \Delta _{33},\ 0,\ \Delta _{13}\right) \big\}$. So we get
  $\mathbf{A}_{08}$ and $\mathbf{A}_{09}$.

\item $\left( \alpha _{5},\alpha _{6}\right) =\left( 0,0\right) $. Then we
may assume $\left( \alpha _{2},\alpha _{4}\right) =\left( 1,0\right) $. So
we get     $\mathbf{A}_{10}$.
\end{itemize}
\end{itemize}
\end{proof}

\subsubsection{The classification Theorem A2}\label{A2}
\begin{theoremA2}Let $\mathbf{N}$ be a nontrivial $3$-dimensional noncommutative Jordan superalgebra of
type $\left( 2,1\right) $. Then $\mathbf{N}$ is isomorphic to one
Jordan superalgebra listed in Theorem \ref{(2,1)},
or one anticommutative superalgebra listed in Theorem \ref{ant21},
or to one of the following
superalgebras:

\begin{longtable}{lllllllllllllllllll}

$\mathbf{N}_{01} $&$:$&$e_{1}e_{1}=e_{1}$&$e_{1}f_{1}=f_{1}$&$f_{1}e_{1}=f_{1}$\\
&&$e_{2}e_{2}=e_{2}$&$f_{1}f_{1}=e_{1}$\\

$\mathbf{N}_{02}^{\alpha \neq 0}$&$:$&
$e_{1}e_{1}=e_{1}$&$e_{1}f_{1}=\left( \frac{1}{2}+\alpha \right) f_{1}$&$f_{1}e_{1}=\left( \frac{1}{2}-\alpha
\right) f_{1}$&$e_{2}e_{2}=e_{2}$\\

$\mathbf{N}_{03}^{\alpha \neq 0}$&$:$&
$e_{1}e_{1}=e_{1}$&$e_{1}f_{1}=\left( \frac{1}{2}+\alpha \right) f_{1}$&$f_{1}e_{1}=\left( \frac{1}{2}-\alpha
\right) f_{1}$\\
&&$e_{2}f_{1}=\left( \frac{1}{2}-\alpha \right)f_{1}$&$f_{1}e_{2}=\left( \frac{1}{2}+\alpha \right) f_{1}$&$e_{2}e_{2}=e_{2}$\\

$\mathbf{N}_{04}$&$:$&$e_{1}e_{1}=e_{1}$&$f_{1}f_{1}=e_{2}$\\

$\mathbf{N}_{05}$&$:$&$e_{1}e_{1}=e_{1}$&$e_{2}f_{1}=f_{1}$ &$f_{1}e_{2}=-f_{1}$&$f_{1}f_{1}=e_{2}$\\

$\mathbf{N}_{06}$&$:$&$e_{1}e_{1}=e_{1}$&$e_{2}f_{1}=f_{1}$&$f_{1}e_{2}=-f_{1}$\\

$\mathbf{N}_{07}^{\alpha \neq 0}$&$:$&$e_{1}e_{1}=e_{1}$&$e_{1}f_{1}=\left( \frac{1}{2}+\alpha \right)f_{1}$&$f_{1}e_{1}=\left( \frac{1}{2}-\alpha\right) f_{1}$\\

$\mathbf{N}_{08}$&$:$&$e_{1}e_{1}=e_{1}$&$e_{1}f_{1}=f_{1}$&$f_{1}e_{1}=f_{1}$&$f_{1}f_{1}=e_{1}$\\

$\mathbf{N}_{09}^{\alpha \neq0}$&$:$&$e_{1}e_{1}=e_{1}$&$e_{1}e_{2}=e_{2}$&$e_{2}e_{1}=e_{2}$\\
&&$e_{1}f_{1}=\left( \frac{1}{2}+\alpha \right) f_{1}$&$f_{1}e_{1}=\left( \frac{1}{2}-\alpha
\right) f_{1}$\\

$\mathbf{N}_{10}$&$:$&$e_{1}e_{1}=e_{1}$&$e_{1}e_{2}=e_{2}$&$e_{2}e_{1}=e_{2}$\\
&&$e_{1}f_{1}=f_{1}$&$f_{1}e_{1}=f_{1}$&  $f_{1}f_{1}=e_{2}$ \\

$\mathbf{N}_{11}$&$:$&$e_{1}e_{1}=e_{1}$&$e_{1}e_{2}=e_{2}$&$e_{2}e_{1}=e_{2}$&$e_{1}f_{1}=f_{1}$\\
&&$f_{1}e_{1}=f_{1}$&$e_{2}f_{1}=f_{1}$&$f_{1}e_{2}=-f_{1}$& $f_{1}f_{1}=e_{2}$ \\

$\mathbf{N}_{12}$&$:$&$e_{1}e_{1}=e_{1}$&$e_{1}e_{2}=e_{2}$&$e_{2}e_{1}=e_{2}$\\
&&$e_{1}f_{1}=f_{1}$&$f_{1}e_{1}=f_{1}$&$e_{2}f_{1}=f_{1}$&$f_{1}e_{2}=-f_{1}$\\

$\mathbf{N}_{13}^{\alpha \neq 0}$&$:$&$e_{1}e_{1}=e_{1}$&
$e_{1}e_{2}=\left(
\frac{1}{2}+\alpha \right) e_{2}$&$e_{2}e_{1}=\left( \frac{1}{2}-\alpha
\right) e_{2}$\\

$\mathbf{N}_{14}^{\alpha }$&$:$&
$e_{1}e_{1}=e_{1}$&$e_{1}e_{2}=\left( \frac{1}{2}+\alpha \right) e_{2}$&$e_{2}e_{1}=\left( \frac{1}{2}-\alpha \right)
e_{2}$&$e_{1}f_{1}=\frac{1}{2}f_{1}$\\
&&$f_{1}e_{1}=\frac{1}{2}f_{1}$&$e_{2}f_{1}=f_{1}$&$f_{1}e_{2}=-f_{1}$&$f_{1}f_{1}=e_{2}$\\

 $\mathbf{N}_{15}^{\alpha }$&$:$&
$e_{1}e_{1}=e_{1}$&$e_{1}e_{2}=\left( \frac{1}{2}+\alpha \right) e_{2}$&$e_{2}e_{1}=\left( \frac{1}{2}-\alpha \right)
e_{2}$\\
&&$e_{1}f_{1}=\frac{1}{2}f_{1}$&$f_{1}e_{1}=\frac{1}{2}f_{1}$&$e_{2}f_{1}=f_{1}$&$f_{1}e_{2}=-f_{1}$\\

$\mathbf{N}_{16}^{\alpha ,\beta }$&$:$&
$e_{1}e_{1}=e_{1}$&$e_{1}e_{2}=\left( \frac{1}{2}+\alpha \right) e_{2}$&$e_{2}e_{1}=\left( \frac{1}{2}-\alpha
\right) e_{2}$\\
&&$e_{1}f_{1}=\left( \frac{1}{2}+\beta \right)
f_{1}$&$f_{1}e_{1}=\left( \frac{1}{2}-\beta \right) f_{1}$&$f_{1}f_{1}=e_{2}$\\

$\mathbf{N}_{17}^{\alpha ,\beta }$&$:$&$e_{1}e_{1}=e_{1}$&
$e_{1}e_{2}=\left( \frac{1}{2}+\alpha \right) e_{2}$&$e_{2}e_{1}=\left( \frac{1}{2}-\alpha
\right) e_{2}$\\
&&$e_{1}f_{1}=\left( \frac{1}{2}+\beta \right)f_{1}$&$f_{1}e_{1}=\left( \frac{1}{2}-\beta \right) f_{1}$\\

$\mathbf{N}_{18}^{\alpha \neq 0}$&$:$&$e_{1}e_{1}=e_{1}$&
$e_{1}e_{2}=\left( \frac{1}{2}+\alpha \right) e_{2}$&
$e_{2}e_{1}=\left( \frac{1}{2}-\alpha\right) e_{2}$\\
&&
$e_{1}f_{1}=f_{1}$&$f_{1}e_{1}=f_{1}$\\

$\mathbf{N}_{19}$&$:$&$e_{1}e_{1}=e_{2}$&$e_{1}f_{1}=f_{1}$&$f_{1}e_{1}=-f_{1}$&$f_{1}f_{1}=e_{2}$\\

$\mathbf{N}_{20}$&$:$&$e_{1}e_{1}=e_{2}$&$f_{1}f_{1}=e_{2}$\\

$\mathbf{N}_{21}$&$:$&$e_{1}e_{1}=e_{2}$&$e_{1}f_{1}=f_{1}$&$f_{1}e_{1}=-f_{1}$

\end{longtable}
\noindent All listed superalgebras are non-isomorphic except: $\mathbf{N}_{03}^{\alpha
}\cong \mathbf{N}_{03}^{-\alpha }.$
\end{theoremA2}

\subsubsection{The proof of  Theorem A2}
Let $\mathbf{N}$ be a nontrivial $3$-dimensional noncommutative Jordan
superalgebra of type $\left( 2,1\right) $. Then $\mathbf{N}^{+}$ is a $3$%
-dimensional Jordan superalgebra of type $\left( 2,1\right) $. Then we may
assume $\mathbf{N}^{+}\in \big\{ \mathbf{J}_{01},\ldots ,\mathbf{J}_{14}\big\} $. So we have the following cases:

\begin{enumerate}[I.]

\item
\underline{$\mathbf{N}^{+}=\mathbf{J}_{01}$}. Then ${\rm Z}^{2}\left( \mathbf{J}_{01},\mathbf{J}_{01}\right) =\left\{ 0\right\} $.

\item
\underline{$\mathbf{N}^{+}=\mathbf{J}_{02}$}. Let $\theta \ =\ \big(B_{1},B_{2},B_{3}\big) \neq 0$ be an arbitrary element of ${\rm Z}^{2}\left(
\mathbf{J}_{02},\mathbf{J}_{02}\right) $. Then
\begin{center}$\theta \ =\ \big( \alpha
_{1}\Delta _{33},\ 0,\ 0\big)$ for some $\alpha _{1}\in \mathbb{C}^{\ast
^{{}}}$.\end{center} The automorphism group ${\rm Aut}\left(
\mathbf{J}_{02}\right) $ consists of the automorphisms $\phi $ given by a
matrix of the following form:%
\begin{equation*}
\phi =%
\begin{pmatrix}
1 & 0 & 0 \\
0 & 1 & 0 \\
0 & 0 & a_{33}%
\end{pmatrix}%
.
\end{equation*}%
Let $\phi =\bigl(a_{ij}\bigr)\in $ ${\rm Aut}\left( \mathbf{J}_{02}\right)
$. Then $\theta \ast \phi \ =\
\big( \alpha _{1}a_{33}^{2}\Delta_{33},\ 0,\ 0\big) $. Hence we have the representative $\big( \Delta
_{33},\ 0, \ 0\big) $. Thus we get the superalgebra $\mathbf{N}_{01}$.

\item
\underline{$\mathbf{N}^{+}=\mathbf{J}_{03}$}. Let $\theta \ =\ \big(B_{1},B_{2},B_{3}\big) \neq 0$ be an arbitrary element of ${\rm Z}^{2}\left(
\mathbf{J}_{03},\mathbf{J}_{03}\right) $. Then
\begin{center}$\theta \ =\ \big( 0,\ 0,\ \alpha_{1}\Delta _{13}\big)$ for some $\alpha _{1}\in \mathbb{C}^{\ast}$.\end{center}
The
automorphism group ${\rm Aut}\left( \mathbf{J}_{03}\right) $ consists of the automorphisms $\phi $ given by a matrix of
the following form:%
\begin{equation*}
\phi =%
\begin{pmatrix}
1 & 0 & 0 \\
0 & 1 & 0 \\
0 & 0 & a_{33}%
\end{pmatrix}%
.
\end{equation*}%
Let $\phi =\bigl(a_{ij}\bigr)\in $ ${\rm Aut}\left( \mathbf{J}_{03}\right)
$. Then $\theta \ast \phi =\theta $. Hence we have the representative  $
\big( 0,\ 0,\  \alpha \Delta _{13}\big) $. Thus we get the
superalgebras $\mathbf{N}_{02}^{\alpha \neq 0}$.

\item
\underline{$\mathbf{N}^{+}=\mathbf{J}_{04}$}. Let $\theta \ =\ \big( B_{1},B_{2},B_{3}\big) \neq 0$ be an arbitrary element of ${\rm Z}^{2}\left(
\mathbf{J}_{04},\mathbf{J}_{04}\right) $. Then
\begin{center}
$\theta \ = \ \big(0,\ 0,\ \alpha _{1}\Delta _{13}-\alpha _{1}\Delta _{23}\big)$ for some $%
\alpha _{1}\in \mathbb{C}^{\ast}$.\end{center} The automorphism group ${\rm Aut}\left( \mathbf{J}_{04}\right) $ consists of the
automorphisms $\phi $ given by a matrix of the following form:%
\begin{equation*}
\phi _{1}=%
\begin{pmatrix}
1 & 0 & 0 \\
0 & 1 & 0 \\
0 & 0 & a_{33}%
\end{pmatrix}%
,\
\phi _{2}=%
\begin{pmatrix}
0 & 1 & 0 \\
1 & 0 & 0 \\
0 & 0 & a_{33}%
\end{pmatrix}%
.
\end{equation*}%
Let $\phi =\bigl(a_{ij}\bigr)\in $ ${\rm Aut}\left( \mathbf{J}_{04}\right)
$. Then $\theta \ast \phi =\theta $ if $\phi =\phi _{1}$ and $\theta \ast
\phi =-\theta $ if $\phi =\phi _{2}$. So we get the superalgebras $\mathbf{N}_{03}^{\alpha \neq 0}$.

\item
\underline{$\mathbf{N}^{+}=\mathbf{J}_{05}$}. Let $\theta \ =\ \big(B_{1},B_{2},B_{3}\big) \neq 0$ be an arbitrary element of ${\rm Z}^{2}\left(
\mathbf{J}_{05},\mathbf{J}_{05}\right) $. Then \begin{center}
$\theta\ =\ \big( 0,\ \alpha_{1}\Delta _{33},\ \alpha _{2}\Delta _{23}\big)$ for some $\alpha
_{1},\alpha _{2}\in \mathbb{C}$.\end{center} The automorphism group ${\rm Aut}\left( \mathbf{J}_{05}\right)$ consists of the automorphisms
$\phi $ given by a matrix of the following form:%
\begin{equation*}
\phi =%
\begin{pmatrix}
1 & 0 & 0 \\
0 & a_{22} & 0 \\
0 & 0 & a_{33}%
\end{pmatrix}%
.
\end{equation*}%
Let $\phi =\bigl(a_{ij}\bigr)\in $ ${\rm Aut}\left( \mathbf{J}_{05}\right)
$. Then $\theta \ast \phi \ =\ \left( 0,\ {\alpha _{1}}{a^{-1}_{22}}
a_{33}^{2}\Delta _{33},\ \alpha _{2}a_{22}\Delta _{23}\right) $. Hence we have
representatives
$\big( 0,\ \Delta _{33},\ 0\big),$
$\big( 0,\ \Delta_{33},\ \Delta _{23}\big) ,$ and
$\big( 0,\ 0,\ \Delta _{23}\big) $. Thus we get the
superalgebras $\mathbf{N}_{04},$ $\mathbf{N}_{05},$ and $\mathbf{N}_{06}$.

\item
\underline{$\mathbf{N}^{+}=\mathbf{J}_{06}$}. Let $\theta \ =\ \big(
B_{1},B_{2},B_{3}\big) \neq 0$ be an arbitrary element of ${\rm Z}^{2}\left(
\mathbf{J}_{06},\mathbf{J}_{06}\right) $. Then
\begin{center}$\theta \ =\ \big(
0, \ 0,\ \alpha _{1}\Delta _{13}\big)$ for some $\alpha _{1}\in \mathbb{C}%
^{\ast}$.\end{center} The automorphism group ${\rm Aut}\left(
\mathbf{J}_{06}\right) $ consists of the automorphisms $\phi $ given by a
matrix of the following form:%
\begin{equation*}
\phi =%
\begin{pmatrix}
1 & 0 & 0 \\
0 & a_{22} & 0 \\
0 & 0 & a_{33}%
\end{pmatrix}%
.
\end{equation*}%
Let $\phi =\bigl(a_{ij}\bigr)\in $ ${\rm Aut}\left( \mathbf{J}_{06}\right).$
Then $\theta \ast \phi =\theta $. Hence, we have the representatives $%
\big( 0,\ 0,\ \alpha \Delta _{13}\big) $. Thus, we get the superalgebras $\mathbf{N}_{07}^{\alpha \neq 0}$.

\item
\underline{$\mathbf{N}^{+}=\mathbf{J}_{07}$}. Let $\theta\  =\ \big(B_{1},B_{2},B_{3}\big) \neq 0$ be an arbitrary element of ${\rm Z}^{2}\left(
\mathbf{J}_{07},\mathbf{J}_{07}\right) $. Then
\begin{center}
    $\theta \ =\ \big( \alpha_{1}\Delta _{33},\ 0,\ 0\big)$ for some $\alpha _{1}\in \mathbb{C}^{\ast}$.
\end{center}
The automorphism group  ${\rm Aut}\left( \mathbf{J}_{07} \right) $ consists of the automorphisms $\phi $ given by a matrix of
the following form:%
\begin{equation*}
\phi =%
\begin{pmatrix}
1 & 0 & 0 \\
0 & a_{22} & 0 \\
0 & 0 & a_{33}%
\end{pmatrix}%
.
\end{equation*}%
Let $\phi =\bigl(a_{ij}\bigr)\in $ ${\rm Aut}\left( \mathbf{J}_{07}\right).$
Then $\theta \ast \phi \ =\ \big( \alpha _{1}a_{33}^{2}\Delta_{33},\ 0,\ 0\big).$
Hence we have the representative $\big( \Delta_{33},\ 0,\ 0\big) $. Thus we get the superalgebra $\mathbf{N}_{08}$.

\item
\underline{$\mathbf{N}^{+}=\mathbf{J}_{08}$}. Then ${\rm Z}^{2}\left( \mathbf{J}_{08},\mathbf{J}_{08}\right) =\left\{ 0\right\} $.

\item
\underline{$\mathbf{N}^{+}=\mathbf{J}_{09}$}. Let $\theta \ =\ \big( B_{1},B_{2},B_{3}\big) \neq 0$ be an arbitrary element of ${\rm Z}^{2}\left(
\mathbf{J}_{09},\mathbf{J}_{09}\right) $. Then
\begin{center}$\theta \ =\ \big(0,\ 0,\ \alpha _{1}\Delta _{13}\big)$ for some $\alpha _{1}\in \mathbb{C}%
^{\ast}$.\end{center}
The automorphism group ${\rm Aut}\left(
\mathbf{J}_{09}\right)$ consists of the automorphisms $\phi $ given by a
matrix of the following form:
\begin{equation*}
\phi =
\begin{pmatrix}
1 & 0 & 0 \\
0 & a_{22} & 0 \\
0 & 0 & a_{33}%
\end{pmatrix}%
.
\end{equation*}%
Let $\phi =\bigl(a_{ij}\bigr)\in $ ${\rm Aut}\left( \mathbf{J}_{09}\right)
$. Then $\theta \ast \phi =\theta $. Hence, we have representatives
$\big( 0,\ 0,\ \alpha \Delta _{13}\big) $. Thus, we get the superalgebras $%
\mathbf{N}_{09}^{\alpha \neq 0}$.

\item
\underline{$\mathbf{N}^{+}=\mathbf{J}_{10}$}. Let $\theta \ =\ \big(B_{1},B_{2},B_{3}\big) \neq 0$ be an arbitrary element of ${\rm Z}^{2}\left(
\mathbf{J}_{10},\mathbf{J}_{10}\right) $. Then
\begin{center}
$\theta\ =\ \big( 0,\ \alpha_{1}\Delta _{33},\ \alpha _{2}\Delta _{23}\big)$ for some $\alpha
_{1},\alpha _{2}\in \mathbb{C}$.\end{center}
The automorphism group ${\rm Aut}\left( \mathbf{J}_{10}\right) $ consists of the automorphisms
$\phi $ given by a matrix of the following form:%
\begin{equation*}
\phi =%
\begin{pmatrix}
1 & 0 & 0 \\
0 & a_{22} & 0 \\
0 & 0 & a_{33}%
\end{pmatrix}%
.
\end{equation*}%
Let $\phi =\bigl(a_{ij}\bigr)\in $ ${\rm Aut}\left( \mathbf{J}_{10}\right)
$. Then
$\theta \ast \phi \ =\
\left( 0,\ {\alpha _{1}}{a^{-1}_{22}} a_{33}^{2}\Delta _{33},\ \alpha _{2}a_{22}\Delta _{23}\right) $. Hence we have
the representatives
$\big( 0,\ \Delta _{33},\ 0\big),$
$\big( 0,\ \Delta_{33},\ \Delta _{23}\big),$ and
$\big( 0,\ 0,\ \Delta _{23}\big) $. Thus we get the
superalgebras $\mathbf{N}_{10},$ $\mathbf{N}_{11},$ and $\mathbf{N}_{12}$.

\item
\underline{$\mathbf{N}^{+}=\mathbf{J}_{11}$}. Let $\theta \ =\ \big(B_{1},B_{2},B_{3}\big) \neq 0$ be an arbitrary element of ${\rm Z}^{2}\left(
\mathbf{J}_{11},\mathbf{J}_{11}\right) $. Then
\begin{center}
    $\theta \ =\ \big( 0,\ \alpha_{1}\Delta _{12},\ 0\big)$ for some $\alpha _{1}\in \mathbb{C}^{\ast}$.
\end{center}
The automorphism group ${\rm Aut}\left( \mathbf{J}_{11}\right) $ consists of the automorphisms $\phi $ given by a matrix of
the following form:%
\begin{equation*}
\phi =%
\begin{pmatrix}
1 & 0 & 0 \\
a_{21} & a_{22} & 0 \\
0 & 0 & a_{33}%
\end{pmatrix}%
.
\end{equation*}%
Let $\phi =\bigl(a_{ij}\bigr)\in $ ${\rm Aut}\left( \mathbf{J}_{11}\right)
$. Then $\theta \ast \phi =\theta $. Hence, we have the representatives $\big( 0,\ \alpha \Delta _{12},\ 0\big) $. Thus, we get the superalgebras $\mathbf{N}_{13}^{\alpha \neq 0}$.

\item
\underline{$\mathbf{N}^{+}=\mathbf{J}_{12}$}. Let $\theta \ =\ \big(B_{1},B_{2},B_{3}\big) \neq 0$ be an arbitrary element of ${\rm Z}^{2}\left(
\mathbf{J}_{12},\mathbf{J}_{12}\right) $. Then
\begin{center}
    $\theta \  =\ \big( 0,\ \alpha_{1}\Delta _{12}+\alpha _{2}\Delta _{33},\ \alpha _{3}\Delta _{13}+\alpha
_{4}\Delta _{23}\big)$ for some $\alpha _{1},\alpha _{2},\alpha
_{3},\alpha _{4}\in \mathbb{C}$.
\end{center}
The automorphism group ${\rm Aut}\left( \mathbf{J}_{12}\right) $ consists of the automorphisms
$\phi $ given by a matrix of the following form:%
\begin{equation*}
\phi =%
\begin{pmatrix}
1 & 0 & 0 \\
a_{21} & a_{22} & 0 \\
0 & 0 & a_{33}%
\end{pmatrix}%
.
\end{equation*}%
Let $\phi =\bigl(a_{ij}\bigr)\in $ ${\rm Aut}\left( \mathbf{J}_{12}\right)
$. Then%
\begin{center}
$\theta \ast \phi \ =\ \left( 0, \ \alpha _{1}\Delta _{12}+{\alpha _{2}}{a^{-1}_{22}}a_{33}^{2}\Delta _{33},\
\left( \alpha _{3}+\alpha _{4}a_{21}\right) \Delta
_{13}+\alpha _{4}a_{22}\Delta _{23}\right) .$
\end{center}
Hence we have representatives \begin{center}
$\big( 0,\ \alpha \Delta _{12}+\Delta_{33},\ \Delta _{23}\big),$
$\big( 0,\ \alpha \Delta _{12},\ \Delta _{23}\big),$
$\big( 0,\ \alpha \Delta _{12}+\Delta _{33},\ \beta \Delta _{13}\big),$ and
$\big(0,\ \alpha \Delta _{12},\ \beta \Delta _{13}\big) $. \end{center}
Thus we get the
superalgebras $\mathbf{N}_{14}^{\alpha },$ $\mathbf{N}_{15}^{\alpha },$
$\mathbf{N}_{16}^{\alpha ,\beta },$ and $\mathbf{N}_{17}^{\alpha ,\beta }$.

\item \underline{$\mathbf{N}^{+}=\mathbf{J}_{13}$}. Let $\theta \ =\ \big(
B_{1},B_{2},B_{3}\big) \neq 0$ be an arbitrary element of ${\rm Z}^{2}\left(
\mathbf{J}_{13},\mathbf{J}_{13}\right) $. Then
\begin{center}$\theta \ =\ \big( 0,\ \alpha _{1}\Delta _{12},\ 0\big)$ for some $\alpha _{1}\in \mathbb{C}^{\ast}$.\end{center}
The automorphism group ${\rm Aut}\left( \mathbf{J}_{13}\right)$ consists of the automorphisms $\phi $ given by a matrix of
the following form:%
\begin{equation*}
\phi =%
\begin{pmatrix}
1 & 0 & 0 \\
a_{21} & a_{22} & 0 \\
0 & 0 & a_{33}%
\end{pmatrix}%
.
\end{equation*}%
Let $\phi =\bigl(a_{ij}\bigr)\in $ ${\rm Aut}\left( \mathbf{J}_{13}\right)
$. Then $\theta \ast \phi =\theta $. Hence, we have the representatives $\big( 0,\ \alpha \Delta _{12},\ 0\big).$
Thus, we get the superalgebras $\mathbf{N}_{18}^{\alpha \neq 0}$.

\item
\underline{$\mathbf{N}^{+}=\mathbf{J}_{14}$}. Let $\theta \ =\ \big(B_{1},B_{2},B_{3}\big) \neq 0$ be an arbitrary element of ${\rm Z}^{2}\left(
\mathbf{J}_{14},\mathbf{J}_{10}\right) $. Then
\begin{center}
    $\theta \ =\ \big( 0,\ \alpha_{1}\Delta _{33},\ \alpha _{2}\Delta _{13}\big)$ for some $\alpha_{1},\alpha _{2}\in \mathbb{C}$.
\end{center} The automorphism group ${\rm Aut}\left( \mathbf{J}_{14}\right) $ consists of the automorphisms
$\phi$ given by a matrix of the following form:
\begin{equation*}
\phi =%
\begin{pmatrix}
a_{11} & 0 & 0 \\
a_{21} & a_{11}^{2} & 0 \\
0 & 0 & a_{33}%
\end{pmatrix}%
.
\end{equation*}%
Let $\phi =\bigl(a_{ij}\bigr)\in $ ${\rm Aut}\left( \mathbf{J}_{14}\right)
$. Then $\theta \ast \phi \ =\
\left( 0,\ {\alpha _{1}}{a_{11}^{-2}}a_{33}^{2}\Delta _{33},\ \alpha _{2}a_{11}\Delta _{13}\right) $. Hence we have
 representatives
$\big( 0,\ \Delta _{33},\ \Delta _{13}\big),$
$\big(0,\ \Delta _{33},\ 0\big),$ and
$\big(0,\ 0,\ \Delta _{13}\big)$.
Thus we get the
superalgebras $\mathbf{N}_{19},$ $\mathbf{N}_{20},$ and $\mathbf{N}_{21}$.

\end{enumerate}

 \subsection{Corollaries}

 \subsubsection{Kokoris  superalgebras}\label{koko}
A flexible algebra $\rm A$ is called   {a Kokoris  algebra}
if $\rm A^+$ is associative.
Kokoris  algebras were introduced by Kokoris in 1960
and they also appear in \cite{sk91}.
  $({\rm A},\cdot) $ is a Kokoris superalgebra if
and only if $({\bf A},\circ ,[\cdot,\cdot])$ is a generic Poisson
superalgebra. Also, observe that Poisson superalgebras are related to  Lie admissible Kokoris superalgebras.

\begin{theorem}\label{K1}
Let $\bf{K}$ be a complex $3$-dimensional Kokoris superalgebra of type $(1,2)$.
Then $\bf{K}$ is an associative commutative superalgebra listed in Theorem \ref{(1,2)},
an anticommutative superalgebra listed in Theorem \ref{ant12} or
it is isomorphic to one
of the following algebras:
\begin{center}

${\rm N}_{03}^{\alpha \neq 0},$
${\rm N}_{04},$
${\rm N}_{05},$
${\rm N}_{06}^{\alpha \neq 0},$ or
${\rm N}_{09}.$
\end{center}

\end{theorem}

\begin{theorem}\label{K2}
Let $\bf{K}$ be a complex $3$-dimensional Kokoris superalgebra of type $(2,1)$.
Then $\bf{K}$ is an associative commutative superalgebra listed in Theorem \ref{(2,1)},
an anticommutative superalgebra listed in Theorem \ref{ant21} or
it is isomorphic to one
of the following algebras:

\begin{center}

${\bf N}_{01},$
${\bf N}_{04},$
${\bf N}_{05},$
${\bf N}_{06},$
${\bf N}_{08},$
${\bf N}_{10},$
${\bf N}_{11},$
${\bf N}_{12},$
${\bf N}_{19},$
${\bf N}_{20},$ or
${\bf N}_{21}.$

\end{center}

\end{theorem}

\subsubsection{Associative    superalgebras}

\begin{theorem}\label{ass1}
Let $\rm{A}$ be a complex $3$-dimensional associative superalgebra of type $(1,2)$.
Then $\rm{A}$ is   one
of the following algebras:

\begin{center}
${\rm J}_{02},$
${\rm J}_{03},$
${\rm J}_{06},$
${\rm J}_{07},$
${\rm J}_{09},$
${\rm A}_{01},$
${\rm A}_{02},$
${\rm A}_{06},$
${\rm N}_{03}^{\alpha \neq0},$
${\rm N}_{04},$
${\rm N}_{06}^{\alpha \neq0},$ \\
${\rm N}_{07}^{\frac 12},$
${\rm N}_{07}^{-\frac 12},$
${\rm N}_{08}^{\frac 12},$
${\rm N}_{08}^{-\frac 12},$
${\rm N}_{09},$
${\rm N}_{10}^{\frac 12,\frac 12},$
${\rm N}_{10}^{\frac 12,-\frac 12},$ or
${\rm N}_{10}^{-\frac 12,-\frac 12}.$

\end{center}
\end{theorem}

\begin{theorem}\label{ass2}
Let $\bf{A}$ be a complex $3$-dimensional associative superalgebra of type $(2,1)$.
Then $\bf{A}$ is   one
of the following algebras:

\begin{center}
${\bf J}_{01},$
${\bf J}_{02},$
${\bf J}_{05},$
${\bf J}_{07},$
${\bf J}_{08},$
${\bf J}_{10},$
${\bf J}_{14},$
${\bf A}_{10},$
${\bf N}_{01},$
${\bf N}_{02}^{\frac 12},$
${\bf N}_{02}^{-\frac 12},$
${\bf N}_{03}^{\frac 12},$
${\bf N}_{04},$
${\bf N}_{07}^{\frac 12},$
${\bf N}_{07}^{-\frac 12},$
${\bf N}_{08},$ \\
${\bf N}_{09}^{\frac 12},$
${\bf N}_{09}^{-\frac 12},$
${\bf N}_{10},$
${\bf N}_{13}^{\frac 12},$
${\bf N}_{13}^{-\frac 12},$
${\bf N}_{17}^{\frac 12,\frac 12},$
${\bf N}_{17}^{-\frac 12,\frac 12},$
${\bf N}_{17}^{\frac 12,-\frac 12},$
${\bf N}_{17}^{-\frac 12,-\frac 12},$
${\bf N}_{18}^{\frac 12},$
${\bf N}_{18}^{-\frac 12},$
or ${\bf N}_{20}.$

\end{center}

\end{theorem}

\subsubsection{Standard    superalgebras}
The notions of standard   algebras were introduced in \cite{Alb}.
An algebra is defined to be  {standard} in case the following two
identities hold:
\begin{eqnarray*}
(x,y,z)+(z,x,y)-(x,z,y) &=&0, \\
(x,y,wz)+(w,y,xz)+(z,y,wx) &=&0.
\end{eqnarray*}%
Standard algebras include all associative algebras and   Jordan
algebras.
It is proved that the variety of standard algebras is just the minimal variety containing the variety of associative algebras and the variety of Jordan algebras   \cite{hac18}.
Moreover, every standard algebra is a noncommutative Jordan algebra and is therefore power-associative.
By some  direct verification of standard    identities in noncommutative Jordan superalgebras,  we have the following statements.

\begin{definition}
    A superalgebra is defined to be  {standard} in case the following two
superidentities hold:
\begin{eqnarray*}
(x,y,z)+ (-1)^{ |z|\big(|x|+|y|\big)}(z,x,y)-(-1)^{|z||y|}(x,z,y) &=&0, \\
(x,y,wz)+(-1)^{|w|\big(|x|+|y|\big)+|x||y|}(w,y,xz)+(-1)^{|x|\big(|y|+|z|+w|\big)+|z|\big(|y|+|w|\big)}(z,y,wx) &=&0.
\end{eqnarray*}%
\end{definition}

\begin{theorem}\label{S1}
Let $\bf{S}$ be a complex $3$-dimensional standard superalgebra of type $(1,2)$.
Then $\bf{S}$ is an associative superalgebra listed in  Theorem \ref{ass1} or
a non-associative Jordan  superalgebra listed in Theorem \ref{(1,2)}.

\end{theorem}

\begin{theorem}\label{S2}
Let $\bf{S}$ be a complex $3$-dimensional standard superalgebra of type $(2,1)$.
Then $\bf{S}$ is an associative superalgebra listed in  Theorem \ref{ass2},
a non-associative Jordan  superalgebra listed in Theorem \ref{(2,1)},  or  isomorphic to  ${\bf N}_{17}^{0, \pm \frac 12}$
or
${\bf N}_{17}^{\pm \frac 12,0}.$

\end{theorem}

 \subsubsection{Anticommutative superalgebras}\label{anticom}

The present subsection collects the algebraic classifications of some  interesting subclasses of anticommutative superalgebras.
Some of the classifications are known\footnote{Each Lie superalgebra is a Malcev superalgebra, and the classification of Malcev superalgebras is known from \cite{AE96}.
The classification of Tortkara superalgebras was given in \cite{BM25}.}, and others are new, but we present all of them to give the complete picture of the differences between these classes of superalgebras.

\begin{definition}\label{def}
    Let $({\rm A}, [\cdot, \cdot])$ be an anticommutative superalgebra, then

\begin{enumerate}
    \item[{\rm (1)}]
 ${\rm A}$ is a Lie superalgebra, if
\begin{center}
    $\big[[x,y],z\big] \ = \  \big[x,[y,z]\big] + \big(-1\big)^{|y||z|} \big[[x,z],y\big].$
\end{center}

\item[{\rm (2)}]  ${\rm A}$ is a Malcev superalgebra $\big($see, {\rm \cite{AE96}}$\big)$, if
\begin{flushleft}
$\Big[\big[[x,y],z\big],t\Big]-
\Big[x,\big[[y,z],t\big]\Big]+\big(-1\big)^{|x||y|}\Big[y,\big[x, [z,t]\big]\Big]=$
\end{flushleft}

\begin{flushright}
$
  \big(-1\big)^{|t|\big(|y|+|z|\big)}\Big[\big[[x,t],y\big],z\Big] +  \big(-1\big)^{|y||z|}\big[[x,z],[y,t]\big].$\end{flushright}

\item[{\rm (3)}]  ${\rm A}$ is a binary Lie superalgebra $\big($see, {\rm \cite{GRS}}$\big)$, if

\begin{flushleft}
$\Big[\big[[x,y],z\big],t\Big]+(-1)^{|x||y|}\Big(\Big[y, \big[ [x,z], t\big]\Big] + \Big[y, \big[ x, [z,t]\big]\Big] - \Big[\big[ y, [x,z]\big],t\Big]\Big) +
$
\end{flushleft}

\begin{flushright}
$(-1)^{|z||t|}\Big(\Big[x, \big[ [y,t], z\big]\Big]  - \Big[\big[[x,y],t\big],z\Big] - \Big[\big[x, [y,t]\big],z\Big]  \Big)=  \Big[x, \big[ y, [z,t]\big]\Big].$\end{flushright}

\item[{\rm (4)}]  ${\rm A}$ is a Tortkara superalgebra $\big($see, {\rm \cite{BM25}}$\big)$, if

\begin{flushleft}
$\Big[ \big[[x,y],z\big]  -  \big [x,[y, z]\big] - \big(-1\big)^{ |y||z|} \big[[x, z], y\big], t\Big]-  \big[[x, y],   [z, t] \big]\ =$
\end{flushleft}

\begin{flushright}
$(-1)^{|x||y|} \Big(\Big[y, \big[x,[z, t] \big]-  \big [[x,z], t\big] -\big(-1 \big)^{ |x||z|} \big[z,[x, t]\big] \Big] +  \big(-1\big)^{   |x||z|} \big[ [y, z],[x, t]\big] \Big). $
\end{flushright}

\item[{\rm (5)}]  ${\rm A}$  is an ${\mathfrak a}\mathfrak{CD}$-superalgebra $\big($see, {\rm \cite{KZ}}$\big)$, if

\begin{flushleft}
$\Big[\big[[x, y], z\big],  t\Big] + \big(-1\big)^{|x||y| } \Big[  y, \big[[x, z], t\big]\Big] -   \Big[x,\big[[y, z], t\big]\Big] =$
\end{flushleft}

\begin{flushright}
$\big(-1\big)^{|z||t|} \Big( \Big[\big[[x, y], t\big], z\Big]
+\big(-1\big)^{|x||y|} \Big[y,\big[[x, t], z\big]\Big]-   \Big[x,\big[[y, t], z\big]\Big] \Big).$\end{flushright}

\item[{\rm (6)}]  ${\rm A}$ is an $\mathfrak{s}_4$-superalgebra $\big($see, {\rm \cite{F01,D09}}$\big)$, if

\begin{flushleft}
$\Big[\big[[x, y], z\big]-\big[x,[y, z]\big]  - \big(-1\big)^{|y||z|} \big[[x,z],y\big], t\Big]    +$\end{flushleft}
\begin{flushleft}\quad \quad \quad \quad \quad $\Big[x,\big[[y, z], t\big]-\big[y,[z,t]\big]  - (-1)^{|z||t|} \big[[y, t],z\big]\Big]
+$\end{flushleft}

\begin{center}$\big(-1\big)^{   |x||y|   }  \Big[y,\big[x,[z, t]\big]- \big[[x, z], t\big]      + \big(-1\big)^{   |z||t|} \big[[x, t],z\big]\Big]
  +$\end{center}
\begin{flushright}
$   \big(-1\big)^{|z||t|}     \Big[\big[x,[y, t]\big] -   \big[[x, y], t\big]
+ \big(-1\big)^{|y||t| } \Big[\big[[x, t], y\big], z\Big]
 \ =\ 0.$
\end{flushright}

\item[{\rm (7)}]  ${\rm A}$ is a   rigid\footnote{Kac and Cantarini \cite{CK} used the term "rigid" algebras, but early, Kantor in 1989 used the term "quasi-conservative" for these algebras \cite{mr}. "Rigid" superalgebras by Kac-Cantarini are different from rigid superalgebras, which will appear in our geometric classification. To avoid a misunderstanding, we are using only the Kantor notation. } (quasi-conservative)  superalgebra $\big($see, {\rm \cite{CK}}$\big)$, if
there is a new multiplication $*$ and a bilinear form $\varphi: {\rm A} \times {\rm A} \to \mathbb C,$ such that

\begin{flushleft}
$\Big[t, \big[z, [x,y]\big]- \big[ [z,x], y\big]\Big]+\big(-1\big)^{|z||x|}\Big( \big[[t,x], [z,y]\big]-   \Big[t, \big[ x, [z, y]\big]\Big] \Big)+$\end{flushleft}

\begin{center}$
\big(-1\big)^{|z||t|}\Big(\Big[\big[z,[t,x]\big],y\Big]- \Big[z, \big[[t,x],y\big]\Big]-\big(-1\big)^{|x||y| } \Big[ \big[z, [t,y]\big],x\Big]\Big)+$\end{center}

\begin{flushright}
$\big(-1\big)^{|t|\big(|z|+|x|\big)} \Big(\big[[z,x],[t,y]\big] -\Big[z, \big[x, [t,y]\big]\Big] \Big)\ =\ $ \end{flushright}
\begin{flushright}
$
\big(-1\big)^{|z||t|} \Big(\big[[z*t,x],y\big] -
\big(-1\big)^{|x||y|}\big[  [z*t,y],x\big]
- \big[z*t, [x,y]\big]   + \varphi(z,t) [x,y] \Big).$
\end{flushright}

\begin{enumerate}

\item[{\rm (7.1)}] If $\varphi(z,t)=0$ and $z*t = \frac 1 3 [x,y]$ for all $z,t \in {\rm A}$ then the superalgebra ${\rm A}$ is called ${\mathfrak a}$-terminal $\big($see, {\rm \cite{p20}}$\big)$.

\item[{\rm (7.2)}] If $\varphi(z,t)=0$ for all $z,t \in {\rm A},$ then the superalgebra ${\rm A}$ is called conservative $\big($see, {\rm \cite{p20}}$\big)$.

\end{enumerate}

\end{enumerate}
\end{definition}

\begin{theorem}
Let ${\rm A}$ be a nontrivial $3$-dimensional anticommutative
superalgebra of type $\left( 1,2\right) $. Then

\begin{longtable}{l |c |c| c| c| c| c| c| c| c|c|c|}
    \hline
     ${\rm A}$& {\rm (1)} & {\rm (2)} &  {\rm (3)} &
 {\rm (4)} & {\rm (5)} & {\rm (6)}  & {\rm (7.1)} & {\rm (7.2)} & {\rm (7)} \\
 \hline
  ${\rm A}_{01}$& \cmark & \cmark& \cmark& \cmark& \cmark& \cmark& \cmark& \cmark & \cmark\\\hline

 ${\rm A}_{02}$& \cmark & \cmark& \cmark& \cmark& \cmark& \cmark& \cmark& \cmark & \cmark\\\hline

 ${\rm A}_{03}$& \xmark & \xmark & \xmark & \xmark & \xmark & \xmark & \xmark& \xmark & \xmark \\\hline

 ${\rm A}_{04}$& \xmark & \xmark & \xmark & \xmark & \xmark & \xmark & \xmark& \xmark & \xmark\\\hline

  ${\rm A}_{05}$& \xmark & \cmark & \cmark & \xmark & \xmark & \xmark & \xmark  & \cmark &  \cmark \\\hline

${\rm A}_{06}$& \cmark & \cmark& \cmark& \cmark& \cmark& \cmark& \cmark& \cmark & \cmark\\\hline

${\rm A}_{07}$& \xmark & \cmark & \cmark & \cmark & \cmark & \cmark & \cmark & \cmark & \cmark\\\hline

 ${\rm A}_{08}^{\alpha}$& \cmark & \cmark& \cmark& \cmark& \cmark& \cmark& \cmark& \cmark & \cmark \\\hline

 ${\rm A}_{09}^{\alpha}$& \xmark & \xmark & \xmark & \xmark & \xmark& \xmark & \xmark& \xmark & \xmark\\\hline

${\rm A}_{10}^{\alpha}$& \xmark & \xmark & \cmark$_{\alpha=1}$  & \xmark & \xmark & \xmark& \xmark& \xmark & \xmark\\
\hline
${\rm A}_{11}^{\alpha \neq 1,\beta}$& \xmark & \xmark & \xmark  & \xmark & \xmark & \xmark& \xmark& \xmark& \xmark\\ \hline

 ${\rm A}_{12}^{\alpha \neq 1}$& \xmark & \xmark & \xmark  & \xmark & \xmark & \xmark& \xmark& \xmark& \xmark\\\hline

  ${\rm A}_{13}$& \xmark & \xmark & \cmark & \xmark & \xmark & \xmark& \xmark & \xmark & \xmark\\\hline

${\rm A}_{14}$& \xmark & \xmark & \xmark  & \xmark & \xmark & \xmark& \xmark  & \xmark & \xmark \\\hline

 ${\rm A}_{15}^{\alpha }$& \xmark & \xmark & \xmark  & \xmark & \xmark & \xmark& \xmark    & \xmark & \xmark\\\hline

 ${\rm A}_{16}$& \xmark & \xmark & \xmark  & \xmark & \xmark & \xmark& \xmark  & \xmark & \xmark \\\hline

 ${\rm A}_{17}$& \cmark & \cmark& \cmark& \cmark& \cmark& \cmark& \cmark & \cmark & \cmark\\\hline

${\rm A}_{18}$& \xmark & \xmark & \cmark & \xmark & \xmark & \xmark& \xmark & \xmark & \xmark\\
    \hline
\end{longtable}

\end{theorem}

\begin{proof}
Checking identities for superalgebras  $(1)$ -- $(6)$ and $(7.1)$ is a routine process, so it will be omitted. The case of conservative and quasi-conservative superalgebras will be given in more detail (we consider only the non-terminal case).

\begin{enumerate}
    \item[${\rm A}_{03}$:]
    Taking $z=f_1, t=e_1, x=f_2,$ and $y=f_2$ in the quasi-conservative relation from (7) of Definition \ref{def},  we find a contradiction with the existence of $*$ and $\varphi.$

  \item[${\rm A}_{04}$:]
    Taking $z=e_1, t=e_1, x=f_2,$ and $y=e_1$ in the quasi-conservative relation from (7) of Definition \ref{def},  we find
    $\varphi(e_1,e_1)=0,$
    but if we take  $z=e_1, t=e_1, x=f_2,$ and $y=f_1,$ then   $\varphi(e_1,e_1)=2.$
That gives a contradiction with the existence of $*$ and $\varphi.$

  \item[${\rm A}_{05}$:] This superalgebra
  is conservative due to the additional multiplication given by
\begin{center}
      $f_1* f_2=e_1,$
  $e_1*f_2=f_1,$
  $f_2*f_1=-2e_1,$
  $f_2*e_1=2f_1.$
\end{center}

  \item[${\rm A}_{09}^\alpha$:]
    Taking $z=f_1, t=f_1, x=f_2,$ and $y=f_2$ in the quasi-conservative relation from (7) of Definition \ref{def},  we find  a contradiction with the existence of $*$ and $\varphi.$

  \item[${\rm A}_{10}^\alpha$:]
    Taking $z=f_2, t=f_1, x=e_1,$ and $y=f_2$ in the quasi-conservative relation from (7) of Definition \ref{def},  we find  $\alpha=\frac 12,$ but for
    $z=f_1, t=f_2, x=f_2,$ and $y=e_1$ we have  $\alpha=0.$
    That gives a contradiction with the existence of $*$ and $\varphi.$

  \item[${\rm A}_{11}^{\alpha, \beta}$:]
    Taking $z=f_2, t=f_1, x=f_2,$ and $y=e_1$ in the quasi-conservative relation from (7) of Definition \ref{def},  we find $\alpha=\frac 12$,
    but if we take  $z=f_1, t=f_2, x=e_1,$ and $y=f_2$  then   $\alpha=0.$
That gives a contradiction with the existence of $*$ and $\varphi.$

  \item[${\rm A}_{12}^{\alpha}$:]
    Taking $z=f_1, t=f_1, x=e_1,$ and $y=f_2$ in the quasi-conservative relation from (7) of Definition \ref{def},  we find  $\alpha = 1,$ but if we take
    $z=f_2, t=f_2, x=f_1,$ and $y=f_1$  then   $\alpha=0.$
That gives a contradiction with the existence of $*$ and $\varphi.$

  \item[${\rm A}_{13}$:]
    Taking $z=e_1, t=f_1, x=f_1,$ and $y=f_1$ in the quasi-conservative relation from (7) of Definition \ref{def},  we find
    $e_1 * f_1 = \frac 23 f_1 + \lambda f_2$ for some $\lambda \in \mathbb C.$ But if we take
   $z=e_1, t=f_1, x=f_1,$ and $y=e_1,$   then  $e_1 * f_1 = f_1 + \mu f_2$  for some $\mu \in \mathbb C.$
That gives a contradiction with the existence of $*$ and $\varphi.$

  \item[${\rm A}_{14}$:]
    Taking $z=f_1, t=f_1, x=f_2,$ and $y=f_2$ in the quasi-conservative relation from (7) of Definition \ref{def},  we find a contradiction with the existence of $*$ and $\varphi.$

  \item[${\rm A}_{15}^\alpha$:]
    Taking $z=f_1, t=f_2, x=f_1,$ and $y=e_1$ in the quasi-conservative relation from (7) of Definition \ref{def},  we find a contradiction with the existence of $*$ and $\varphi.$

  \item[${\rm A}_{16}$:]
    Taking $z=e_1, t=f_1, x=e_1,$ and $y=e_1$ in the quasi-conservative relation from (7) of Definition \ref{def},  we find a contradiction with the existence of $*$ and $\varphi.$

      \item[${\rm A}_{18}$:]
    Taking $z=f_2, t=f_1, x=f_2,$ and $y=e_1$ in the quasi-conservative relation from (7) of Definition \ref{def},  we find that $\varphi(f_2,f_1)=0$ and $\varphi(f_2,f_1)=1.$
    That gives a contradiction with the existence of $*$ and $\varphi.$

\end{enumerate}
\end{proof}

\begin{theorem}
Let ${\rm A}$ be a nontrivial $3$-dimensional anticommutative
superalgebra of type $\left(2,1\right) $. Then

\begin{longtable}{l |c |c| c| c| c| c| c| c| c|c|c|}
    \hline
     ${\rm A}$& {\rm (1)} & {\rm (2)} &  {\rm (3)} &
 {\rm (4)} & {\rm (5)} & {\rm (6)}  & {\rm (7.1)} & {\rm (7.2)} & {\rm (7)} \\
 \hline

 $\mathbf{A}_{01}^{\alpha}$ & \xmark & \xmark & \cmark$_{\alpha=-1}$  & \xmark & \xmark & \xmark& \xmark & \xmark& \xmark \\

\hline

$\mathbf{A}_{02}^{\alpha}$& \xmark & \xmark & \xmark  & \xmark & \xmark & \xmark& \xmark& \xmark& \xmark  \\\hline

$\mathbf{A}_{03}^{\alpha}$&
\cmark$_{\alpha=-\frac 12}$ & \cmark$_{\alpha \in \big\{ -\frac 12, 1\big\}}$& \cmark& \cmark$_{\alpha=-1}$ & \cmark$_{\alpha=- \frac 12}$ & \cmark& \cmark$_{\alpha=- \frac 12}$ & \cmark & \cmark\\\hline

$\mathbf{A}_{04}$& \xmark & \xmark & \xmark  & \xmark & \xmark & \xmark& \xmark & \xmark& \xmark \\\hline

$\mathbf{A}_{05}^{\alpha}$& \cmark & \cmark& \cmark& \cmark& \cmark& \cmark& \cmark & \cmark & \cmark\\\hline

$\mathbf{A}_{06}$& \xmark & \xmark & \xmark  & \cmark & \xmark & \cmark & \xmark& \cmark & \cmark\\\hline

$\mathbf{A}_{07}$& \cmark & \cmark& \cmark& \cmark& \cmark& \cmark& \cmark & \cmark & \cmark\\\hline

$\mathbf{A}_{08}$& \xmark & \xmark & \cmark  & \xmark & \cmark & \cmark & \xmark& \cmark& \cmark\\\hline

$\mathbf{A}_{09}$& \xmark & \xmark& \cmark& \xmark & \xmark & \xmark & \xmark& \xmark & \xmark\\\hline

$\mathbf{A}_{10}$& \cmark & \cmark& \cmark& \cmark& \cmark& \cmark& \cmark& \cmark & \cmark\\
\hline

\end{longtable}

\end{theorem}

\begin{proof}
Checking identities for superalgebras  $(1)$ -- $(6)$ and $(7.1)$ is a routine process, so it will be omitted. The case of conservative and quasi-conservative superalgebras will be given in more detail.

\begin{enumerate}
      \item[${\mathbf A}_{01}^\alpha$:]
    Taking $z=e_1, t=f_1, x=f_1,$ and $y=e_2$ in the quasi-conservative relation from (7) of Definition \ref{def},  we find that $\alpha=0,$ but for
    $z=f_1, t=e_1, x=e_2,$ and $y=f_1,$ we have that $\alpha=-\frac 12.$
    That gives a contradiction with the existence of $*$ and $\varphi.$

      \item[${\mathbf A}_{02}^\alpha$:]
    Taking $z=f_1, t=f_1, x=e_1,$ and $y=e_2$ in the quasi-conservative relation from (7) of Definition \ref{def},  we find   a contradiction with the existence of $*$ and $\varphi.$

      \item[${\mathbf A}_{03}^\alpha$:] This superalgebra is conservative due to the additional multiplication given by
      \begin{center}
          $f_1 * e_2 =(1 +\alpha) f_1,$
          $e_2 * f_1 = \alpha f_1,$
          $e_2*e_2 = (1+2\alpha) e_2.$
      \end{center}

      \item[${\mathbf A}_{04}$:]
    Taking $z=f_1,$ $t=e_1,$ $x=f_1,$ and $y=f_1$ in the quasi-conservative relation from (7) of Definition \ref{def},  we find that $e_1 * f_1 = f_1 + \lambda f_2$ for some $\lambda \in \mathbb C.$ But for
    $z=f_1,$ $t=e_1,$ $x=e_1,$ and $y=e_2,$ we have that $f_1 * e_1 =-f_1 + \mu f_2$ for some $\mu \in \mathbb C.$
    That gives a contradiction with the existence of $*$ and $\varphi.$

\item[${\mathbf A}_{06}$:] This superalgebra is conservative due to the additional multiplication given by
      \begin{center}
          $f_1 * e_1 =- f_1,$
          $f_1 * e_2 = - f_1,$
          $e_1*f_1 =   f_1,$
 $e_2*e_1 =   -e_1,$
 $e_2*e_2 =   -e_2.$

      \end{center}

\item[${\mathbf A}_{08}$:] This superalgebra is conservative due to the additional multiplication given by
      \begin{center}
          $f_1 * e_1 = f_1,$
          $e_1 * f_1 =  f_1,$
          $e_1*e_1 =   2e_1.$

      \end{center}
     \item[${\mathbf A}_{09}$:]
    Taking $z=e_1,$ $t=f_1,$ $x=f_1,$ and $y=f_1$ in the quasi-conservative relation from (7) of Definition \ref{def},  we find that $e_1 * f_1 = \frac 23 f_1 + \lambda f_2$ for some $\lambda \in \mathbb C.$
    But for
    $z=e_1,$ $t=f_1,$ $x=f_1,$ and $y=e_1,$ we have that $f_1 * e_1 =f_1 + \mu f_2$ for some $\mu \in \mathbb C.$
    That gives a contradiction with the existence of $*$ and $\varphi.$

\end{enumerate}

\end{proof}

\begin{corollary}\label{conj}
    The conjecture of Grishkov-Shestakov\footnote{\cite[Conjecture 1.1]{GRS}:
   Let $ \rm B$   be a nontrivial complex  finite-dimensional simple binary Lie superalgebra. Then $\rm  B$ is a simple Lie superalgebra or ${\rm dim} \ \rm  B = 2$.} is true for $3$-dimensional binary Lie superalgebras.
\end{corollary}

\begin{theorem}
Let ${\rm A}$ be a nontrivial $3$-dimensional anticommutative
superalgebra of type $\left(3,0\right) $. Then

\begin{longtable}{l |lll |c |c| c| c| c| c| c| c| c|c|c|}
    \hline
     ${\rm A}$& ${\rm product}\footnote{The classification of $3$-dimensional anticommutative algebras is given in \cite{ikv}.} $ &&&{\rm (1,2,3,5,7)} &
 {\rm (4)} &   {\rm (6)}    \\
 \hline

$\mathfrak{g}_{1}$ &
$[e_2,e_3] =e_1$ && & \cmark  & \cmark &\cmark   \\
\hline

$\mathfrak{g}_2$
&  $[e_1,e_3] =e_1$ &  $[e_2,e_3]=e_2$  &&  \cmark &\cmark &\cmark   \\
\hline

$\mathfrak{g}^{\alpha}_3$&
$[e_1,e_3] =e_1+e_2$ & $[e_2,e_3]=\alpha e_2$& & \cmark &\cmark &\cmark   \\
\hline

$\mathfrak{g}_4$&
$[e_1,e_2] =e_3$&$ [e_1,e_3]=-e_2$&$  [e_2,e_3]=e_1$&  \cmark &\xmark &\cmark  \\
\hline

$\mathcal{A}_1^{\alpha}$&
 $[e_1,e_2] =e_1+e_2$& $[e_1,e_3]=e_2$&$    [e_2,e_3]=\alpha e_3$ & \xmark & \cmark$_{\alpha=0}$ &  \cmark  \\
\hline

$\mathcal{A}_2$&
$[e_1,e_2]=e_1$&$  [e_2,e_3]=e_2$ &&   \xmark & \cmark  &   \cmark  \\
\hline

$\mathcal{A}_3$&
$[e_1,e_2]=e_3$&$ [e_1,e_3]=e_1$&$ [e_2,e_3]=e_2$ & \xmark & \xmark    &  \cmark  \\
\hline
\end{longtable}
\end{theorem}

\begin{proof}
Checking identities for superalgebras  $(1)$ -- $(6)$ is a routine process so it will be omitted. The case of conservative and quasi-conservative superalgebras will be given in more detail.
As each algebra $\mathfrak{g}_i$ is Lie, it is conservative and quasi-conservative.

\begin{enumerate}
    \item[$\mathcal{A}_1^{\alpha}$:]
    Taking $z=e_1, t=e_1, x=e_2,$ and $y=e_3$ in the quasi-conservative relation from (7) of Definition \ref{def}, we find the contradiction with existence of $*$ and $\varphi.$

    \item[$\mathcal{A}_2$:]
    Taking $z=e_1, t=e_2, x=e_2,$ and $y=e_2$ in the quasi-conservative relation from (7) of Definition \ref{def},  we find that $e_1*e_2 \in \langle e_2 \rangle,$ but if we take
    $z=e_1, t=e_2, x=e_2,$ and $y=e_3,$ we find that $e_1*e_2+e_1 \in \langle e_2 \rangle,$
    the contradiction with the existence of $*$ and $\varphi.$

    \item[$\mathcal{A}_3$:]
    Taking $z=e_1, t=e_1, x=e_2,$ and $y=e_3$ in the quasi-conservative relation from (7) of Definition \ref{def}, we find the contradiction with the existence of $*$ and $\varphi.$

\end{enumerate}

\end{proof}

\section{The geometric classification of superalgebras}

\subsection{Preliminaries: degenerations and geometric classification}

Let \( V = V_0 \oplus V_1 \) be a \( \mathbb{Z}_2 \)-graded vector space with a fixed homogeneous basis
\( \big\{e_1, \ldots, e_m, f_1, \ldots, f_n \big\}\).
A  superalgebra structure on \(V\) can be described via structure constants
\((\alpha_{ij}^k, \beta_{ij}^k, \gamma_{ij}^k, \delta_{ij}^k) \in \mathbb{C}^{m^3+3mn^2}\), where the multiplication is defined as:
\[
e_i e_j = \sum_{k=1}^{m} \alpha_{ij}^k e_k, \quad
e_i f_j = \sum_{k=1}^{n} \beta_{ij}^k f_k, \quad
f_i e_j = \sum_{k=1}^{n} \gamma_{ij}^k f_k, \quad
f_i f_j = \sum_{k=1}^{m} \delta_{ij}^k e_k.
\]
Let $\mathcal{S}^{m,n}$ denote the set of all superalgebras of dimension $(m,n)$ defined by a family of polynomial superidentities $T$, regarded as a subset $\mathbb{L}(T)$ of the affine variety $\operatorname{Hom}(V \otimes V, V)$. Then $\mathcal{S}^{m,n}$ is a Zariski-closed subset of the variety $\operatorname{Hom}(V \otimes V, V)$.

The group $G = (\operatorname{Aut} V)_0 \simeq \operatorname{GL}(V_0) \oplus \operatorname{GL}(V_1)$ acts on $\mathcal{S}^{m,n}$ by conjugation:
\[
(g * \mu)(x \otimes y) = g \mu(g^{-1} x \otimes g^{-1} y),
\]
for all $x, y \in V$, $\mu \in \mathbb{L}(T)$, and $g \in G$.

Let $\mathcal{O}(\mu)$ denote the orbit of $\mu \in \mathbb{L}(T)$ under the action of $G$, and let $\overline{\mathcal{O}(\mu)}$ be the Zariski closure of $\mathcal{O}(\mu)$. Suppose $J, J' \in \mathcal{S}^{m,n}$ are represented by $\lambda, \mu \in \mathbb{L}(T)$, respectively. We say that $\lambda$ degenerates to $\mu$, denoted $\lambda \to \mu$, if $\mu \in \overline{\mathcal{O}(\lambda)}$. In this case, we have $\overline{\mathcal{O}(\mu)} \subset \overline{\mathcal{O}(\lambda)}$. Therefore, the notion of degeneration does not depend on the particular representatives, and we write $J \to J'$ instead of $\lambda \to \mu$, and $\mathcal{O}(J)$ instead of $\mathcal{O}(\lambda)$.
We write $J \not\to J'$ to indicate that $J' \notin \overline{\mathcal{O}(J)}$.

If $J$ is represented by $\lambda \in \mathbb{L}(T)$, we say that $J$ is \textit{rigid} in $\mathbb{L}(T)$ if $\mathcal{O}(\lambda)$ is an open subset of $\mathbb{L}(T)$. A subset of a variety is called \textit{irreducible} if it cannot be written as a union of two proper closed subsets. A maximal irreducible closed subset is called an \textit{irreducible component}. In particular, $J$ is rigid in $\mathcal{S}^{m,n}$ if and only if $\overline{\mathcal{O}(\lambda)}$ is an irreducible component of $\mathbb{L}(T)$. It is a well-known fact that every affine variety admits a unique decomposition into finitely many irreducible components.

To find degenerations, we use the standard methods, described in \cite{GRH,als,ahk,BC99} and so on.
To prove a non-degeneration $J \not\to J',$ we use the standard argument from Lemma, whose proof is the same as the proof of   \cite[Lemma 1.5]{GRH}.

\begin{lemma}\label{gmain}
Let $\mathfrak{B}$ be a Borel subgroup of ${\rm GL}(\mathbb V)$ and ${\rm R}\subset \mathbb{L}(T)$ be a $\mathfrak{B}$-stable closed subset.
If $J  \to J'$ and  the superalgebra $J $ can be represented by a structure $\mu\in{\rm R}$, then there is $\lambda\in {\rm R}$ representing $J'$.
\end{lemma}

To prove our results, we will need the following statement.

\begin{theorem}[see, \cite{ahk}]\label{asscomjord}
The following statements are true:

\begin{enumerate}
    \item[$({\rm 1})$]
    The variety of associative commutative  superalgebras of type $(1,2)$
    has dimension $3,$
    four rigid superalgebras and
    four irreducible components given by
$\mathcal{C}_1=\overline{\mathcal{O}({\rm J}_{02})},$
$\mathcal{C}_2=\overline{\mathcal{O}({\rm J}_{06})},$
$\mathcal{C}_3=\overline{\mathcal{O}({\rm J}_{07})},$ and
$\mathcal{C}_4=\overline{\mathcal{O}({\rm J}_{09})}.$

    \item[$({\rm 2})$]
    The variety of associative commutative  superalgebras of type $(2,1)$
    has dimension $4,$
    two  rigid superalgebras and
    two irreducible components given by
$\mathcal{C}_1=\overline{\mathcal{O}({\bf J}_{01})}$  and
$\mathcal{C}_2=\overline{\mathcal{O}({\bf J}_{02})}.$

    \item[$({\rm 3})$]
    The variety of Jordan  superalgebras of type $(1,2)$
    has dimension $3,$
    seven rigid superalgebras and
    seven irreducible components given by
$\mathcal{C}_1=\overline{\mathcal{O}({\rm J}_{01})},$
$\mathcal{C}_2=\overline{\mathcal{O}({\rm J}_{04})},$
$\mathcal{C}_3=\overline{\mathcal{O}({\rm J}_{05})},$
$\mathcal{C}_4=\overline{\mathcal{O}({\rm J}_{06})},$
$\mathcal{C}_5=\overline{\mathcal{O}({\rm J}_{07})},$
$\mathcal{C}_6=\overline{\mathcal{O}({\rm J}_{10})},$ and
$\mathcal{C}_7=\overline{\mathcal{O}({\rm J}_{11})}.$

   \item[$({\rm 2})$]
    The variety of Jordan   superalgebras of type $(2,1)$
    has dimension $4,$
    seven   rigid superalgebras and
    seven irreducible components given by
$\mathcal{C}_1=\overline{\mathcal{O}({\bf J}_{01})},$
$\mathcal{C}_2=\overline{\mathcal{O}({\bf J}_{02})},$
$\mathcal{C}_3=\overline{\mathcal{O}({\bf J}_{03})},$
$\mathcal{C}_4=\overline{\mathcal{O}({\bf J}_{04})},$
$\mathcal{C}_5=\overline{\mathcal{O}({\bf J}_{08})},$
$\mathcal{C}_6=\overline{\mathcal{O}({\bf J}_{12})},$ and
$\mathcal{C}_7=\overline{\mathcal{O}({\bf J}_{13})}.$

\end{enumerate}

\end{theorem}

\subsection{Anticommutative superalgebras}

\begin{theorem}\label{geoant12}
For subvarieties of anticommutative superalgebras of type $(1,2),$ the following statements are true:

\begin{enumerate}
    \item[$({\rm 1})$]
    The variety of Lie superalgebras has dimension $4,$
    one rigid superalgebra and
    two irreducible components given by
$\mathcal{C}_1=\overline{\mathcal{O}( {\rm A}_{01})}$ and
$\mathcal{C}_2=\overline{\mathcal{O}({\rm A}_{08}^{\alpha})}.$

  \item[$({\rm 2})$]
    The variety of Malcev superalgebras has dimension $4,$
    one    rigid superalgebra and
    two irreducible components given by
$\mathcal{C}_1=\overline{\mathcal{O}( {\rm A}_{05})}$ and
$\mathcal{C}_2=\overline{\mathcal{O}({\rm A}_{08}^{\alpha})}.$

  \item[$({\rm 3})$]
    The variety of binary Lie superalgebras has dimension $4,$
    four  rigid superalgebras and
    five irreducible components given by
$\mathcal{C}_1=\overline{\mathcal{O}( {\rm A}_{05})},$ \
$\mathcal{C}_2=\overline{\mathcal{O}( {\rm A}_{10}^1)},$ \
$\mathcal{C}_3=\overline{\mathcal{O}( {\rm A}_{13})},$ \
$\mathcal{C}_4=\overline{\mathcal{O}( {\rm A}_{18})},$ and
$\mathcal{C}_5=\overline{\mathcal{O}({\rm A}_{08}^{\alpha})}.$ \

  \item[$({\rm 4})$]
    The variety of Tortkara superalgebras, $\mathfrak{aCD}$-superalgebras,
    $\mathfrak{s}_4$-superalgebras and
    $\mathfrak a$-terminal superalgebras coincide, and they have dimension $4,$
    two    rigid superalgebras and
    three irreducible components given by
$\mathcal{C}_1=\overline{\mathcal{O}( {\rm A}_{01})},$ \
$\mathcal{C}_2=\overline{\mathcal{O}( {\rm A}_{07})},$ and
$\mathcal{C}_3=\overline{\mathcal{O}({\rm A}_{08}^{\alpha})}.$

   \item[$({\rm 5})$]
    The variety of anticommutative superalgebras has dimension $7,$
    irreducible and
    without rigid  superalgebras.

\end{enumerate}

\end{theorem}

\begin{proof}
Below, we provide all necessary arguments for generations and non-degenerations.

\begin{enumerate}
    \item[Lie.]
 ${\rm A}_{08}^{\alpha}$ is anticommutative (as an algebra), but
${\rm A}_{01}$ is commutative (as an algebra), hence   ${\rm A}_{08}^{\alpha} \not\to {\rm A}_{01}.$ The rest of degenerations are

${\rm A}_{01}\xrightarrow{ \big(      e_1,\  tf_1,\  \frac 12 f_1+ f_2 \big)} {\rm A}_{02};$ \
${\rm A}_{08}^{-1} \xrightarrow{ \big(  te_1,\  tf_1-tf_2,\  f_1+f_2 \big)} {\rm A}_{06};$ \
${\rm A}_{08}^{\frac{1-t}{1+t}} \xrightarrow{ \big(  (1+t)e_1,\  tf_1-tf_2,\  f_1+f_2 \big)} {\rm A}_{17}.$

  \item[Mal.] ${\rm dim} \  \big({\rm A}_{08}^\alpha\big)^2  > {\rm dim} \  \big({\rm A}_{05}\big)^2,$ hence
 ${\rm A}_{05} \not\to {\rm A}_{08}^{\alpha}.$ The rest of degenerations are
  \begin{center}
${\rm A}_{05}\xrightarrow{ \big(      te_1,\ f_1+t^2f_2,\   tf_2  \big)} {\rm A}_{01};$ \
${\rm A}_{05}\xrightarrow{ \big(     2 te_1,\ 2 t^2f_1,\   f_1+tf_2 \big)} {\rm A}_{07}.$ \
  \end{center}

\item[BL.]
${\rm dim} \  {\mathcal{O}( {\rm A}_{10}^1)} \ = \
{\rm dim} \ {\mathcal{O}( {\rm A}_{13})} \ = \
{\rm dim} \ {\mathcal{O}( {\rm A}_{18})}\ = \
{\rm dim} \ {\mathcal{O}({\rm A}_{08}^{\alpha})} \ = \ 4.$
Hence, they give irreducible components.

\item[Tor.]
${\rm dim} \  {\mathcal{O}( {\rm A}_{01})} \ = \
{\rm dim} \ {\mathcal{O}( {\rm A}_{07})} \ = \ 3.$
Hence, they give irreducible components.

\item[Ant.]  To simplify our degenerations, we will use the following notations:

\begin{center}

${\mathfrak X}:=\sqrt{(\alpha-1)^2-4 t^2};$ \
${\mathfrak Y}:=\sqrt{(t+1) \left(\alpha t^2+1\right)};$\
${\mathfrak Z}:=-\frac{  \sqrt{\alpha+t}}{(\alpha-1) (\alpha+2 t-1)}.$

\end{center}

\begin{longtable}{lcr}
${\rm A}_{11}^{-1,-1} $&${\xrightarrow{ \big(    {\bf i} te_1,\
\frac{(1+{\bf i})\sqrt{t}}{2\sqrt{2}} f_1-\frac{(1+{\bf i})\sqrt{t}}{2\sqrt{2}}f_2,\
\frac{1-{\bf i} }{2\sqrt{2t}}f_1+\frac{1-{\bf i} }{2\sqrt{2t}}f_2  \big)}}$&$ {\rm A}_{03}$ \\

${\rm A}_{11}^{-1,\frac{1-t^2}{1+t^2}}$&$ {\xrightarrow{ \big(     te_1,\
-\frac{\sqrt{t^3+t}}{2} f_1-\frac{\sqrt{t^3+t}}{2} f_2,\
-\frac{\sqrt{t+t^{-1}}}{2 }f_1+\frac{\sqrt{t+t^{-1}}}{2}f_2\big)}} $&${\rm A}_{04}$ \\

${\rm A}_{11}^{-1,0}$&${\xrightarrow{ \big(     te_1,\
\frac{t}{\sqrt{2}}f_1-\frac{{\bf i} t}{\sqrt{2}}f_2,\
\frac{1}{\sqrt{2}}f_1+\frac{{\bf i}}{\sqrt{2}}f_2\big)}} $&${\rm A}_{05}$ \\

${\rm A}_{11}^{\frac{{1+\mathfrak X}+\alpha}{1-{\mathfrak X}+\alpha},-\frac{\alpha-1}{2 t}}$&${\xrightarrow{ \big(    \frac{1}{2} \left(1-{\mathfrak X}+\alpha\right)e_1,\
t^2f_1+\frac{2 t^3}{{\mathfrak X}+\alpha-1}f_2,\
-\frac{1}{2} t \left({\mathfrak X}-\alpha+1\right)f_1+t^2f_2\big)}} $&${\rm A}_{08}^\alpha$ \\

${\rm A}_{11}^{\frac{1+{\mathfrak X}+\alpha}{1-{\mathfrak X}+\alpha},\frac{{\mathfrak X}+\alpha-4 t^5+2 t^2-1}{2 (\alpha-1) t^4}}$&${\xrightarrow{ \big(    \frac{1}{2} \left(1-{\mathfrak X}+\alpha\right)e_1,\
t^2f_1+\frac{2 t^3}{{\mathfrak X}+\alpha-1}f_2,\
-\frac{1}{2} t \left({\mathfrak X}-\alpha+1\right)f_1+t^2f_2\big)}}$&$ {\rm A}_{09}^\alpha$ \\

${\rm A}_{11}^{\frac{\alpha-t^2+t-1}{(\alpha+t-1) (\alpha+t)}, \
-1 }$&${\xrightarrow{ \big(   ( \alpha+t)e_1,\
 t (\alpha+t-1) {\mathfrak Z} \big(f_1+f_2\big),\
{\mathfrak Z}\big( (\alpha+t-1)^2 f_1- t^2 f_2\big) \big)}} $&${\rm A}_{10}^\alpha$ \\

${\rm A}_{11}^{\frac{(\alpha-1) (\alpha+2 t-1)}{(\alpha+t-1)^2 (\alpha+t)}, \ 1 }$&${\xrightarrow{ \big(   (\alpha+t)e_1,\
t (\alpha+t-1) {\mathfrak Z} \big( f_1+ f_2 \big),\
{\mathfrak Z} \big( (\alpha+t-1)^2  f_1 +  t^2 f_2\big)\big)}} $&${\rm A}_{12}^\alpha$ \\

${\rm A}_{11}^{\frac{1}{t^{10}}-1, \ \frac{t^5 \left(2 t^{10}-3\right)}{\sqrt{2}} }$&${\xrightarrow{ \big(   t^{10}e_1,\
-\sqrt{2} t^{10}f_1+\frac{t^5}{1-2 t^{10}}f_2,\
-\frac{\sqrt{2} t^8 \left(t^{10}-1\right)}{2 t^{10}-1}f_1+\frac{t^{13}}{2 t^{10}-1}f_2\big)}} $&${\rm A}_{13}$ \\

${\rm A}_{11}^{\frac{1}{t^4}-1, \ \frac{1-2 t^6+2 t^4+t^2}{t \sqrt{4-8 t^4+8 t^2} }   }$&${\xrightarrow{ \big(   t^4e_1,\
\frac{2 t^5}{1-2 t^4}f_1+\frac{t^2 \sqrt{1-2 t^4+2 t^2}}{2 t^4-1}f_2,\
t^3 \left(\frac{1}{2 t^4-1}-1\right)f_1+\frac{t^4 \sqrt{1-2 t^4+2 t^2}}{1-2 t^4}f_2\big)}} $&${\rm A}_{14}$ \\

${\rm A}_{11}^{\frac{1-t}{t+1}, \ \frac{1-\alpha t^2}{\alpha t^2+1}   }$&${\xrightarrow{ \big(   (t+1)e_1,\
\frac{{\mathfrak Y}}{2}f_1+\frac{{\mathfrak Y}}{2}f_2,\
\frac{{\mathfrak Y}}{2 t}f_1-\frac{{\mathfrak Y}}{2 t}f_2\big)}} $&${\rm A}_{15}$ \\

${\rm A}_{11}^{\frac{1-t}{t+1}, \ -\frac{2 t^5+t+1}{2 t^2 \sqrt{t^6+t^2+t}}   }$&${\xrightarrow{ \big(   (t+1)e_1,\
\sqrt {t^6+t^2+t }f_1-t^3f_2,\
{\sqrt{t^4+1+t^{-1} }}f_1+t^2f_2
\big)}}$&$ {\rm A}_{16}$ \\

${\rm A}_{11}^{\frac{1-t}{t+1}, \ \frac{t-1}{2}    }$&${\xrightarrow{ \big(   (t+1)e_1,\
tf_1-tf_2,\
1f_1+1f_2
\big)}}$&$ {\rm A}_{18}$ \\
\end{longtable}

\end{enumerate}
\end{proof}

\begin{theorem}\label{geoant21}

For subvarieties of anticommutative superalgebras of type $(2,1),$ the following statements are true:

\begin{enumerate}
    \item[$({\rm 1})$]
    The variety of Lie and ${\mathfrak a}$-terminal superalgebras coincide and they have  dimension $3,$
    one rigid superalgebra and
    two irreducible components given by
$\mathcal{C}_1=\overline{\mathcal{O}({\bf A}_{03}^{-\frac 12} )}$ and
$\mathcal{C}_2=\overline{\mathcal{O}({\bf A}_{05}^{\alpha })}.$

  \item[$({\rm 2})$]
    The variety of Malcev superalgebras has dimension $3,$
    two      rigid superalgebras and
    three  irreducible components given by
$\mathcal{C}_1=\overline{\mathcal{O}({\bf A}_{03}^{-\frac 12} )},$ \
$\mathcal{C}_2=\overline{\mathcal{O}({\bf A}_{03}^{1} )},$ and
$\mathcal{C}_3=\overline{\mathcal{O}({\bf A}_{05}^{\alpha })}.$ \

  \item[$({\rm 3})$]
    The variety of binary Lie superalgebras has dimension $4,$
    two   rigid superalgebras and
    three irreducible components given by
$\mathcal{C}_1=\overline{\mathcal{O}( {\bf A}_{01}^{-1})},$ \
$\mathcal{C}_2=\overline{\mathcal{O}( {\bf A}_{03}^{\alpha})},$ and
$\mathcal{C}_3=\overline{\mathcal{O}( {\bf A}_{09})} .$ \

  \item[$({\rm 4})$]
    The variety of Tortkara superalgebras   has dimension $4,$
    two    rigid superalgebras and
    two irreducible components given by
$\mathcal{C}_1=\overline{\mathcal{O}( {\bf A}_{03}^{-1})}$  and
$\mathcal{C}_2=\overline{\mathcal{O}( {\bf A}_{06})} .$ \

 \item[$({\rm 5})$]
    The variety of $\mathfrak{aCD}$-superalgebras has dimension $3,$
    two    rigid superalgebras and
    three  irreducible components given by
$\mathcal{C}_1=\overline{\mathcal{O}({\bf A}_{03}^{-\frac 12} )},$ \
$\mathcal{C}_2=\overline{\mathcal{O}({\bf A}_{05}^{\alpha })},$ and
$\mathcal{C}_3=\overline{\mathcal{O}({\bf A}_{08})}.$ \

 \item[$({\rm 6})$]
    The variety of $\mathfrak{s}_4$-superalgebras has dimension $4,$
    one    rigid superalgebra and
    two  irreducible components given by
$\mathcal{C}_1=\overline{\mathcal{O}({\bf A}_{03}^{\alpha} )}$ and
$\mathcal{C}_2=\overline{\mathcal{O}({\bf A}_{08})}.$ \

   \item[$({\rm 7})$]
    The variety of anticommutative superalgebras has dimension $6,$
    irreducible and without rigid superalgebras.

\end{enumerate}

\end{theorem}

\begin{proof}
Below, we provide all necessary arguments for generations and non-degenerations.

\begin{enumerate}
    \item[Lie.] We have to find the following two degenerations
 ${\bf A}_{03}^{-\frac 12}\xrightarrow{ \big( e_1,\  te_2,\   f_1  \big)} {\bf A}_{10},$ \
${\bf A}_{05}^{\frac 1t}\xrightarrow{ \big( e_1+te_2,\   t^2e_2,\  tf_1  \big)} {\bf A}_{07}$
   and mention that
${\rm dim} \ \mathcal{O}({\bf A}_{03}^{-\frac 12}) \ = \
 {\rm dim} \ \mathcal{O}({\bf A}_{05}^{\alpha }) \ = \ 3.$

    \item[Mal.] It is only enough to mention that \
${\rm dim} \ \mathcal{O}({\bf A}_{03}^{-\frac 12}) \ = \
{\rm dim} \ \mathcal{O}({\bf A}_{03}^{1}) \ = \
 {\rm dim} \ \mathcal{O}({\bf A}_{05}^{\alpha }) \ = \ 3.$

     \item[BL.] We have to find the following two degenerations
${\bf A}_{03}^{\alpha }\xrightarrow{ \big(      e_1,\   e_2,\   tf_1  \big)} {\bf A}_{05}^{\alpha},$
${\bf A}_{03}^{\frac{1}{t} }\xrightarrow{ \big(    e_1+te_2,\   t^2e_2,\ tf_1  \big)} {\bf A}_{07}$
   and mention that \begin{center}

${\rm dim} \  \mathcal{O}( {\bf A}_{01}^{-1}) \  =\
{\rm dim} \  \mathcal{O}({\bf A}_{03}^{\alpha})  \ =\
 {\rm dim}\   \mathcal{O}({\bf A}_{09}) \  = \  4.$

   \end{center}
 \item[Tor.] We have to find the following two degenerations
${\bf A}_{06}\xrightarrow{ \big(    te_1 ,\  \alpha e_1+e_2,\    f_1  \big)} {\bf A}_{05}^{\alpha},$
${\bf A}_{03}^{-1}\xrightarrow{ \big(    e_1 ,\  t e_2,\    f_1  \big)} {\bf A}_{10}$
and mention that
$ {\bf A}_{06} \not\to  {\bf A}_{03}^{-1}$ due to $\big( (  {\bf A}_{06})_1 \big)^2=0.$

\item[$\mathfrak{aCD}.$]
 It is only enough to mention that
${\rm dim} \ \mathcal{O}( {\bf A}_{03}^{-\frac 12}) \ = \
 {\rm dim} \ \mathcal{O}({\bf A}_{05}^{\alpha }) \  = \
 {\rm dim} \ \mathcal{O}({\bf A}_{08}) \ = \   3.$

 \item[${\mathfrak s}_4$.] We have to find the two following degenerations
${\bf A}_{03}^{\frac{1}{t}}\xrightarrow{ \big(   -te_1+te_2 ,\ t^2e_2,\   t  f_1  \big)} {\bf A}_{08}$
and mention that
${\rm dim} \ \mathcal{O}({\bf A}_{03}^{\alpha}) \ = \
 {\rm dim} \ \mathcal{O}({\bf A}_{06}) \ = \   4.$

\item[Ant.]   All needed degenerations are given below.

\begin{longtable}{lcr|lcr}
${\bf A}_{02}^{\alpha } $&${\xrightarrow{ \big(
  te_1,\
e_2,\
 f_1  \big)}}$&$ {\rm A}_{01}^\alpha$ &

${\bf A}_{02}^{\alpha } $&${\xrightarrow{ \big(
  t^2e_2,\
te_1,\
 t f_1  \big)}}$&$ {\rm A}_{03}^\alpha$ \\
\hline
${\bf A}_{02}^{\frac{1}{t^2} } $&${\xrightarrow{ \big(
\frac{t-1}{t}e_1+te_2,\
-\frac{1}{t^2}e_1+1e_2,\
 t f_1  \big)}}$&$ {\rm A}_{04}$ &

${\bf A}_{02}^{0} $&${\xrightarrow{ \big(
e_1+e_2,\
 e_2,\
 t f_1  \big)}}$&$ {\rm A}_{06}$ \\
\hline

\multicolumn{6}{c}{
${\bf A}_{02}^{\frac{1}{t^2}} $\ ${\xrightarrow{ \big(
-e_1+2 t^2e_2,\
-\frac{1}{t}e_1+te_2,\
 t f_1  \big)}}$ \ $ {\rm A}_{09}$} \\

\end{longtable}

\end{enumerate}
\end{proof}

\subsection{Kokoris superalgebras}

\begin{theorem}\label{Kgeo1}
The variety of complex $3$-dimensional Kokoris    superalgebras of type $(1,2)$  has
dimension  $7$   and it has  $6$  irreducible components defined by
\begin{center}
$\mathcal{C}_1=\overline{\mathcal{O}( {\rm J}_{07})},$ \
$\mathcal{C}_2=\overline{\mathcal{O}( {\rm J}_{09})},$ \
$\mathcal{C}_3=\overline{\mathcal{O}( {\rm A}_{11}^{\alpha,\beta})},$ \
$\mathcal{C}_4=\overline{\mathcal{O}( {\rm N}_{03}^{\alpha})},$ \
$\mathcal{C}_5=\overline{\mathcal{O}( {\rm N}_{05}^{\alpha})},$ and
$\mathcal{C}_6=\overline{\mathcal{O}( {\rm N}_{09})}.$ \

 \end{center}
In particular, there are only $3$ rigid superalgebras in this variety.

\end{theorem}

\begin{proof}
All needed degenerations are given below.
\begin{longtable}{lcr|lcr}

${\rm N}_{09}^{0} $&${\xrightarrow{ \big(   te_1,\
tf_1,\
 t^2f_2  \big)}}$&$ {\rm J}_{03}$ &
 ${\rm N}_{05}^{0} $&${\xrightarrow{ \big(   e_1,\
tf_1,\
 f_2  \big)}}$&$ {\rm J}_{06}$  \\

\hline

${\rm N}_{03}^{-\frac{t}{t+2}} $&${\xrightarrow{ \big(    2 te_1,\
tf_1,\
f_1+(t+2)f_2  \big)}}$&$ {\rm N}_{04}$ &

${\rm N}_{05}^{\alpha} $&${\xrightarrow{ \big(   te_1,\
tf_1,\
 t^2f_2  \big)}}$&$ {\rm N}_{06}^\alpha$ \\

\end{longtable}
\noindent
All non-degenerations follow from the following observations.
\begin{enumerate}
    \item[$\bullet$]
$\big \{ {\rm N}_{03}^{\alpha}, \ {\rm N}_{05}^{\alpha}  \big\}  \not\to
\big\{  {\rm J}_{07},  {\rm J}_{09} \big\},$ since ${\rm N}_{03}^{\alpha}$ and ${\rm N}_{05}^{\alpha} $ are nilpotent.

\item[$\bullet$]
$ {\rm N}_{09}   \not\to
\big\{  {\rm J}_{07},  {\rm J}_{09} \big\},$ since
${\rm N}_{09}$ satisfies $\big\{ c_{11}^1=c_{12}^2, c_{13}^3=0 \big\}.$

\item[$\bullet$]
${\rm dim} \  \overline{\mathcal{O}( {\rm N}_{03}^{\alpha})}\ =\
{\rm dim} \ \overline{\mathcal{O}( {\rm N}_{05}^{\alpha})}  \ = \
{\rm dim} \ \overline{\mathcal{O}( {\rm N}_{09})} \ = \ 4.$

\end{enumerate}

\end{proof}

\begin{theorem}\label{Kgeo2}
The variety of complex $3$-dimensional Kokoris    superalgebras of type $(2,1)$  has
dimension  $6$   and it has  $5$  irreducible components defined by
\begin{center}
$\mathcal{C}_1=\overline{\mathcal{O}( {\bf J}_{01})},$ \
$\mathcal{C}_2=\overline{\mathcal{O}( {\bf J}_{02})},$ \
$\mathcal{C}_3=\overline{\mathcal{O}( {\bf A}_{02}^{\alpha})},$ \
$\mathcal{C}_4=\overline{\mathcal{O}( {\bf N}_{05})},$ \ and
$\mathcal{C}_5=\overline{\mathcal{O}( {\bf N}_{09}^{\alpha})}.$ \

 \end{center}
In particular, there are only $3$ rigid superalgebras in this variety.

\end{theorem}

\begin{proof}
All needed degenerations are ${\bf N}_{05}  \ {\xrightarrow{ \big(
e_1,\
t^2e_2,\
 tf_1  \big)}} \  {\bf N}_{04}$ and
${\bf N}_{05} \ {\xrightarrow{ \big(
e_1,\
e_2,\
 tf_1  \big)}} \ {\bf N}_{06}.$

\noindent
 All non-degenerations follow from the following observations.
 \begin{enumerate}
     \item[$\bullet$]
$  ({\bf N}_{05})_0  \not\to
\big\{  ({\bf J}_{01})_0,  ({\bf J}_{02})_0 \big\},$ hence
$ {\bf N}_{05}  \not\to
\big\{  {\bf J}_{01},  {\bf J}_{02} \big\}.$

 \item[$\bullet$]
$   {\bf N}_{05}   \not\to  {\bf N}_{09}^{\alpha},$ since
${\bf N}_{05}$ satisfies $\{ c_{13}^3=-c_{31}^3, \ c_{23}^3=-c_{32}^3 \big\}.$

\item[$\bullet$]
${\rm dim} \ \overline{\mathcal{O}( {\bf N}_{05})}  \ = \ 5;$ \
${\rm dim} \ \overline{\mathcal{O}( {\bf N}_{09}^{\alpha})} \ = \ 4.$

 \end{enumerate}

\end{proof}

\subsection{Associative superalgebras}

\begin{theorem}\label{assgeo1}
The variety of complex $3$-dimensional associative superalgebras of type $(1,2)$  has
dimension  $4$   and it has  $12$  irreducible components defined by
\begin{center}
$\mathcal{C}_1=\overline{\mathcal{O}( {\rm J}_{07})},$ \
$\mathcal{C}_2=\overline{\mathcal{O}( {\rm J}_{09})},$ \
$\mathcal{C}_3=\overline{\mathcal{O}( {\rm N}_{03}^{\alpha})},$ \
$\mathcal{C}_4=\overline{\mathcal{O}( {\rm N}_{06}^{\alpha})},$ \\
$\mathcal{C}_5=\overline{\mathcal{O}( {\rm N}_{07}^{\frac 12})},$ \
$\mathcal{C}_6=\overline{\mathcal{O}( {\rm N}_{07}^{-\frac 12})},$ \
$\mathcal{C}_7=\overline{\mathcal{O}( {\rm N}_{08}^{\frac 12})},$ \
$\mathcal{C}_8=\overline{\mathcal{O}( {\rm N}_{08}^{-\frac 12})},$ \\
$\mathcal{C}_{9}=\overline{\mathcal{O}( {\rm N}_{09})},$ \
$\mathcal{C}_{10}=\overline{\mathcal{O}( {\rm N}_{10}^{\frac 12, \frac 12})},$ \
$\mathcal{C}_{11}=\overline{\mathcal{O}( {\rm N}_{10}^{-\frac 12, \frac 12})},$ \ and
$\mathcal{C}_{12}=\overline{\mathcal{O}( {\rm N}_{10}^{-\frac 12, -\frac 12})}.$ \

 \end{center}
In particular, there are only $10$ rigid superalgebras in this variety.

\end{theorem}

\begin{proof}
All needed degenerations are
${\rm N}_{06}^{\frac{2}{t}-1}\ {\xrightarrow{ \big(  -t^2e_1,\
tf_1,\
f_1+2f_2 \big)}} \ {\rm A}_{06}$ and
${\rm N}_{09} \ {\xrightarrow{ \big(  e_1,\
tf_1,\
f_2 \big)}} \ {\rm J}_{06}.$

\noindent The dimensions of orbit closures are given below:
\begin{longtable}{lclclclclclclclclcl}

 &&&& ${\rm dim}\ \overline{\mathcal{O}( {\rm N}_{03}^{\alpha})}$& $=$&

${\rm dim}\ \overline{\mathcal{O}( {\rm N}_{09})}$& $=$&$4$\\

&&
${\rm dim}\ \overline{\mathcal{O}( {\rm N}_{06}^{\alpha})}$& $=$&
${\rm dim}\ \overline{\mathcal{O}( {\rm N}_{07}^{\frac 12})}$& $=$ &
${\rm dim}\ \overline{\mathcal{O}( {\rm N}_{07}^{-\frac 12})}$& $=$ \\&&
${\rm dim}\ \overline{\mathcal{O}( {\rm N}_{08}^{\frac 12})}$& $=$&
${\rm dim}\ \overline{\mathcal{O}( {\rm N}_{08}^{-\frac 12})}$& $=$&
${\rm dim}\ \overline{\mathcal{O}( {\rm N}_{10}^{-\frac 12, \frac 12})}$&$=$&$3$\\


${\rm dim}\ \mathcal{O}( {\rm J}_{07})$& $=$&
${\rm dim}\ \mathcal{O}( {\rm J}_{09})$& $=$&
${\rm dim}\ \mathcal{O}( {\rm N}_{10}^{\frac 12, \frac 12})$& $=$&
${\rm dim}\ \mathcal{O}( {\rm N}_{10}^{-\frac 12, -\frac 12})$&$=$&$1$\\

\end{longtable}
\noindent All reasons for non-degenerations are given below.
\begin{enumerate}
    \item[$\bullet$]    ${\rm N}_{03}^{\alpha} \not\to {\rm N}_{06}^{\alpha},$
    since $ ({\rm N}_{03}^{\alpha})_0({\rm N}_{03}^{\alpha})_1=0.$
    \item[$\bullet$]
    ${\rm N}_{03}^{\alpha}$ and  ${\rm N}_{06}^{\alpha}$ are nilpotent, hence
     $\big\{ {\rm N}_{03}^{\alpha}, \  {\rm N}_{06}^{\alpha} \big\} \not\to \big\{ {\rm J}_{07}, \  {\rm J}_{09}, \ {\rm N}_{07}^{\alpha}, \ {\rm N}_{08}^{\alpha}, \  \ {\rm N}_{10}^{\alpha, \beta} \big\}.$

    \item[$\bullet$]
    ${\rm N}_{09}$ satisfies $\left\{
    \begin{array}{l}
    c_{11}^1=c_{12}^2=c_{21}^2,\\
    c_{13}^3=c_{31}^3=0
    \end{array}\right\},$ hence
       ${\rm N}_{09}  \not\to \big\{ {\rm J}_{07},   {\rm J}_{09}, {\rm N}_{03}^{\alpha}, {\rm N}_{07}^{\alpha},  {\rm N}_{08}^{\alpha},    {\rm N}_{10}^{\alpha, \beta} \big\}.$

    \item[$\bullet$]
    ${\rm N}_{07}^\alpha$ satisfies
    $\left\{
    \begin{array}{l}
    c_{12}^2= (\frac 12 +\alpha) c_{11}^1,\   c_{13}^3=c_{31}^3=0,\\
    c_{21}^2=(\frac 12 -\alpha) c_{11}^1
    \end{array} \right\},$ hence
     ${\rm N}_{07}^\alpha  \not\to \big\{ {\rm J}_{07},  \ {\rm J}_{09}, \  {\rm N}_{10}^{\alpha, \beta} \big\}.$

    \item[$\bullet$]
    ${\rm N}_{08}^\alpha$ satisfies
    $\left\{
    \begin{array}{l}
    c_{12}^2=  c_{11}^1, \   c_{13}^3= (\frac 12 +\alpha) c_{11}^1, \\
    c_{21}^2=  c_{11}^1, \ c_{31}^3=(\frac 12 -\alpha) c_{11}^1
    \end{array}\right\},$ hence
${\rm N}_{08}^\alpha  \not\to \big\{ {\rm J}_{07}, \  {\rm J}_{09},  \ {\rm N}_{10}^{\alpha, \beta} \big\}.$

\item[$\bullet$]
    ${\rm N}_{10}^{-\frac 12, \frac 12}$ satisfies $\big\{ c_{12}^2=0, c_{21}^2=  c_{11}^1, c_{13}^3=  c_{11}^1, c_{31}^3=0 \big\},$ hence
    ${\rm N}_{10}^{-\frac 12, \frac 12}  \not\to \big\{ {\rm J}_{07},   {\rm J}_{09},   {\rm N}_{10}^{\alpha, \alpha} \big\}.$

\end{enumerate}

 \end{proof}

\begin{theorem}\label{Kgeo2}
The variety of complex $3$-dimensional associative  superalgebras of type $(2,1)$  has
dimension  $5$   and it has  $13$  irreducible components defined by
\begin{center}
$\mathcal{C}_1=\overline{\mathcal{O}( {\bf J}_{01})},$ \
$\mathcal{C}_2=\overline{\mathcal{O}( {\bf N}_{01})},$ \
$\mathcal{C}_3=\overline{\mathcal{O}( {\bf N}_{02}^{\frac 12})},$ \
$\mathcal{C}_4=\overline{\mathcal{O}( {\bf N}_{02}^{-\frac  12})},$ \
$\mathcal{C}_5=\overline{\mathcal{O}( {\bf N}_{03}^{\frac 12})},$ \\
$\mathcal{C}_6=\overline{\mathcal{O}( {\bf N}_{13}^{\frac 12})},$ \
$\mathcal{C}_{7}=\overline{\mathcal{O}( {\bf N}_{13}^{-\frac  12})},$ \
$\mathcal{C}_{8}=\overline{\mathcal{O}( {\bf N}_{17}^{\frac 12, \frac 12})},$ \
$\mathcal{C}_{9}=\overline{\mathcal{O}( {\bf N}_{17}^{-\frac  12, \frac 12})},$ \\
$\mathcal{C}_{10}=\overline{\mathcal{O}( {\bf N}_{17}^{\frac 12, -\frac 12})},$ \
$\mathcal{C}_{11}=\overline{\mathcal{O}( {\bf N}_{17}^{-\frac  12, -\frac 12})},$ \
$\mathcal{C}_{12}=\overline{\mathcal{O}( {\bf N}_{18}^{\frac 12})},$ \ and
$\mathcal{C}_{13}=\overline{\mathcal{O}( {\bf N}_{18}^{-\frac  12})}.$ \

 \end{center}
In particular, there are only $13$ rigid superalgebras in this variety.

\end{theorem}

\begin{proof}
All needed degenerations are given below.
\begin{longtable}{lcr|lcr}

${\bf N}_{01}$&${\xrightarrow{ \big(
t^2e_1 ,\
 te_2,\
 tf_1  \big)}}$&$ {\bf A}_{10}$ &

${\bf N}_{01}$&${\xrightarrow{ \big(
e_1 ,\
 e_2,\
 tf_1  \big)}}$&$ {\bf J}_{02}$ \\
\hline

${\bf N}_{01}$&${\xrightarrow{ \big(
e_2 ,\
 t^2e_1,\
 tf_1  \big)}}$&$ {\bf N}_{04}$ &

${\bf N}_{02}^\alpha $&${\xrightarrow{ \big(
e_1,\
te_2,\
 f_1  \big)}}$&$ {\bf N}_{07}^{\alpha}$ \\
\hline

${\bf N}_{01}$&${\xrightarrow{ \big(
e_1 ,\
 te_2,\
 f_1  \big)}}$&$ {\bf N}_{08}$ &

${\bf N}_{02}^\alpha$&${\xrightarrow{ \big(
e_1+e_2 ,\
 te_2,\
 f_1  \big)}}$&$ {\bf N}_{09}^\alpha$ \\ \hline

${\bf N}_{01}$&${\xrightarrow{ \big(
e_1+e_2,\
t (t+1)e_1+te_2,\
tf_1  \big)}}$&$ {\bf N}_{10}$   &

${\bf N}_{01} $&$ {\xrightarrow{ \big(
te_1+(t-t^3)e_2,\
(t^4+t^2)e_1+ \left(t^2-t^4\right)e_2,\
t^2f_1  \big)}} $&$  {\bf N}_{20}$    \\

\end{longtable}
\noindent
The dimensions of orbit closures are given below:
\begin{longtable}{lclclclclclclclclcl}

&&&&${\rm dim}\ \overline{\mathcal{O}( {\bf N}_{01})}$& $=$&$5$\\

${\rm dim}\ \overline{\mathcal{O}( {\bf J}_{01})}$& $=$&
${\rm dim}\ \overline{\mathcal{O}( {\bf N}_{02}^{\pm \frac12 })}$& $=$&
${\rm dim}\ \overline{\mathcal{O}( {\bf N}_{03}^{ \frac12 })}$& $=$& $4$\\


${\rm dim}\ \overline{\mathcal{O}( {\bf N}_{13}^{\pm \frac12 })}$& $=$&
${\rm dim}\ \overline{\mathcal{O}( {\bf N}_{17}^{\pm \frac12, \pm \frac 12 })}$& $=$ &
${\rm dim}\  \overline{\mathcal{O}( {\bf N}_{18}^{\pm \frac12 })}$& $=$&$2$\\

\end{longtable}
\noindent All reasons for non-degenerations are given below.
\begin{enumerate}
    \item[$\bullet$] ${\bf N}_{01}$ is commutative, hence ${\bf N}_{01} \not\to \big\{
 {\bf N}_{02}^{\pm \frac12 }, {\bf N}_{03}^{  \frac12 },
 {\bf N}_{13}^{\pm \frac12 },  {\bf N}_{17}^{\pm \frac12, \pm \frac 12 }, {\bf N}_{18}^{\pm \frac12 } \big\}.$

\item[$\bullet$]  $({\bf N}_{02}^{\pm \frac12 })_0$ and $({\bf N}_{03}^{  \frac12 })_0 $ are commutative,
 hence $ \big\{{\bf N}_{02}^{\pm \frac12 }, {\bf N}_{03}^{\pm \frac12 }  \big\} \not\to \big\{
 {\bf N}_{13}^{\pm \frac12 },  {\bf N}_{17}^{\pm \frac12, \pm \frac 12 }, {\bf N}_{18}^{\pm \frac12 } \big\}.$

\item[$\bullet$] ${\bf N}_{01}$ satisfies $\big \{ c_{11}^1=c_{13}^3, \  c_{22}^2=c_{23}^2 \big\}$,
hence ${\bf N}_{01} \not\to {\bf J}_{01}$





\end{enumerate}

\end{proof}

\subsection{Standard superalgebras}

\begin{theorem}\label{stgeo1}
The variety of complex $3$-dimensional standard superalgebras of type $(1,2)$  has
dimension  $4$   and it has  $16$  irreducible components defined by
\begin{center}
$\mathcal{C}_1=\overline{\mathcal{O}( {\rm J}_{01})},$ \
$\mathcal{C}_2=\overline{\mathcal{O}( {\rm J}_{04})},$ \
$\mathcal{C}_3=\overline{\mathcal{O}( {\rm J}_{05})},$ \
$\mathcal{C}_4=\overline{\mathcal{O}( {\rm J}_{07})},$ \
$\mathcal{C}_5=\overline{\mathcal{O}( {\rm J}_{10})},$ \
$\mathcal{C}_6=\overline{\mathcal{O}( {\rm J}_{11})},$ \
$\mathcal{C}_7=\overline{\mathcal{O}( {\rm N}_{03}^{\alpha})},$ \
$\mathcal{C}_8=\overline{\mathcal{O}( {\rm N}_{06}^{\alpha})},$ \
$\mathcal{C}_{9}=\overline{\mathcal{O}( {\rm N}_{07}^{\frac 12})},$ \
$\mathcal{C}_{10}=\overline{\mathcal{O}( {\rm N}_{07}^{-\frac 12})},$ \
$\mathcal{C}_{11}=\overline{\mathcal{O}( {\rm N}_{08}^{\frac 12})},$ \
$\mathcal{C}_{12}=\overline{\mathcal{O}( {\rm N}_{08}^{-\frac 12})},$ \
$\mathcal{C}_{13}=\overline{\mathcal{O}( {\rm N}_{09})},$ \
$\mathcal{C}_{14}=\overline{\mathcal{O}( {\rm N}_{10}^{\frac 12, \frac 12})},$ \
$\mathcal{C}_{15}=\overline{\mathcal{O}( {\rm N}_{10}^{-\frac 12, \frac 12})},$ \ and
$\mathcal{C}_{16}=\overline{\mathcal{O}( {\rm N}_{10}^{-\frac 12, -\frac 12})}.$ \

 \end{center}
In particular, there are only $14$ rigid superalgebras in this variety.

\end{theorem}

\begin{proof}
Thanks to Theorem \ref{S1}, we have the union of the associative and Jordan parts,
hence, we have to choose irreducible components from Theorems \ref{asscomjord} and \ref{assgeo1}.

\end{proof}

\begin{theorem}\label{stgeo2}
The variety of complex $3$-dimensional    standard superalgebras of type $(2,1)$  has
dimension  $5$   and it has  $22$  irreducible components defined by
\begin{center}
$\mathcal{C}_1=\overline{\mathcal{O}( {\bf J}_{01})},$ \
$\mathcal{C}_2=\overline{\mathcal{O}( {\bf J}_{03})},$ \
$\mathcal{C}_3=\overline{\mathcal{O}( {\bf J}_{04})},$ \
$\mathcal{C}_4=\overline{\mathcal{O}( {\bf J}_{11})},$ \
$\mathcal{C}_5=\overline{\mathcal{O}( {\bf J}_{12})},$ \
$\mathcal{C}_6=\overline{\mathcal{O}( {\bf J}_{13})},$ \\
$\mathcal{C}_7=\overline{\mathcal{O}( {\bf N}_{01})},$ \
$\mathcal{C}_8=\overline{\mathcal{O}( {\bf N}_{02}^{\frac 12})},$ \
$\mathcal{C}_9=\overline{\mathcal{O}( {\bf N}_{02}^{-\frac  12})},$ \
$\mathcal{C}_{10}=\overline{\mathcal{O}( {\bf N}_{03}^{\frac 12})},$ \
$\mathcal{C}_{11}=\overline{\mathcal{O}( {\bf N}_{13}^{\frac 12})},$ \
$\mathcal{C}_{12}=\overline{\mathcal{O}( {\bf N}_{13}^{-\frac  12})},$ \\
$\mathcal{C}_{13}=\overline{\mathcal{O}( {\bf N}_{17}^{\frac 12, 0})},$ \
$\mathcal{C}_{14}=\overline{\mathcal{O}( {\bf N}_{17}^{0, \frac 12})},$ \
$\mathcal{C}_{15}=\overline{\mathcal{O}( {\bf N}_{17}^{0, -\frac 12})},$ \
$\mathcal{C}_{16}=\overline{\mathcal{O}( {\bf N}_{17}^{-\frac 12, 0})},$ \
$\mathcal{C}_{17}=\overline{\mathcal{O}( {\bf N}_{17}^{\frac 12, \frac 12})},$ \\
$\mathcal{C}_{18}=\overline{\mathcal{O}( {\bf N}_{17}^{-\frac  12, \frac 12})},$ \
$\mathcal{C}_{19}=\overline{\mathcal{O}( {\bf N}_{17}^{\frac 12, -\frac 12})},$ \
$\mathcal{C}_{20}=\overline{\mathcal{O}( {\bf N}_{17}^{-\frac  12, -\frac 12})},$ \
$\mathcal{C}_{21}=\overline{\mathcal{O}( {\bf N}_{18}^{\frac 12})},$ \ and
$\mathcal{C}_{22}=\overline{\mathcal{O}( {\bf N}_{18}^{-\frac  12})}.$ \

 \end{center}
In particular, there are only $22$ rigid superalgebras in this variety.

\end{theorem}
\begin{proof}
Considering non-associative non-Jordan standard superalgebras ${\bf N}_{17}^{\pm \frac 12, 0}$ and ${\bf N}_{17}^{0, \pm\frac 12},$
we found that they give irreducible components of dimension $2$.
Hence, all irreducible components of our variety will be defined by irreducible components of associative superalgebras,
non-associative Jordan superalgebras, and cited standard superalgebras.
\end{proof}

\subsection{Noncommutative Jordan superalgebras}

\begin{theorem}\label{NCJgeo1}
The variety of complex $3$-dimensional noncommutative Jordan superalgebras of type $(1,2)$  has
dimension  $7$   and it has  $9$  irreducible components defined by
\begin{center}
$\mathcal{C}_1=\overline{\mathcal{O}( {\rm J}_{07})},$ \
$\mathcal{C}_2=\overline{\mathcal{O}( {\rm J}_{11})},$ \
$\mathcal{C}_3=\overline{\mathcal{O}( {\rm A}_{11}^{\alpha,\beta})},$ \
$\mathcal{C}_4=\overline{\mathcal{O}( {\rm N}_{01}^{\alpha })},$ \
$\mathcal{C}_5=\overline{\mathcal{O}( {\rm N}_{07}^{\alpha})},$ \\
$\mathcal{C}_6=\overline{\mathcal{O}( {\rm N}_{08}^{\alpha})},$ \
$\mathcal{C}_7=\overline{\mathcal{O}( {\rm N}_{09})},$ \
$\mathcal{C}_8=\overline{\mathcal{O}( {\rm N}_{10}^{\alpha, \beta})},$ \ and
$\mathcal{C}_9=\overline{\mathcal{O}( {\rm N}_{12}^{\alpha})}.$ \

 \end{center}
In particular, there are only $3$ rigid superalgebras in this variety.

\end{theorem}

\begin{proof}
Below, we   mention all necessary degenerations.

\begin{longtable}{lcr|lcr}
${\rm N}_{09} $&$ {\xrightarrow{ \big(
e_1,\
t f_1,\
 f_2
\big)}} $&$ {\rm J}_{06}$ &

${\rm N}_{12}^0 $&$ {\xrightarrow{ \big(
e_1,\
\frac{1}{t}f_1,\
t f_2
\big)}} $&$ {\rm J}_{10}$ \\
\hline

${\rm N}_{01}^{-\frac{t}{t+2}} $&$ {\xrightarrow{ \big(
t \sqrt{2t+4}e_1,\
\frac{\sqrt{t+2}}{\sqrt{2}}f_1-(t+2)f_2,\
\frac{t}{\sqrt{2t+4}}f_1+tf_2
\big)}} $&$  {\rm N}_{02}$   &

 ${\rm N}_{01}^{\alpha} $&$ {\xrightarrow{ \big(
t  e_1,\
t f_1 ,\
 f_2
\big)}} $&$ {\rm N}_{03}^\alpha$  \\
\hline

${\rm N}_{03}^{-\frac{t}{t+2}} $&${\xrightarrow{ \big(    2 te_1,\
tf_1,\
f_1+(t+2)f_2  \big)}}$&$ {\rm N}_{04}$ &

 ${\rm N}_{01}^{\alpha}$&$  {\xrightarrow{ \big(
2t\alpha  e_1,\
t f_1 +f_2 ,\
 2\alpha t^2 f_2
\big)}} $&$  {\rm N}_{05}^\alpha$   \\
\hline

 ${\rm N}_{01}^{\alpha} $&$ {\xrightarrow{ \big(
  e_1,\
t f_1 ,\
   f_2
\big)}} $&$ {\rm N}_{06}^\alpha$  &

 ${\rm N}_{10}^{\alpha-t, \alpha+t}$ &  ${\xrightarrow{ \big(
  e_1,\
tf_1+tf_2 ,\
- f_1+ f_2
\big)}}$ & ${\rm N}_{11}^\alpha$ \\
\hline

\end{longtable}

\noindent The dimensions of orbit closures are given below:
\begin{longtable}{lclclclclclclclclcl}

&&&&${\rm dim}\ \overline{\mathcal{O}( {\rm A}_{11}^{\alpha, \beta})}$& $=$& $7$\\

${\rm dim}\ \overline{\mathcal{O}( {\rm N}_{01}^{\alpha})}$& $=$&
${\rm dim}\ \overline{\mathcal{O}( {\rm N}_{10}^{\alpha,\beta})}$& $=$&
${\rm dim}\ \overline{\mathcal{O}( {\rm N}_{12}^\alpha)}$& $=$&$5$\\

${\rm dim}\ \overline{\mathcal{O}( {\rm N}_{07}^{\alpha})}$& $=$&
${\rm dim}\ \overline{\mathcal{O}( {\rm N}_{08}^{\alpha})}$& $=$&
${\rm dim}\ \overline{\mathcal{O}( {\rm N}_{09})}$& $=$&$4$\\

&&&&${\rm dim}\ \overline{\mathcal{O}( {\rm J}_{11})}$& $=$&$2$\\

&&&&${\rm dim}\ \overline{\mathcal{O}( {\rm J}_{07})}$& $=$&$1$\\

\end{longtable}
\noindent All reasons for non-degenerations are following:
\begin{enumerate}
    \item[$\bullet$]
    ${\rm N}_{01}^{\alpha}$ is nilpotent, hence
    ${\rm N}_{01}^{\alpha} \not\to
    \big\{   {\rm N}_{07}^{\alpha},   {\rm N}_{08}^{\alpha},   {\rm N}_{09},   {\rm J}_{11},   {\rm J}_{07} \big\}.$

 \item[$\bullet$]
    ${\rm N}_{10}^{\alpha,\beta}$ satisfies $\big\{
    c_{12}^2 + c_{21}^2=   c_{11}^1, \   c_{13}^3 +c_{31}^3=  c_{11}^1 \big\},$  hence
 ${\rm N}_{10}^{\alpha,\beta}   \not\to
    \big\{   {\rm N}_{07}^{\alpha},   {\rm N}_{08}^{\alpha},   {\rm N}_{09},   {\rm J}_{11},   {\rm J}_{07} \big\}.$

 \item[$\bullet$]
    ${\rm N}_{12}^{\alpha}$ satisfies $\left\{
\begin{array}{l}
c_{11}^1 = c_{12}^2 + c_{13}^3,\
c_{21}^3 + c_{12}^3 = 0, \\
    c_{31}^2 + c_{13}^2 = 0, \
    c_{11}^1 = c_{31}^3 + c_{13}^3
    \end{array}\right\},$    hence
 ${\rm N}_{12}^{\alpha}   \not\to
    \big\{   {\rm N}_{07}^{\alpha},   {\rm N}_{08}^{\alpha},   {\rm N}_{09},   {\rm J}_{11},   {\rm J}_{07} \big\}.$

 \item[$\bullet$]
    ${\rm N}_{07}^{\alpha}$ satisfies $\big\{
c_{11}^1 = c_{12}^2 + c_{21}^2,\
c_{13}^3 = c_{31}^3 = 0   \big\},$   hence
 ${\rm N}_{07}^{\alpha}   \not\to
    \big\{      {\rm J}_{11},   {\rm J}_{07} \big\}.$

 \item[$\bullet$]
    ${\rm N}_{08}^{\alpha}$ satisfies $\big\{
2c_{11}^1 = c_{12}^2 + c_{21}^2,\
c_{11}^1 =c_{13}^3 + c_{31}^3    \big\},$   hence
 ${\rm N}_{05}^{\alpha}   \not\to
    \big\{      {\rm J}_{11},   {\rm J}_{07} \big\}.$

     \item[$\bullet$]
    ${\rm N}_{09}$ satisfies $\big\{
2c_{11}^1 = c_{12}^2 + c_{21}^2,\
c_{13}^3 = c_{31}^3 = 0  \big\},$   hence
 ${\rm N}_{09}^   \not\to
    \big\{      {\rm J}_{11},   {\rm J}_{07} \big\}.$

\end{enumerate}

 \end{proof}

\begin{theorem}\label{NCJgeo2}
The variety of complex $3$-dimensional noncommutative Jordan superalgebras of type $(2,1)$  has
dimension  $6$   and it has  $10$  irreducible components defined by
\begin{center}
$\mathcal{C}_1=\overline{\mathcal{O}( {\bf J}_{01})},$ \
$\mathcal{C}_2=\overline{\mathcal{O}( {\bf A}_{02}^{\alpha})},$ \
$\mathcal{C}_3=\overline{\mathcal{O}( {\bf N}_{01} )},$ \
$\mathcal{C}_4=\overline{\mathcal{O}( {\bf N}_{02}^{\alpha })},$ \
$\mathcal{C}_5=\overline{\mathcal{O}( {\bf N}_{03}^{\alpha })},$ \\
$\mathcal{C}_6=\overline{\mathcal{O}( {\bf N}_{05} )},$ \
$\mathcal{C}_7=\overline{\mathcal{O}( {\bf N}_{11} )},$ \
$\mathcal{C}_8=\overline{\mathcal{O}( {\bf N}_{13}^{\alpha })},$ \
$\mathcal{C}_9=\overline{\mathcal{O}( {\bf N}_{14}^{\alpha })},$ \ and
$\mathcal{C}_{10}=\overline{\mathcal{O}( {\bf N}_{18}^{\alpha })}.$ \

 \end{center}
In particular, there are only $4$ rigid superalgebras in this variety.

\end{theorem}

\begin{proof}
Below, we will mention all necessary degenerations.

\begin{longtable}{lcr|lcr}

${\bf N}_{01}$&${\xrightarrow{ \big(
e_2 ,\
 t^2e_1,\
 tf_1  \big)}}$&$ {\bf N}_{04}$ &

${\bf N}_{05} $&$ {\xrightarrow{ \big(
e_1,\
e_2,\
 tf_1  \big)}} $&$ {\bf N}_{06}$ \\
\hline

 ${\bf N}_{02}^\alpha $&${\xrightarrow{ \big(
e_1,\
te_2,\
 f_1  \big)}}$&$ {\bf N}_{07}^{\alpha}$ &

${\bf N}_{01}$&${\xrightarrow{ \big(
e_1 ,\
 te_2,\
 f_1  \big)}}$&$ {\bf N}_{08}$ \\ \hline

${\bf N}_{02}^\alpha$&${\xrightarrow{ \big(
e_1+e_2 ,\
 te_2,\
 f_1  \big)}}$&$ {\bf N}_{09}^\alpha$  &

${\bf N}_{01}$&${\xrightarrow{ \big(
e_1+e_2,\
t (t+1)e_1+te_2,\
tf_1  \big)}}$&$ {\bf N}_{10}$   \\\hline

 ${\bf N}_{11}  $&$ {\xrightarrow{ \big(
e_1,\
e_2,\
 t f_1
\big)}} $&$ {\bf N}_{12}$ &

${\bf N}_{14}^\alpha $&$ {\xrightarrow{ \big(
e_1,\
e_2,\
 t f_1
\big)}} $&$ {\bf N}_{15}^\alpha$ \\\hline

${\bf N}_{14}^\alpha $&$ {\xrightarrow{ \big(
e_1+ \beta e_2,\
te_1+t (\beta+t)e_2,\
tf_1
\big)}} $&$ {\bf N}_{16}^{\alpha,\beta}$ &

 ${\bf N}_{16}^{\alpha,\beta} $&$ {\xrightarrow{ \big(
e_1,\
e_2,\
 t f_1
\big)}} $&$ {\bf N}_{17}^{\alpha,\beta}$ \\\hline

\multicolumn{3}{c|}{
 ${\bf N}_{05}  \  {\xrightarrow{ \big(
te_1+(t-t^3)e_2,\
(t^4+t^2)e_1+ \left(t^2-t^4\right)e_2,\
t^2f_1
\big)}} \ {\bf N}_{19}$} &

\multicolumn{3}{c}{
${\bf N}_{01} \  {\xrightarrow{ \big(
te_1+(t-t^3)e_2,\
(t^4+t^2)e_1+ \left(t^2-t^4\right)e_2,\
t^2f_1  \big)}} \  {\bf N}_{20}$}    \\
\hline

\multicolumn{6}{c}{
 ${\bf N}_{19}  \  {\xrightarrow{ \big(
e_1,\
e_2,\
 t f_1
\big)}} \  {\bf N}_{21}$}  \\

\end{longtable}

The dimensions of orbit closures are given below:
\begin{longtable}{lclclclclclclclclcl}

&&&&&&${\rm dim}\ \overline{\mathcal{O}( {\bf A}_{02}^{\alpha})}$& $=$&
${\rm dim}\ \overline{\mathcal{O}( {\bf N}_{14}^{\alpha})}$& $=$&$6$\\

${\rm dim}\ \overline{\mathcal{O}( {\bf N}_{01})}$& $=$&
${\rm dim}\ \overline{\mathcal{O}( {\bf N}_{02}^{\alpha})}$& $=$&
${\rm dim}\ \overline{\mathcal{O}( {\bf N}_{03}^{\alpha})}$& $=$&
${\rm dim}\ \overline{\mathcal{O}( {\bf N}_{05})}$& $=$ &
 ${\rm dim}\ \overline{\mathcal{O}( {\bf N}_{11})}$& $=$&
$5$\\

&&&&
&&&&
${\rm dim}\ \overline{\mathcal{O}( {\bf J}_{01})}$& $=$&$4$\\

&&&&&&
${\rm dim}\ \overline{\mathcal{O}( {\bf N}_{13}^{\alpha})}$& $=$&
${\rm dim}\ \overline{\mathcal{O}( {\bf N}_{18}^{\alpha})}$& $=$&$3$\\

\end{longtable}

 \noindent All reasons for non-degenerations are as follows:
\begin{enumerate}
    \item[$\bullet$]

$({\bf N}_{14}^\alpha)_0$ is non-semisimple and ${\bf N}_{14}^\alpha$ satisfies

$\big\{
c_{11}^1 = c_{31}^3 + c_{13}^3, \
    c_{12}^2 c_{22}^2 = c_{11}^1 c_{21}^1, \
    c_{11}^1 = c_{21}^2 + c_{12}^2, \
      c_{22}^2 = c_{32}^3 + c_{23}^3, \
           c_{11}^1 c_{33}^1 = -c_{33}^2 c_{22}^2 \big\},$

    $ \quad\quad\quad\quad\quad\quad \quad\quad \quad\quad\quad\quad$       hence
  ${\bf N}_{14}^\alpha \not\to
  \big\{ {\bf N}_{01},  \ {\bf N}_{02}^{\alpha}, \ {\bf N}_{03}^{\alpha},\  {\bf N}_{05},   \  {\bf N}_{11}, \ {\bf N}_{13}^{\alpha}, \ {\bf N}_{18}^{\alpha},  \ {\bf J}_{01} \big\}.$

 \item[$\bullet$]

${\bf N}_{01}$ is commutative (as algebra) and   satisfies $\big\{
c_{11}^1 =  c_{13}^3, \ c_{12}^2c_{12}^2=c_{11}^1 c_{12}^2 + c_{11}^2 c_{22}^2       \big\},$

    $ \quad\quad\quad\quad\quad\quad \quad\quad\quad\quad\quad\quad\quad\quad\quad\quad\quad\quad\quad\quad\quad\quad\quad  $       hence
  ${\bf N}_{01}  \not\to
  \big\{    {\bf N}_{13}^{\alpha}, \ {\bf N}_{18}^{\alpha},  \ {\bf J}_{01} \big\}.$

 \item[$\bullet$]
${\bf N}_{02}^\alpha$  satisfies $\big\{
c_{11}^1 =  c_{13}^3+c_{31}^3, \ c_{12}^1=c_{23}^3=0        \big\},$
              hence
  ${\bf N}_{02}^\alpha  \not\to
  \big\{    {\bf N}_{13}^{\alpha}, \ {\bf N}_{18}^{\alpha},  \ {\bf J}_{01} \big\}.$

 \item[$\bullet$]
  ${\rm dim} \ \big( ({\bf N}_{05})_0 \big)^2 =1,$
              hence
  ${\bf N}_{05}  \not\to
  \big\{    {\bf N}_{13}^{\alpha}, \ {\bf N}_{18}^{\alpha},  \ {\bf J}_{01} \big\}.$

 \item[$\bullet$]
$({\bf N}_{11})_0$ is commutative and non-semisimple,           hence
  ${\bf N}_{11}  \not\to
  \big\{    {\bf N}_{13}^{\alpha}, \ {\bf N}_{18}^{\alpha},  \ {\bf J}_{01} \big\}.$

    \end{enumerate}

 \end{proof}


\end{document}